\documentclass[authoryear]{elsarticle}
\RequirePackage[OT1]{fontenc}
\usepackage[utf8]{inputenc}
\RequirePackage{amsthm,amsmath,natbib}
\usepackage{hyperref}
\usepackage[ngerman,english]{babel}
\usepackage{a4wide}
\usepackage{calc}
\usepackage{times}
\usepackage{listings}
\usepackage{changes}
\usepackage{fancyhdr}
\usepackage{amssymb,amsmath,amstext,amsthm, ams fonts, amsfonts,mathdots}
\usepackage{color,graphics,graphicx}
\usepackage{framed}
\usepackage{dsfont}
\usepackage{bbm}
\usepackage{nicefrac}
\newcommand{\KLEINO}{{\scriptstyle{\mathcal{O}}}}
\renewcommand{\tilde}{\widetilde}

\newcommand{\iu}{\mathrm i}

\renewcommand{\P}{\mathbb{P}}
\renewcommand{\Re}{\operatorname{Re}}
\newcommand{\F}{\mathcal{F}}
\newcommand{\N}{\mathds{N}}
\newcommand{\R}{\mathds{R}}
\DeclareMathAccent{\verywidehat}{\mathord}{largesymbols}{'144}
\newcommand{\var}{\mathbb{V}\hspace*{-0.05cm}\textnormal{a\hspace*{0.02cm}r}}
\newcommand{\Cov}{\mathbb{C}\textnormal{O\hspace*{0.02cm}V}}

\newcommand{\cov}{\mathbb{C}\textnormal{o\hspace*{0.02cm}v}}

\newtheorem{prop}{Proposition}[section]
\newtheorem{ass}[prop]{Assumption}

\newtheorem{cor}[prop]{Corollary}

\newtheorem{lem}[prop]{Lemma}
\newtheorem{thm}[prop]{Theorem}
\newtheorem{rem}[prop]{Remark}

\allowdisplaybreaks[4]

\global\long\def\phi{\varphi}
\global\long\def\epsilon{\varepsilon}
\global\long\def\theta{\vartheta}
\global\long\def\E{\mathbb{E}}

\global\long\def\N{\mathbb{N}}

\global\long\def\R{\mathbb{R}}
\global\long\def\F{\mathcal{F}}
\global\long\def\le{\leqslant}
\global\long\def\ge{\geqslant}

\global\long\def\1{\mathbbm1}
\global\long\def\d{\mathrm{d}}
 \global\long\def\subset{\subseteq}

\global\long\def\argmin{\arg\,\min}

\begin{document}
\begin{frontmatter}
\title{Volatility estimation for stochastic PDEs using high-frequency observations}
\author[1]{Markus Bibinger} 
\author[2]{Mathias Trabs\footnote{Corresponding author.}}

\address[1]{Fachbereich 12 Mathematik und Informatik, Philipps-Universität Marburg, bibinger@uni-marburg.de}
\address[2]{Fachbereich Mathematik, Universität Hamburg, mathias.trabs@uni-hamburg.de}
\normalsize

\begin{abstract}
We study the parameter estimation for parabolic, linear, second-order, stochastic partial differential equations (SPDEs) observing a mild solution on a discrete grid in time and space. A high-frequency regime is considered where the mesh of the grid in the time variable goes to zero. Focusing on volatility estimation, we provide an explicit and easy to implement method of moments estimator based on squared increments. The estimator is consistent and admits a central limit theorem. This is established moreover for the joint estimation of the integrated volatility and parameters in the differential operator in a semi-parametric framework. Starting from a representation of the solution of the SPDE with Dirichlet boundary conditions as an infinite factor model and exploiting mixing-type properties of time series, the theory considerably differs from the statistics for semi-martingales literature. The performance of the method is illustrated in a simulation study.

\begin{keyword}
high-frequency data\sep
stochastic partial differential equation\sep
random field\sep
realized volatility\sep
mixing-type limit theorem\\[.25cm]
\MSC[2010] {62M10} \sep {60H15}
\end{keyword}
\end{abstract}
\end{frontmatter}
\thispagestyle{plain}

\section{Introduction}
\enlargethispage*{.8cm}
\subsection{Overview}

Motivated by random phenomena in natural science as well as by mathematical finance, stochastic partial differential equations (SPDEs) have been intensively studied during the last fifty years with a main focus on theoretical analytic and probabilistic aspects. Thanks to the exploding number of available data and the fast progress in information technology, SPDE models become nowadays increasingly popular for practitioners, for instance, to model neuronal systems or interest rate fluctuations. Consequently, statistical methods are required to calibrate this class of complex models. While in probability theory there are recently enormous efforts to advance research for SPDEs, for instance \citet{hairer2013} was able to solve the KPZ equation, there is scant groundwork on statistical inference for SPDEs and there remain many open questions. 

Considering {\emph{discrete high-frequency data}} in time, our aim is to extend the well understood statistical theory for semi-martingales, see, for instance, \cite{JP} and \cite{inference}, to parabolic SPDEs which can be understood as infinite dimensional stochastic differential equations. Generated by an infinite factor model, high-frequency dynamics of the SPDE model differ from the semi-martingale case in several ways. The SPDE model induces a distinctive, much rougher behavior of the marginal processes over time for a fixed spatial point compared to semi-martingales or diffusion processes. In particular, they have infinite quadratic variation and a nontrivial quartic variation, cf.\ \cite{swanson2007}. Also, we find non-negligible negative autocovariances of increments such that the semi-martingale theory and martingale central limit theorems are not applicable to the marginal processes. Nevertheless, we show that the fundamental concept of \emph{realized volatility} as a key statistic can be adapted to the SPDE setting. While this work provides a foundation and establishes first results, we are convinced that more concepts from the high-frequency semi-martingale literature can be fruitfully transferred to SPDE models.

We consider the following linear parabolic SPDE with one space dimension
\begin{gather}
\d X_t(y)=\Big(\theta_2\frac{\partial^{2}X_t(y)}{\partial y^{2}}+\theta_1\frac{\partial X_t(y)}{\partial y}+\theta_0 X_t(y)\Big)\,\d t+\sigma_t\,\d B_{t}(y),\quad X_0(y)=\xi(y)\label{eq:spde}\\
(t,y)\in\R_+\times[y_{min},y_{max}],\quad X_{t}(y_{min})=X_{t}(y_{max})=0\text{ for }t\ge0,\nonumber 
\end{gather}
where $B_{t}$ is defined as a cylindrical Brownian motion in a Sobolev space on $[y_{min},y_{max}]$, the initial value $\xi$ is independent from $B$, and with parameters $\theta_0,\theta_1\in\R$ and $\theta_2>0$ and some volatility function $\sigma$ which we assume to depend only on time. The simple Dirichlet boundary conditions $X_{t}(y_{min})=X_{t}(y_{max})=0$ are natural in many applications. We also briefly touch on enhancements to other boundary conditions.

A solution $X=\{X_t(y),(t,y)\in[0,T]\times[y_{min},y_{max}]\}$ of \eqref{eq:spde} will be observed on a discrete grid $(t_i,y_j)\subset[0,T]\times[y_{min},y_{max}]$ for $i=1,\dots,n$ and $j=1,\dots,m$ in a fixed rectangle. More specifically we consider equidistant time points $t_i=i\Delta_n$. While the parameter vector $\theta=(\theta_0,\theta_1,\theta_2)^{\top}$ might be known from the physical foundation of the model, estimating $\sigma^2$ quantifies the level of variability or randomness in the system. This constitutes our first target of inference. An extension of our approach for inference on $\theta$ will also be discussed.

\subsection{Literature}

The key insight in the pioneering work by \cite{huebnerEtAl1993} is that for a large class of parabolic differential equations the solution of the SPDE can be written as a Fourier series where each Fourier coefficient follows an ordinary stochastic differential equation of Ornstein-Uhlenbeck type. Hence, statistical procedures for these latter processes offer a possible starting point for inference on the SPDE. Most of the available literature on statistics for SPDEs studies a scenario with observations of the Fourier coefficients in the spectral representation of the equation, see for instance \citet{huebnerroz95} and \citet{CialencoGlattHoltz2011} for the maximum likelihood estimation for linear and non-linear SPDEs, respectively. \citet{bishwal2002} discusses a Bayesian approach and \citet{cialenco2016trajectory} a trajectory fitting procedure. \cite{lototsky2000} and \citet{prakasarao2002} have studied nonparametric estimators when the Fourier coefficients are observed continuously or in discrete time, respectively. We refer to the surveys by \citet{lototsky2009} and \cite{Cialenco2018} for an overview on the existing theory.

The (for many applications) more realistic scenario where the solution of a SPDE is observed only at discrete points in time and space has been considered so far only in very few works. \citet{Markussen2003} has derived asymptotic normality and efficiency for the maximum likelihood estimator in a parametric problem for a parabolic SPDE. \cite{mohapl1997} has considered maximum likelihood and least squares estimators for discrete observations of an elliptic SPDE where the dependence structure of the observations is however completely different (in fact simpler) from the parabolic case. 

Our theoretical setup differs from the one in \citet{Markussen2003} in several ways. First, we introduce a time varying volatility $\sigma_t$ in the disturbance term. Second, we believe that discrete high-frequency data in time under infill asymptotics, where we have a finite time horizon $[0,T]$ and in the asymptotic theory the mesh of the grid tends to zero, are most informative for many prospective data applications. \citet{Markussen2003} instead focused on a low-frequency setup where the number of observations in time tends to infinity at a fixed time step. Also his number of observations in space is fixed, but we allow for an increasing number of spatial observations, too.

There are two very recent related works by \cite{cialenco} and \cite{chong}. In these independent projects the authors establish results for power variations in similar models, when $\theta_1=0$. \cite{cialenco} prove central limit theorems in a parametric model when either $m=1$ and $n\rightarrow\infty$, or $n=1$ and $m\rightarrow\infty$. For an unbounded spatial domain, \cite{chong} establishes a limit theorem for integrated volatility estimation when $m$ is fixed and $n\rightarrow\infty$. While some of our results are related, we provide the first limit theorem under double asymptotics when $n\rightarrow\infty$ and $m\rightarrow\infty$, and for the joint estimation of the volatility and an unknown parameter in the differential operator. Interestingly, the proofs of the three works are based on very different technical ingredients -- while we use mixing theory for time series the results by \cite{cialenco} rely on tools from Malliavin calculus and \cite{chong} conducts a martingale approximation by truncation and blocking techniques to apply results by \cite{jacodkey}.

\subsection{Methodology}

The literature on statistics for high-frequency observations of semi-martingales suggests to use the sum of squared increments in order to estimate the volatility. While in the present infinite dimensional model this remains generally possible, there are some deep structural differences. Supposing that $\sigma_t$ is at least $1/2$-Hölder regular, we show for any fixed spatial point $y\in(y_{min},y_{max})$ the convergence in probability
\[
  RV_n(y):=\frac1{n\sqrt{\Delta_n}}\sum_{i=1}^n (\Delta_i X)^2(y)\overset{\P}{\to}\frac{e^{-y\,\vartheta_1/\vartheta_2}}{\sqrt{\vartheta_2\pi}}\,\int_0^T\sigma^2_t\d t\quad\text{for}\quad n\to\infty,
\]
where $\Delta_iX(y):=X_{i\Delta_n}(y)-X_{(i-1)\Delta_n}(y),\Delta_n=T/n$. For a semi-martingale instead, squared increments are of order $\Delta_n$ and not $\sqrt{\Delta_n}$. We also see a dependence of the realized volatility on the spatial position $y$. Assuming $\theta_1,\theta_2$ to be known and exploiting this convergence, the construction of a consistent method of moments estimator for the integrated volatility is obvious. However, the proof of a central limit theorem for such an estimator is non-standard. 

The method of moments can be extended from one to $m$ spatial observations where $m$ is a fixed integer as $n\to\infty$, or in the double asymptotic regime where $n\to\infty$ and $m\to\infty$. It turns out that the realized volatilities $(RV_n(y_j))_{1\le j\le m}$ even de-correlate asymptotically when $m^2\Delta_n\to 0$. We prove that the final estimator attains a parametric $\sqrt{m n}$-rate and satisfies a central limit theorem. Since we have negative autocorrelations between increments $(\Delta_iX(y_j))_{1\le i\le n}$, a martingale central limit theorem cannot be applied. Instead, we apply a central limit theorem for weakly dependent triangular arrays by \citet{peligradUtev1997}. Introducing a quarticity estimator, we also provide a feasible version of the limit theorem that allows for the construction of confidence intervals. 

In view of the complex probabilistic model, our estimator is strikingly simple. Thanks to its explicit form that does not rely on optimization algorithms, the method can be easily implemented and the computation is very fast. An M-estimation approach can be used to construct a joint estimator of the parameters in \eqref{eq:spde}. While the volatility estimator is first constructed in the simplified parametric framework, by a careful decomposition of approximation errors we derive asymptotic results for the semi-parametric estimation of the integrated volatility under quite general assumptions.  

\subsection{Applications}
\enlargethispage*{.25cm}
\begin{figure}[t]
\centering
\caption{A simulated random field solving the SPDE.}
\includegraphics[width=12.0cm]{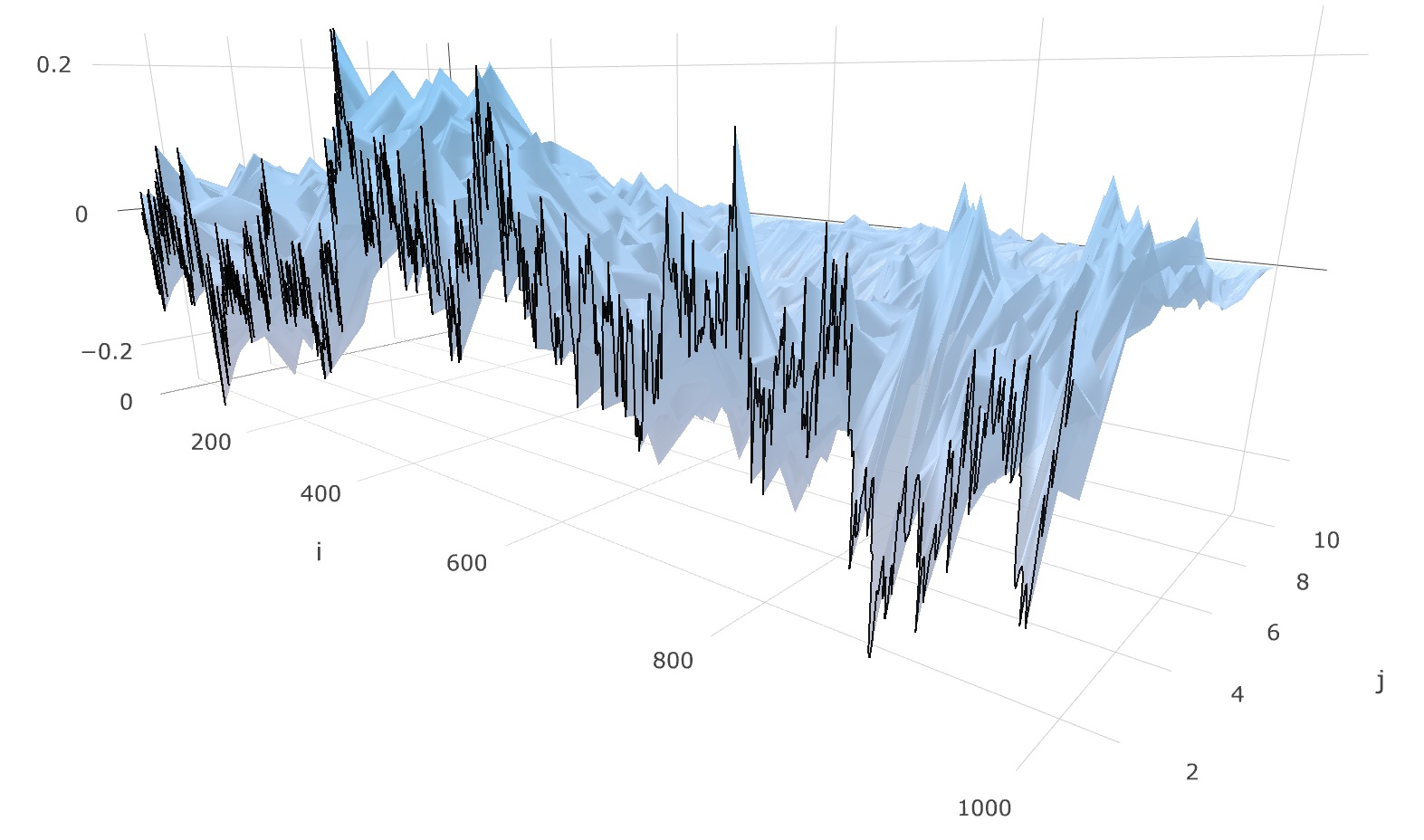}
\begin{quote}{\footnotesize \textit{Notes: One simulated solution field of the SPDE at times $i/n,i=0,\ldots,1000$ at space points $j/m,j=1,\ldots,10$. Implementation and parameter details are given in Section \ref{sec:5}. One marginal process for $j=1$ is highlighted by the black line.}} \end{quote}
\label{Fig:Intro}
\end{figure}

Although \eqref{eq:spde} is a relatively simple SPDE model, it already covers many interesting applications. Let us give some examples:
\begin{enumerate}
 \item The \textit{stochastic heat equation} 
 \[
    \d X_t(y)=\theta_2\frac{\partial^{2}X_t(y)}{\partial y^{2}}\,\d t+\theta_0 X_t(y)\,\d t+\sigma_t\,\d B_{t}(y)\,,
 \]
 with thermal diffusivity $\theta_2>0$, is the prototypical example for parabolic SPDEs. With a dissipation rate $\theta_0\ne 0$, the SPDE describes the temperature in space and time of a system under cooling or heating.
 \item The \textit{cable equation} is a basic PDE-model in neurobiology, cf.\ \citet{tuckwell2013}. Modeling the nerve membrane potential, its stochastic version is given by
 \[
  c_m \,{\d V_t(y)}=\Big(\frac{1}{r_i}\frac{\partial^{2}V_t(y)}{\partial y^{2}}-\frac{V_t(y)}{r_m}\Big)\d t+\sigma_t\,\d B_{t}(y),\quad 0<y<l,t>0,
 \]
 where $y$ is the length from the left endpoint of the cable with total length $l>0$ and $t$ is time. The parameters are the membrane capacity $c_m>0$, the resistance of the internal medium $r_i>0$ and the membrane resistance $r_m>0$. The noise terms represent the synaptic input. The first mathematical treatment of this equation goes back to \cite{walsh1986}.
 \item The \textit{term structure model} by \cite{cont2004}. It describes a yield curve as a random field, i.e.\;a continuous function of time $t$ and time to maturity $y\in[y_{min},y_{max}]$. Differently than classical literature on time-continuous interest rate models in mathematical finance, including the landmark paper by \cite{HJM} and ensuing work, which focus on modeling arbitrage-free term structure evolutions under risk-neutral measures, \cite{cont2004} aims to describe real-world dynamics of yield curves. In particular, the model meets the main stylized facts of interest rate movements as mean reversion and humped term structure of volatility. Thereto, the instantaneous forward rate curve (FRC) $r_t(y)$ is decomposed
\begin{align}\label{frc}r_t(y)=r_t(y_{min})+(r_t(y_{max})-r_t(y_{min}))\big(Y(y)+X_t(y)\big)\,.\end{align} 
When the short rate and the spread are factored out, the FRC is determined by a deterministic shape function $Y(y)$ reflecting the average profile and the stochastic deformation process $X_t(y)$ modeled by the SPDE 
\begin{gather}
\d X_t(y)=\Big(\frac{\partial X_t(y)}{\partial y}+\frac{\kappa}{2}\frac{\partial^{2}X_t(y)}{\partial y^{2}}\Big)\,\d t+\sigma_t\,\d B_{t}(y),\label{eq:spde2}\\
(t,y)\in\R_+\times[y_{min},y_{max}],\quad X_{t}(y_{min})=X_{t}(y_{max})=0\;\text{for}~ t\ge0,\nonumber 
\end{gather}
with a parameter $\kappa>0$. Figure \ref{Fig:Intro} visualizes one generated discretization of $X_t(y)$ with $\sigma\equiv 1/4$ and $\kappa=0.4$ from our Monte Carlo simulations. Observations of the FRC can either be reconstructed from observed swap rates or prices of future contracts are used. For instance, \cite{bouchaud} show how to reconstruct discrete recordings of $X_t(y)$ from observed prices of futures.\\
There are more SPDE approaches to interest rates including \cite{musiela:1993}, \cite{santaclaraSornette2000} and \cite{Bagchi2001}. 
\end{enumerate}

\subsection{Outline \& Contribution}

In Section \ref{sec:2}, we present probabilistic preliminaries on the model and our precise theoretical setup and observation model. In Section \ref{sec:3}, we start considering a parametric problem with $\sigma$ in \eqref{eq:spde} constant. The statistical properties of the discretely observed random field generated by the SPDE are examined and illustrated. We develop the method of moments estimator for the volatility and present the central limit theorem. In Section \ref{sec:estTheta} we verify that the same estimator is suitable to estimate the integrated volatility in a nonparametric framework with time-dependent volatility. 
Section \ref{sec:estTheta} addresses further a simultaneous estimation of $\theta$ and $\sigma^2$ based on a least squares approach. Joint asymptotic normality of this estimator is established with parametric rate. Section~\ref{sec:5} discusses the simulation of the SPDE model and the implementation of the methods and confirms good finite-sample properties in a Monte Carlo study. Proofs are given in Section~\ref{sec:6}.

\section{Probabilistic structure and statistical model\label{sec:2}}
\enlargethispage*{.5cm}
Without loss of generality, we set $y_{min}=0$ and $y_{max}=1$.
We apply the semi-group approach by \cite{daPratoZabczyk1992} to analyze the SPDE \eqref{eq:spde} which is associated to the differential operator
\[A_\theta :=\theta_0+\theta_1\frac{\partial}{\partial y}+\theta_2\frac{\partial^{2}}{\partial y^{2}}\]
such that the SPDE \eqref{eq:spde} reads as $\d X_t(y)=A_{\theta} X_t(y)\,\d t+\sigma_t\,\d B_{t}(y)$. The eigenfunctions $e_k$ of $A_\theta$ with corresponding eigenvalues $-\lambda_k$ are given by
\begin{align}\label{eq:eigenf}
  e_k(y)&=\sqrt{2}\sin\big(\pi ky\big)\exp\big(-\frac{\theta_1}{2\theta_2}y\big),\quad y\in [0,1],\\
  \lambda_k&=-\theta_0+\frac{\theta_1^2}{4\theta_2}+\pi^2k^2\theta_2,\qquad k\in\N.\label{eq:eK}
\end{align}
This eigendecomposition takes a key role in our analysis and for the probabilistic properties of the model. Of particular importance is that $\lambda_k$ increase proportionally with $k^2$, $\lambda_k\propto k^2$, and that $e_k(y)$ depend exponentially on $y$ with an additional oscillating factor.
Note that $(e_k)_{k\ge1}$ is an orthonormal basis of the Hilbert space $H_\theta:=\{f\colon[0,1]\to\R: \|f\|_\theta<\infty\}$ with
\begin{align*}
  \langle f,g\rangle_{\theta}&:=\int_{0}^{1}e^{y\theta_1/\theta_2}f(y)g(y)\,\d y\qquad\text{and}\qquad \|f\|_\theta^2:=\langle f,f\rangle_\theta.
\end{align*}
We choose $H_\theta$ as state space for the solutions of \eqref{eq:spde}. In particular, we suppose that $\xi\in H_{\theta}$. $A_\theta$ is self-adjoint on $H_\theta$ and the domain of $A_\theta$ is $\{f\in H_\theta: \|f'\|_\theta,\|f''\|_\theta<\infty, f(0)=f(1)=0\}$. The cylindrical Brownian motion $(B_t)_{t\ge 0}$ in \eqref{eq:spde} can be defined via
\[
  \langle B_t,f\rangle_\theta=\sum_{k\ge1}\langle f,e_k\rangle_\theta W_t^k,\quad f\in H_\theta, t\ge0,
\]
for independent real-valued Brownian motions $(W_t^k)_{t\ge0}$, $k\ge 1$. 
$X_{t}(y)$ is called a \emph{mild solution} of (\ref{eq:spde}) on $[0,T]$ if it satisfies the integral equation for any $t\in[0,T]$
\[X_t=e^{tA_\theta}\xi+\int_0^te^{(t-s)A_\theta}\sigma_s\,\d B_s\quad a.s.\]
Defining the coordinate processes $x_{k}(t):=\langle X_{t},e_{k}\rangle_\theta,t\ge0,$ for any $k\ge1$, the random field $X_{t}(y)$ can thus be represented as the infinite factor model
\begin{equation}
X_{t}(y)=\sum_{k\ge1}x_{k}(t)e_{k}(y)\quad\text{with}\quad 
x_{k}(t)=e^{-\lambda_{k}t}\langle \xi,e_{k}\rangle_{\theta}+\int_{0}^{t}e^{-\lambda_{k}(t-s)}\sigma_s\,\d W_{s}^{k}.\label{eq:Representation}
\end{equation}
In other words, the coordinates $x_k$ satisfy the Ornstein-Uhlenbeck dynamics
\begin{equation}
\d x_{k}(t)=-\lambda_{k}x_{k}(t)\d t+\sigma_t\,\d W_{t}^{k},\quad x_{k}(0)=\langle \xi,e_{k}\rangle_{\theta}\,.\label{eq:ornsteinUhlenbeck}
\end{equation}
There is a modification of the stochastic convolution $\int_0^\cdot e^{(\cdot-s)A_\theta}\sigma_s\,\d B_s$ which is continuous in time and space and thus we may assume that $(t,y)\mapsto X_t(y)$ is continuous, cf.\ \citet[Thm. 5.22]{daPratoZabczyk1992}. The neat representation \eqref{eq:Representation} separates dependence on time and space and connects the SPDE model to stochastic processes in Hilbert spaces. This structure has been exploited in several works, especially for financial applications, see, for instance, \cite{Bagchi2001} and \cite{schmidt}. We consider the following observation scheme.
\begin{ass}[Observations]\label{Obs}
  Suppose we observe a mild solution $X$ of (\ref{eq:spde}) on a discrete grid $(t_{i},y_{j})\in[0,1]^2$ with 
  \begin{equation}\label{eq:observations}
  t_{i}=i\Delta_{n}\quad\text{for}\quad i=0,\dots,n\quad\text{and}\quad \delta\le y_{1}<y_2<\dots <y_m\le 1-\delta
  \end{equation}
  where $n,m\in\N$ and $\delta>0$. We consider an infill asymptotics regime where $n\Delta_n=1$, $\Delta_n\rightarrow 0$ as $n\rightarrow\infty$ and $m<\infty$ is either fixed or it may diverge, $m\rightarrow\infty$, such that $m=m_n=\mathcal{O}(n^{\rho})$ for some $\rho\in(0,1/2)$. We assume that $m\cdot\min_{j=2,\dots,m}|y_j-y_{j-1}|$ is bounded from below, uniformly in $n$. 
\end{ass}
Here and throughout, we write $A_n=\mathcal O(B_n)$, if there exist a constant $C>0$ independent of $m,n$, and of potential indices $1\le i\le n$ or $1\le j\le m$, and some $n_0\in\N$, such that $|A_n|\le C B_n $ for all $n\ge n_0$. While we have fixed the time horizon to $T=1$, the subsequent results can be easily extended to any finite $T$.
We write $\Delta_i X(y)=X_{i\Delta_n}(y)-X_{(i-1)\Delta_n}(y),1\le i\le n$, for the increments of the marginal processes in time and analogously for increments of other processes.

The condition $m=\mathcal{O}(n^{\rho})$ for some $\rho\in(0,1/2)$ implies that we provide asymptotic results for discretizations which are finer in time than in space. For instance in the case of an application to term structure data, this appears quite natural, since we can expect many observations over time of a moderate number of different maturities. 
Especially for financial high-frequency intra-day data, we typically have at hand several thousands intra-day price recordings per day, but usually at most 10 frequently traded different maturities. For instance, Eurex offers four classes of German government bonds with different contract lengths.\footnote{see \url{www.eurexchange.com/exchange-en/products/int}} For each, futures with 2-3 different maturities are traded on a high-frequency basis. While we can thus extract observations in at most 12 different spatial points from this data, 5000-10000 intra-daily prices are available, see \cite[Sec.\ 4]{ecb}. The relation between $m=15$ and $n$ is similar in \cite{bouchaud}, where, however, daily prices are used.

We impose the following mild regularity conditions on the initial value of the SPDE \eqref{eq:spde}.

\begin{ass}\label{cond} In \eqref{eq:spde} we assume that
  \begin{enumerate}
  \item[(i)] either $\E[\langle\xi,e_k\rangle_\theta]=0$ for all $k\ge1$ and $\sup_k \lambda_k\E[\langle\xi,e_k\rangle_\theta^2]<\infty$ holds true or \\$\E[\|A_\theta^{1/2}\xi\|_\theta^2]<\infty$;
  \item[(ii)] $(\langle\xi,e_k\rangle_\theta)_{k\ge1}$ are independent.
  \end{enumerate}
\end{ass}
This assumption is especially satisfied if $\xi$ is distributed according to the stationary distribution of \eqref{eq:spde}, where $\langle \xi,e_k\rangle_\theta$ are independently $\mathcal N(0,\sigma^2/(2\lambda_k))$-distributed. Note also that $\E[\|A_\theta^{1/2}\xi\|_\theta^2]<\infty$ implies $\sup_k \lambda_k\E[\langle\xi,e_k\rangle_\theta^2]=\sup_k \E[\langle A_\theta^{1/2}\xi,e_k\rangle_\theta^2]<\infty$. Assuming independence of the sequence $(\langle\xi,e_k\rangle_\theta)_{k\ge1}$ is a convenient condition for an analysis of the variance-covariance structure of our estimator, but could be replaced by other moment-type conditions on the coefficients of the initial value. 

\begin{rem}
  Due to the first-order term in \eqref{eq:spde}, when $\theta_1\ne 0$, our solution is in general not symmetric in space. One may expect, however, symmetry in the following sense: For $\theta_0 = 0$, starting in $\xi=0$, the law of $X_t(1/2+r), r \in[0,1/2]$, with $\theta_1>0$ should be the same as the law of $X_t(1/2-r)$ with the parameter $-\theta_1$. This property is satisfied, when we illustrate the solution \eqref{eq:Representation} with respect to the modified eigenfunctions
  \[
    \tilde e_k(y):=\sqrt{2}\sin(\pi ky)\exp\big(-\frac{\theta_1}{2\theta_2}\big(y-\frac{1}{2}\big)\big)=e_k(y)e^{\theta_1/(4\theta_2)}\,,
  \]
  corresponding to the same eigenvalues $-\lambda_k$, which form an orthonormal basis of $L^2([0,1])$ with respect to the scalar product $\langle f,g\rangle=\int_0^1e^{(y-1/2)\theta_1/\theta_2}f(y)g(y)\d y$. Then $\tilde e_k(1/2+r)$ for parameter $\theta_1$ equals $\pm \tilde e_k(1/2-r)$ for parameter $-\theta_1$, where the sign depends on $k$, and the symmetric centered Gaussian distribution of $x_k(t)$ ensures the discussed symmetry of the law of $X_t(y)$ around $y=1/2$. This adjustment by multiplying the constant $\exp(\theta_1/(4\theta_2))$, which depends on $\theta$ but not on $y$ or $t$, can be easily incorporated in all our results.
	
\end{rem}

\section{Estimation of the volatility parameter\label{sec:3}}
In this section, we consider the parametric problem where $\sigma$ is a constant volatility parameter. Suppose first that we want to estimate $\sigma^2$ based on $n\rightarrow\infty$ discrete recordings of $X_t(y)$, along only one spatial observation $y\in(\delta,1-\delta),\,\delta>0$, $m=1$. Let us assume first that $\theta$ is known. We thus know $(\lambda_{k},e_{k})_{k\ge1}$ explicitly from \eqref{eq:eK}. In many applications, $\theta$ will be given in the model or reliable estimates could be available from historical data. Nevertheless, we address the parameter estimation of $\theta$ in Section \ref{sec:estTheta}.

Volatility estimation relies typically on squared increments of the observed process. We develop a method of moment estimator here utilizing squared increments $(\Delta_i X)^2(y),1\le i\le n$, too. However, the behavior of the latter is quite different compared to the standard setup of high-frequency observations of semi-martingales. 

Though we do not observe the factor coordinate processes in \eqref{eq:Representation}, understanding the dynamics of their increments is an important ingredient of the analysis of $(\Delta_i X)^2(y),1\le i\le n$. The increments of the Ornstein-Uhlenbeck processes from \eqref{eq:Representation} can be decomposed as
  \begin{align}
    \Delta_i x_k&=\underbrace{\langle \xi, e_k\rangle_\theta\big(e^{-\lambda_ki\Delta_n}-e^{-\lambda_k(i-1)\Delta_n}\big)}_{=A_{i,k}}+\underbrace{\int_0^{(i-1)\Delta_n}\sigma e^{-\lambda_k((i-1)\Delta_n-s)}(e^{-\lambda_k \Delta_n}-1)\,\d W_s^k}_{=B_{i,k}}\notag\\
    &\hspace*{7cm}+\underbrace{\int_{(i-1)\Delta_n}^{i\Delta_n}\sigma e^{-\lambda_k(i\Delta_n-s)}\,\d W_s^k\,.}_{=C_{i,k}}\label{eq:ABC}
  \end{align}
For a fixed finite intensity parameter $\lambda_k$, observing $\Delta_i x_k$ corresponds to the classical semi-martingale framework. In this case, $A_{i,k}$, $B_{i,k}$ are asymptotically negligible and the exponential term in $C_{i,k}$ is close to one, such that $\sigma^2$ is estimated via the sum of squared increments. Convergence to the quadratic variation when $n\rightarrow\infty$ implies consistency of this estimator, well-known as the \emph{realized volatility}. 

It is important that our infinite factor model of $X$ includes an infinite sum involving $(\Delta_i x_k)_{k\ge 1}$ with increasing $\lambda_k\propto k^2$. In this situation the terms $B_{i,k}$ are not negligible and induce a quite distinct behavior of the increments. The larger $k$, the stronger the mean reversion effect and the smaller the variance of the Ornstein-Uhlenbeck coordinate processes, such that higher addends have decreasing influence in \eqref{eq:Representation}. In simulations, for instance in Figure \ref{Fig:1}, $X_t(y)$ is generated summing over $1\le k\le K$, up to some sufficiently large cut-off frequency $K$. An important consequence is that the squared increments $(\Delta_iX(y))^2$ are of order $\Delta_n^{1/2}$, while for semi-martingales the terms $(W^k_{(i+1)\Delta_n}-W^k_{i\Delta_n})^{2}$ induce the order $\Delta_n$. Hence, the aggregation of the coefficient processes leads to a rougher path $t\mapsto X_t(y)$.
\begin{prop}\label{propExp}
  On Assumptions \ref{Obs} and \ref{cond}, for constant $\sigma$, we have uniformly in $y\in[\delta,1-\delta]$ that
  \begin{align}\label{sqin}
    \E\big[(\Delta_i X)^2(y)\big]&=\Delta_n^{1/2}e^{-y\,\vartheta_1/\vartheta_2}\frac{\sigma^2}{\sqrt{\vartheta_2\pi}}+r_{n,i}+\mathcal O\big(\Delta_n^{3/2}\big)
  \end{align}
  for $i=1,\ldots,n$,  with terms $r_{n,i}$ that satisfy $\sup_{1\le i\le n} |r_{n,i}|=\mathcal O(\Delta_n^{1/2}),\sum_{i=1}^n r_{n,i}=\mathcal O(\Delta_n^{1/2})$, and become negligible when summing all squared increments:
  \begin{align*}
   \E\Big[\frac1{n\Delta_n^{1/2}}\sum_{i=1}^n(\Delta_i X)^2(y)\Big]&=e^{-y\,\vartheta_1/\vartheta_2}\frac{\sigma^2}{\sqrt{\vartheta_2\pi}} +\mathcal O\big(\Delta_n\big). \end{align*}
\end{prop}
At first, the expressions in \eqref{sqin} hinge on the time point $i\Delta_n$, but this dependence becomes asymptotically negligible at first order. Moreover, the first-order expectation of squared increments shows no oscillating term in $y$, only the exponential term. This crucially simplifies the structure. The dependence on the spatial coordinate $y$ is visualized in the left plot of Figure \ref{Fig:1}.
\begin{figure}[t]
\centering
\caption{Realized volatilities and autocorrelations of increments.}
\includegraphics[width=70mm]{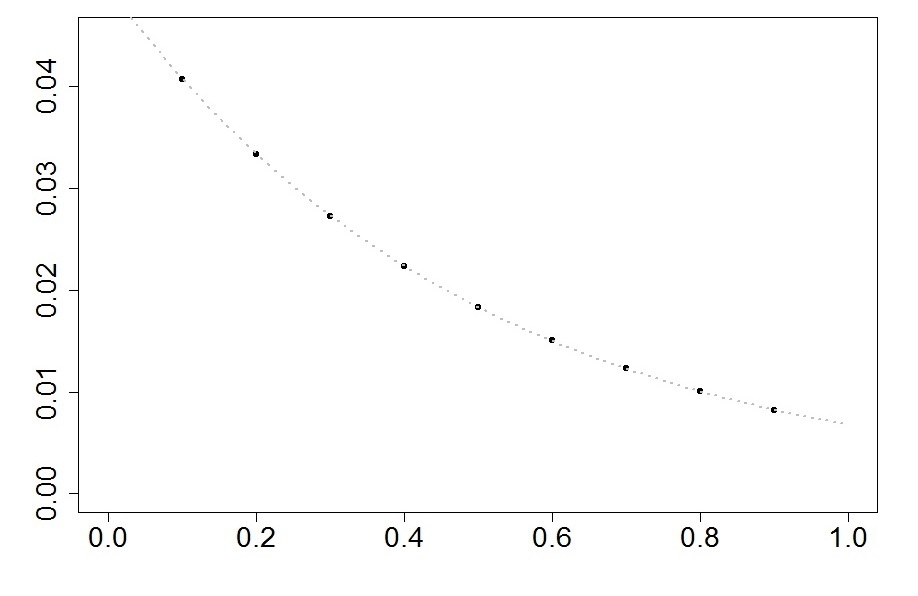}\includegraphics[width=70mm]{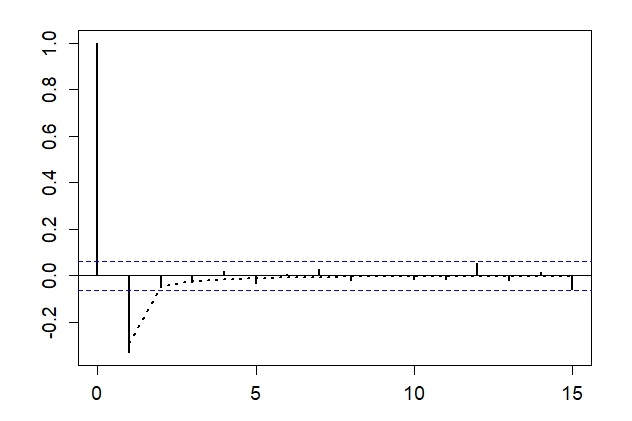}
\begin{quote}{\footnotesize \textit{Notes: In the left panel, the dark points show averages of realized volatilities $n^{-1}\Delta_n^{-1/2}\sum_{i=1}^n (\Delta_i X)^2(y_j)$ for $y_j=j/10,j=1,\ldots,9$ based on 1,000 Monte Carlo replications with $n=1,000$ and a cut-off frequency $K=10,000$. The grey dashed line depicts the function $(\pi\vartheta_2)^{-1/2}\exp(- y\vartheta_1/\vartheta_2)\sigma^2$, $\sigma=1/4,\vartheta_0=0,\vartheta_1=1,\vartheta_2=1/2$. The right panel shows empirical autocorrelations from one of these Monte Carlo iterations for $y=8/10$. The theoretical decay of autocorrelations for lags $|j-i|= 1,\ldots,15$ is added by the dashed line. }}\end{quote}
\label{Fig:1}
\end{figure}
In view of the complex model, a method of moments approach based on \eqref{sqin} brings forth a surprisingly simple consistent estimator for the squared volatility
\begin{align}\label{estm1}
  \hat\sigma^2_{y}=\frac{\sqrt{\pi\, \vartheta_2}}{n\sqrt{\Delta_n}}\sum_{i=1}^n(\Delta_iX)^2(y) \exp(y\,\vartheta_1/\vartheta_2)\,.
\end{align}
In particular, an average of rescaled squared increments (a rescaled realized volatility) depends on $y$ (only) via the multiplication with $\exp(y\,\vartheta_1/\vartheta_2)$. Also, the laws of $(\Delta_iX)^2(y) $ depend on $i$, but we derive a consistent estimator \eqref{estm1} as a non-weighted average over all time instants. The negligibility of the terms $r_{n,i}$ in \eqref{sqin} renders this nice and simple structure. Note that the estimator $\hat\sigma^2_y$ depends on $\theta$. The case of unknown $\theta$ is discussed in Section \ref{sec:estTheta}.

Assessing the asymptotic distribution of the volatility estimator, however, requires a more careful handling of the higher order terms. Especially, we cannot expect that the increments $(\Delta_i X)^2(y)$, $1\le i\le n$, are uncorrelated. 
\begin{prop}\label{propacf}
   On Assumptions \ref{Obs} and \ref{cond}, for constant $\sigma$, the covariances of increments\\ $(\Delta_i X(y))_{i=1,\ldots,n}$ satisfy uniformly in $y\in[\delta,1-\delta]$ for $|j-i|\ge 1$:
  \begin{align}
    \label{acfeq}&\cov\big(\Delta_i X(y),\Delta_j X(y)\big)\\
    &\qquad=-\Delta_n^{1/2}\sigma^2\frac{e^{-y\,\vartheta_1/\vartheta_2}}{2\sqrt{\theta_2\pi}}\,\big(2\sqrt{|j-i|}-\sqrt{|j-i|-1}-\sqrt{|j-i|+1}\big)+r_{i,j}+\mathcal O(\Delta_n^{3/2}),\notag
  \end{align}
  where the remainder terms $r_{i,j}$ are negligible in the sense that $\sum_{i,j=1}^nr_{i,j}=\mathcal{O}(1)$.
\end{prop}
While standard semi-martingale models lead to  discrete increments that are (almost) uncorrelated (martingale increments are uncorrelated), \eqref{acfeq} implies negative autocovariances. This fact highlights another crucial difference between our discretized SPDE model and classical theory from the statistics for semi-martingales literature. We witness a strong negative correlation of consecutive increments for $|j-i|=1$. Autocorrelations decay proportionally to $|j-i|^{-3/2}$ with increasing lags, see Figure \ref{Fig:1}. Covariances hinge on $y$ at first order only via the exponential factor $\exp(-y\,\vartheta_1/\vartheta_2)$, such that autocorrelations do not depend on the spatial coordinate $y$. For lag $|j-i|=1$, the theoretical first-order autocorrelation given in Figure \ref{Fig:1} is $(\sqrt{2}-2)/2\approx -0.292$, while for lag 2 and lag 3 we have factors of approximately $-0.048$ and $-0.025$. For lag 10 the factor decreases to $-0.004$. Also the autocorrelations from simulations in Figure \ref{Fig:1} show this fast decay of serial correlations.

Due to the correlation structure of the increments $(\Delta_i X)_{1\le i\le n}$, martingale central limit theorems as typically used in the volatility literature do not apply. In particular, Jacod's stable central limit theorem, cf.\ \citet[Theorem 3--1]{jacodkey}, which is typically exploited in the literature on high-frequency statistics can not be (directly) applied. \cite{power} and related works provide a stable limit theorem for power variations of stochastic integrals with respect to general Gaussian processes. However, it does also not apply to our model based on the spectral decomposition \eqref{eq:Representation}.

Instead, we exploit the decay of the autocovariances and the profound theory on \emph{mixing} for Gaussian sequences to establish a central limit theorem. In particular, \cite{utev} has proved a central limit theorem for $\rho$-mixing triangular arrays. However, the sequence $(\Delta_{i}X(y))_{i\ge1}$, which is stationary under the stationary initial distribution, has a spectral density with zeros such that we can not conclude that it is $\rho$-mixing based on the results by \cite{rozanov} and the result by \cite{utev} can not be applied directly. A careful analysis of the proof of Utev's central limit theorem shows that the abstract $\rho$-mixing assumption can be replaced by two explicit conditions on the variances of partial sums and on covariances of characteristic functions of partial sums. The resulting generalized central limit theorem has been reported by \citet[Theorem B]{peligradUtev1997} and we can verify these generalized mixing-type conditions in our setup, see Proposition~\ref{prop:mixing}. In particular, we prove in Proposition \ref{varprop} that $\var(\hat\sigma_y^2)$ is of the order $n^{-1}$, such that the estimator attains a usual parametric convergence rate. Altogether, we establish the following central limit theorem for estimator \eqref{estm1}.
\begin{thm}\label{cltm1}
  On Assumptions \ref{Obs} and \ref{cond}, for constant $\sigma$, for any $y\in[\delta,1-\delta]$ the estimator \eqref{estm1} obeys as $n\rightarrow\infty$ the central limit theorem
  \begin{align}\label{clt1}n^{1/2}\big(\hat\sigma_y^2-\sigma^2\big)\stackrel{d}{\rightarrow} \mathcal{N}\big(0,\pi\Gamma\sigma^4\big)\,,\end{align}
  with Gaussian limit law where $\Gamma\approx 0.75$ is a numerical constant analytically given in \eqref{eq:constantVariance}.
\end{thm}
\eqref{clt1} takes a very simple form, relying, however, on several non-obvious computations and bounds in the proof in Section \ref{sec:6}. 

Consider next an estimation of the volatility when we have at hand observations $(\Delta_i X)^2(y),1\le i\le n$, along several spatial points $y_1,\ldots,y_m$.
Proposition~\ref{varprop} shows that the covariances $\cov(\hat\sigma^2_{y_1},\hat\sigma^2_{y_2})$ vanish asymptotically for $y_1\ne y_2$ as long as $m^2\Delta_n\to0$. Therefore, estimators in different spatial points de-correlate asymptotically. Since the realized volatilities hinge on $y$ at first order only via the factor $\exp(-y\vartheta_1/\vartheta_2)$, the first-order variances of the rescaled realized volatilities do not depend on $y\in(0,1)$. We thus define the volatility estimator
\begin{align}\label{estm}
 \hat\sigma_{n,m}^2:=&\frac 1m \sum_{j=1}^m \hat\sigma^2_{y_j} 
 = \,\frac{\sqrt{\pi\, \vartheta_2}}{m\,n\sqrt{\Delta_n}} \sum_{j=1}^m \sum_{i=1}^n(\Delta_iX)^2(y_j) \exp(y_j\,\vartheta_1/\vartheta_2)\,.
\end{align}
\begin{thm}\label{cltm2}
  On Assumptions \ref{Obs} and \ref{cond}, for constant $\sigma$ and for some sequence $m=m_n$, the estimator \eqref{estm} obeys as $n\rightarrow\infty$ the central limit theorem
  \begin{align}\label{clt2}(m_n\cdot n)^{1/2}\big(\hat\sigma_{n,m_n}^2-\sigma^2\big)\stackrel{d}{\rightarrow} \mathcal{N}\big(0,\pi\Gamma\sigma^4\big)\,,\end{align}
  with Gaussian limit law where $\Gamma\approx 0.75$ is a numerical constant analytically given in \eqref{eq:constantVariance}.
\end{thm}
Since $m n$ is the total number of observations of the random field $X$ on the grid, the estimator $\hat\sigma_{n,m}^2$ achieves the parametric rate $\sqrt{mn}$ as for independent observations if the discretizations are finer in time than in space under Assumption \ref{Obs}.
The constant $\pi\Gamma\approx 2.357$ in the asymptotic variance is not too far from 2 which is clearly a lower bound, since $2\sigma^4$ is the Cramér-Rao lower bound for estimating the variance $\sigma^2$ from i.i.d.\,standard normals. In fact, the term $(\pi\Gamma-2)\sigma^4$ is the part of the variance induced by non-negligible covariances of squared increments.

In order to use the previous central limit theorem to construct confidence sets or tests for the volatility, let us also provide a normalized version applying a quarticity estimator.

\begin{prop}\label{propfclt}Suppose that $\sup_k\lambda_k\E[\langle \xi,e_k\rangle_{\theta}^l]<\infty$ for $l=4,8$. Then, on Assumptions \ref{Obs} and \ref{cond}, for constant $\sigma$, the quarticity estimator
  \begin{align}\label{quart}\tilde\sigma_{n,m}^4=\frac{\theta_2\pi}{3m}\sum_{j=1}^m\sum_{i=1}^n (\Delta_i X)^4(y_j) \exp(2y_j\theta_1/\theta_2)\,,\end{align}
  satisfies $\tilde\sigma_{n,m_n}^4 \stackrel{\P}{\rightarrow}\sigma^4$. In particular, we have as $n\to\infty$ that
  \begin{align}\label{fclt}(m_n\cdot n)^{1/2}(\pi\Gamma\tilde\sigma_{n,m_n}^4)^{-1/2}\big(\hat\sigma_{n,m_n}^2-\sigma^2\big)\stackrel{d}{\rightarrow} \mathcal{N}\big(0,1\big)\,.\end{align}
\end{prop}
The mild additional moment assumptions in Proposition \ref{propfclt} are particularly satisfied when the initial condition is distributed according to the stationary distribution. \eqref{fclt} holds as $m_n\to\infty$, and also for any fix $m$, as $n\to\infty$.\\
Finally, let us briefly discuss the effect of different boundary conditions. In the case of a non-zero Dirichlet boundary condition our SPDE model reads as
\begin{gather*}
  \begin{cases}
    \d X_t(y)=(A_\theta X_t(y)) \d t +\sigma_t\,\d B_t(y)&\qquad \text{in }[0,1],\\
     X_0=\xi,\quad X_t(0)=g(0),\quad X_t(1)=g(1),   
  \end{cases}
\end{gather*}
where we may assume that the function $g$ is twice differentiable with respect to $y$. Then $\bar X_t:=(X_t-g)\in H_\theta$ is a mild solution of the zero boundary problem
\begin{gather*}
  \begin{cases}
    \d \bar X_t=(A_\theta \bar X_t) \d t -(A_\theta g)\d t+\sigma_t\,\d B_t&\qquad \text{in }[0,1],\\
     \bar X_0=\xi-g,\quad\bar X_t(0)=0,\quad \bar X_t(1)=0.   
  \end{cases}
\end{gather*}
Hence,
\[
  \bar X(t)=e^{tA_\theta}(\xi-g)+\int_0^te^{(t-s)A_\theta}A_\theta g\,\d t+\int_0^te^{(t-s)A_\theta}\sigma_t\,\d B_t\quad a.s.
\]
In the asymptotic analysis of estimator \eqref{estm}, we would have to take into account the additional deterministic and known term $\int_0^te^{(t-s)A_\theta}A_\theta g\,\d t$, which is possible with only minor modifications.

Especially for the cable equation, Neumann's boundary condition is also of interest and corresponds to sealed ends of the cable. In this case we have $\frac{\partial}{\partial y}X_t(0)=\frac{\partial}{\partial y}X_t(1)=0$ and $\theta_1=0$, such that we only need to replace the eigenfunction $(e_k)$ from \eqref{eq:eK} by 
\(
  \bar e_k(y)=\sqrt 2\cos(\pi ky).
\) We expect that our estimators achieve the same results in this case. 

\section{Estimation of the integrated volatility with unknown parameters in the differential operator}\label{sec:estTheta}
Concerning the estimation of $\theta=(\theta_0,\theta_1, \theta_2)^{\top}$, we have to start with the following restricting observation: For any fixed $m$ the process $Z:=(X_t(y_1),\dots,X_t(y_m))_{t\in[0,1]}$ is a multivariate Gaussian process with continuous trajectories (for a modification and under a Gaussian initial distribution), cf.\ \citet[Thm.\ 5.22]{daPratoZabczyk1992}. Consequently, its probability distribution is characterized by the mean and the covariance operator. It is then well known from the classical theory on statistics for stochastic processes, that even with continuous observations of the process $Z$, the underlying drift parameter cannot be consistently estimated from observations within a fixed time horizon. Therefore, $\theta_0$ is not identifiable in our high-frequency observation model. Any estimator of $\theta$ should only rely on the covariance operator. Proposition~\ref{propacf} reveals that for $\Delta_n\to0$ the latter depends only on
\[\sigma^2_0:=\sigma^2/\sqrt{\theta_2}\quad\text{and}\quad \varkappa:=\theta_1/\theta_2.\]
If both $\sigma$ and $\theta$ are unknown, the \emph{normalized volatility parameter} $\sigma^2_0$ and the \emph{curvature parameter} $\varkappa$ thus seem the only identifiable parameters based on high-frequency observations over a fixed time horizon. This allows, for instance, the complete calibration of the parameters in the yield curve model \eqref{eq:spde2}.

In order to estimate $\varkappa$, we require observations in at least two distinct spatial points $y_1$ and $y_2$. The behavior of the estimates of the realized volatilities along different $y$, as seen in \eqref{sqin} and Figure~\ref{Fig:1}, motivate the following curvature estimator
\begin{align}\label{kappaest1}
  \tilde\varkappa=\frac{\log\big(\sum_{i=1}^n (\Delta_i X)^2(y_1)\big)-\log\big(\sum_{i=1}^n (\Delta_i X)^2(y_2)\big)}{y_2-y_1}\,.
\end{align} 
The previous analysis and the delta method facilitate a central limit theorem with $\sqrt{n}$-convergence rate for this estimator $\tilde\varkappa$.\\ 
Especially in the finance and econometrics literature, heterogeneous volatility models with dynamic time-varying volatility are of particular interest. In fact, different in nature than a maximum likelihood approach, our moment estimators \eqref{estm1} and \eqref{estm} serve as semi-parametric estimators in the general time-varying framework. In the sequel, we consider $(\sigma_s)_{s\in[0,1]}$ a time-varying volatility function.
\begin{ass}\label{assvola}Assume $(\sigma_s)_{s\in[0,1]}$ is a strictly positive deterministic function that is $\alpha$-Hölder regular with Hölder index $\alpha\in(1/2,1]$, that is,
\(\big|\sigma_{t+s}-\sigma_{t}\big|  \le C\,s^{\alpha},\)\;
for all $0\leq t< t+s\leq1$, and some positive constant $C$.
\end{ass}
The \emph{integrated volatility} $\int_0^1\sigma_s^2\, \d s$ aggregates the overall variability of the model and is one standard measure in econometrics and finance to quantify risk. For volatility estimation of a continuous semi-martingale, the properties of the realized volatility $RV$ in a nonparametric setting with time-varying volatility are quite close to the parametric case. In fact, the central limit theorem $\sqrt{n}(RV-\sigma^2)\stackrel{d}{\rightarrow}\mathcal{N}(0,2\sigma^4)$ extends to $\sqrt{n}(RV-\int_0^1\sigma_s^2\,\d s)\stackrel{d}{\rightarrow}\mathcal{N}(0,2\int_0^1\sigma_s^4\,\d s)$ under mild conditions on $(\sigma_s)_{s\in[0,1]}$, see \cite[Sec.\ 5.6.1]{JP}. In our model, however, such a generalization is more difficult due to the non-negligibility of the stochastic integrals over the whole past $B_{i,k}$ in \eqref{eq:ABC}.\\ 
For unknown $\theta$ and unknown time-dependent volatility, we propose a least squares approach for the semi-parametric estimation of
\[IV_0:=\frac{1}{\sqrt{\theta_2}}\int_0^1\sigma_s^2\,\d s\quad\text{and}\quad \varkappa=\theta_1/\theta_2\,.\]
In view of \eqref{sqin}, we rewrite 
\begin{align}\label{zj}
  Z_j:=\frac1{n\sqrt{\Delta_n}}\sum_{i=1}^n(\Delta_iX)^2(y_j)=f_{IV_{0},\varkappa}(y_j)+\delta_{n,j}\qquad\text{for}\quad f_{s,k}(y):=\frac s{\sqrt\pi}e^{-ky},
\end{align}
in the form of a parametric regression model with the parameter $\eta=(IV_0,\varkappa)$ and with non-standard observation errors $\delta_{n,j}$. We define 
\begin{equation}\label{eq:leastSquares}
 (\widehat{IV}_{\negthinspace 0},\hat{\varkappa}):=\argmin_{s,k}\sum_{j=1}^{m}\Big(Z_{j}-f_{s,k}(y_{j})\Big)^{2}. 
\end{equation}
This M-estimator is a classical least squares estimator. Combining classical theory on minimum contrast estimators with
an analysis of the random variables $(\delta_{n,j})$ yields the
following limit theorem. For simplicity we suppose that $IV_0$ and $\varkappa$ belong to a compact parameter set. 
\begin{thm}\label{thm:leastSquares}
  Grant Assumptions \ref{Obs} with $y_1=\delta$, $y_m=1-\delta$ and $m|y_j-y_{j-1}|$ uniformly bounded from above and from below, \ref{cond} and \ref{assvola}. Let $\eta=(IV_0,\varkappa)\in\Xi$ for some compact subset $\Xi\subset(0,\infty)\times[0,\infty)$.
  Then the estimators $(\widehat{IV}_{\negthinspace 0},\hat{\varkappa})$ from \eqref{eq:leastSquares} satisfy for a sequence $m=m_n\to\infty$ with $\sqrt{m_n}\Delta_n^{\alpha'-1/2}\to 0$ for some $\alpha'<\alpha$, as $n\to\infty$ the central limit theorem 
  \begin{align}\label{cltlq}
    (m_n\cdot n)^{1/2}\big((\widehat{IV}_{\negthinspace 0},\hat{\varkappa})^\top-({IV}_{0},\varkappa)^\top\big)\stackrel{d}{\rightarrow} \mathcal N\Big(0,\frac{\Gamma\pi}{\theta_2}\int_0^1\negthinspace \sigma_s^4\,\d s\, V(\eta)^{-1}U(\eta)V(\eta)^{-1}\Big)
  \end{align}
  with strictly positive definite matrices
  \begin{align}
      U(\eta)&:=\begin{pmatrix} \int_{\delta}^{1-\delta}e^{-4\varkappa y}\,\d y & -IV_{0}\int_{\delta}^{1-\delta}ye^{-4\varkappa y}\,\d y\\ -IV_{0}\int_{\delta}^{1-\delta}ye^{-4\varkappa y}\,\d y & IV_0^2\int_{\delta}^{1-\delta}y^2e^{-4\varkappa y}\,\d y\end{pmatrix},\label{eq:U}\\
      V(\eta)&:=\begin{pmatrix} \int_{\delta}^{1-\delta}e^{-2\varkappa y}\,\d y & -IV_{0}\int_{\delta}^{1-\delta}ye^{-2\varkappa y}\,\d y\\ -IV_{0}\int_{\delta}^{1-\delta}ye^{-2\varkappa y}\,\d y & IV_0^2\int_{\delta}^{1-\delta}y^2e^{-2\varkappa y}\,\d y\end{pmatrix}.\label{eq:V}
  \end{align}
\end{thm}
It can be easily deduced from the proof that we obtain an analogous result for any fixed $m\ge2$ where the integrals $\int_\delta^{1-\delta}h(y)\,\d y$, for a generic function $h$, have to be replaced by $\frac{1}{m}\sum_{j=1}^mh(y_j)$. In particular, the Cauchy-Schwarz inequality shows that the determinant of $V(\eta)$ is non-zero and thus $V$ is invertible. The condition $y_1=\delta$ and $y_m=1-\delta$ is without loss of generality, in general we integrate over $[y_1,y_m]$ in the entries of $U(\eta)$ and $V(\eta)$. A sufficient condition that $\sqrt{m_n}\Delta_n^{\alpha'-1/2}\to 0$ for some $\alpha'<\alpha$, is $\alpha>(1+\rho)/2$ in Assumption \ref{assvola} with $\rho$ from Assumption \ref{Obs}. For any fix $m$, we only need that $\alpha>1/2$.\\
For the case that $\theta_2$ is known, we can generalize the results from Section \ref{sec:3} to the estimation of integrated volatility. We keep to the notation for the estimators motivated by the parametric model. 
\begin{cor}\label{gclt}
  On Assumptions \ref{Obs}, \ref{cond} and \ref{assvola}, for any $y\in[\delta,1-\delta]$, the estimator \eqref{estm1} obeys as $n\rightarrow\infty$ the central limit theorem
  \begin{align}\label{gcltm1eq}n^{1/2}\Big(\hat\sigma_y^2-\int_0^1\sigma_s^2\,\d s\Big)\stackrel{d}{\rightarrow} \mathcal{N}\Big(0,\pi\Gamma\int_0^1\sigma_s^4\,\d s\Big)\,.\end{align} 
	For a sequence $m=m_n\to\infty$ with $\sqrt{m_n}\Delta_n^{\alpha'-1/2}\to 0$ for some $\alpha'<\alpha$, the estimator \eqref{estm} obeys as $n\rightarrow\infty$ the central limit theorem
  \begin{align}\label{gclteq}
    (m_n\cdot n)^{1/2}\Big(\hat\sigma_{n,m_n}^2-\int_0^1\sigma_s^2\,\d s\Big)\stackrel{d}{\rightarrow}\mathcal{N}\Big(0,\pi\Gamma\int_0^1\sigma_s^4\,\d s\Big)\,.
  \end{align}
	When $\sup_k\lambda_k\E[\langle \xi,e_k\rangle_{\theta}^l]<\infty$ for $l=4,8$, we obtain the normalized version
  \begin{align}\label{gcltf}(m_n\cdot n)^{1/2}(\pi\Gamma\tilde\sigma_{n,m_n}^4)^{-1/2}\Big(\hat\sigma_{n,m_n}^2-\int_0^1\sigma_s^2\,\d s\Big)\stackrel{d}{\rightarrow} \mathcal{N}\big(0,1\big)\,,\end{align}
with the quarticity estimator from \eqref{quart} which satisfies $\tilde\sigma_{n,m_n}^4\stackrel{\P}{\rightarrow}\int_0^1\sigma_s^4\,\d s$.
\end{cor}
Using a spatial average of sums of fourth powers of increments as in \eqref{quart} and plug-in, yields as well a normalized version of \eqref{cltlq}.

\section[Simulations]{Simulations\label{sec:5}}
\begin{figure}[t]
\centering
\caption{Scaled variances of estimators and comparison of curvature estimators.}
\includegraphics[width=6.44cm]{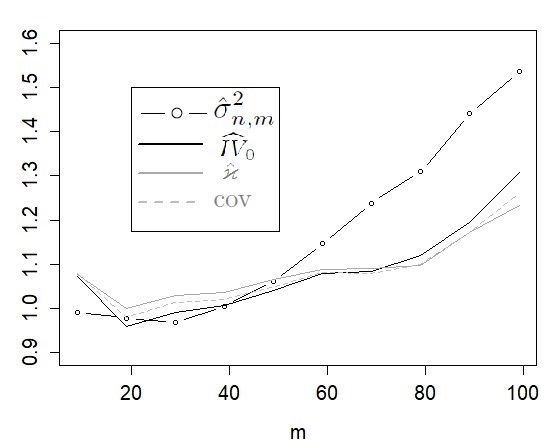}\includegraphics[width=6.716cm]{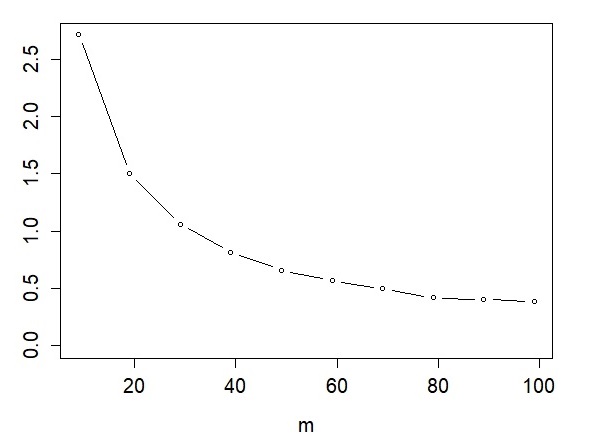}
\begin{quote}{\footnotesize \textit{Notes: Ratios of Monte Carlo and theoretical asymptotic variances and covariance of the estimators \eqref{estm} and \eqref{eq:leastSquares} for $n=1,000$ and $m=i\cdot 10-1,i=1,\ldots,10$ (left). The lines $\hat\sigma^2_{n,m},\widehat{IV}_{\negthinspace 0},\hat{\varkappa}$ and $\operatorname{cov}$ give these ratios for estimator \eqref{estm} and the variances and covariance of the least squares estimator \eqref{eq:leastSquares}. Ratio of variances of estimator $\hat\varkappa$ from \eqref{eq:leastSquares} and $\tilde\varkappa$ from \eqref{kappaest1} (right).}} \end{quote}
\label{Fig:simplot}
\end{figure}
Since our SPDE model leads to Ornstein-Uhlenbeck coordinates with decay-rates $\lambda_k\propto k^2,k\ge 1$ going to infinity, an \emph{exact simulation} of the Ornstein-Uhlenbeck processes turns out to be crucial. In particular, an Euler-Maruyama scheme is not suitable to simulate this model. For some solution process $x(t)=x(0)e^{-\lambda t}+\sigma\int_0^t e^{-\lambda(t-s)} \,\d W_s$ with fixed volatility $\sigma>0$ and decay-rate $\lambda>0$, increments can be written
\begin{align*}x(t+\tau)=x(t)e^{-\lambda \tau}+\sigma\int_t^{t+\tau}e^{-\lambda(t+\tau-s)} \,\d W_s\,,\end{align*}
with variances $\sigma^2(1-\exp(-2\lambda\tau))/(2\lambda)$. Let $\Delta_N$ be a grid mesh with $N=\Delta_N^{-1}\in\N$. Given $x(0)$, the exact simulation iterates
\begin{align*}x(t+\Delta_N)=x(t)e^{-\lambda \Delta_N}+\sigma\sqrt{\frac{1-\exp(-2\lambda\Delta_N)}{2\lambda}} \mathcal{N}_t\,,\end{align*}
at times $t=0,\Delta_N,\ldots, (N-1)\Delta_N$ with i.i.d.\;standard normal random variables $\mathcal{N}_t$. We set $x(0)=0$. We choose $N=n=1,000$ to simulate the discretizations of $K$ independent Ornstein-Uhlenbeck processes with decay-rates given by the eigenvalues $\eqref{eq:eK},k=1,\ldots,K$. The cut-off frequency $K$ needs to be sufficiently large. In the following, we take $K=10,000$. When $K$ is set too small, we witness an additional negative bias in the estimates due to the deviation of the simulated to the theoretical model. For a spatial grid $y_j,j=1,\ldots,m$, $y_j\in[0,1]$, we obtain the simulated observations of the field by
\(X_{i/n}(y_j)=\sum_{k=1}^K x_k(i/n)\,e_k(y_j),\)
using the eigenfunctions from \eqref{eq:eigenf}.

We implement the parametric model \eqref{eq:spde2} with $\kappa=2/5$ and $\sigma=1/4$, or \eqref{eq:spde} with $\theta_0=0,\theta_1=1$ and $\theta_2=1/5$, respectively. The curvature parameter is thus $\varkappa=5$ and $\sigma_0^2=\sqrt{5}/16$. We generate discrete observations equi-spaced in time and space with $y_j=j/(m+1),j=1,\ldots,m$. The least squares estimator \eqref{eq:leastSquares} can be computed using the R function \lstinline$nls$, see \cite{nls}.\\ More precisely, with \lstinline$rv$ denoting the vector $(\sqrt{\Delta_n}\sum_{i=1}^n (\Delta_i X)^2(y_j))_{1\le j\le m}$ and \lstinline$y$ the vector of spatial coordinates, the command\,\footnote{The R-code for these simulations is available at \url{www.github.com/bibinger/SPDEstatistics}}\\[.2cm]
\lstinline$ls<-nls(formula=rv~theta1/sqrt(pi)*exp(-theta2*V2),start=list(theta1 = 1, theta2 = 1),$\\
\hspace*{10cm}\lstinline$data=data.frame(t(rbind(rv,y))))$\\[.2cm]
performs the least squares estimation and \lstinline$coef(ls)[[1]]$ and \lstinline$coef(ls)[[2]]$ give the estimates $\hat\sigma_0^2$ and $\hat\varkappa$, respectively. In our simulations, different start values did not influence the results. 
For configurations with $m\le \sqrt{n}\approx 31.6$, our Monte Carlo results are considerably close to the theoretical asymptotic values.  
In the left-hand plot of Figure \ref{Fig:simplot}, we visualize the ratios of normalized Monte Carlo (co-)variances from 3,000 iterations divided by their theoretical asymptotic counterparts for the volatility estimator \eqref{estm}, which uses the true parameter $\kappa$, as well as for the least squares estimator \eqref{eq:leastSquares}. The normalization is a multiplication with factors $mn$, for fixed $n=1,000$ and for increasing values $m=i\cdot 10-1,i=1,\ldots,10$. To determine the asymptotic covariance matrix in \eqref{cltlq}, we compute means for the different values of $m$ in \eqref{eq:U} and \eqref{eq:V} instead of the fixed integrals in the limit theorem. Values of these ratios close to 1 confirm a good fit by the asymptotic theory. In the cases $m=9,19,29<\sqrt{n}$, all ratios are close to one. Especially for estimator \eqref{estm}, the Monte Carlo variance is very close to the theoretical value. In the cases $m>\sqrt{n}$, when Assumption \ref{Obs} is apparently violated, the ratios slowly increase with the strongest increase for estimator $\hat\sigma^2_{n,m}$. The plot indicates heuristically that when $m>\sqrt{n}$, the rate of the estimator will not be $\sqrt{mn}$. Moreover, it shows how fast the ratio increases with a growing number of spatial observations. The right-hand plot of Figure \ref{Fig:simplot} compares the variances of the two curvature estimators $\hat\varkappa$ from \eqref{eq:leastSquares} and $\tilde\varkappa$ from \eqref{kappaest1}. For $m\le 30$ spatial observations and $n=1,000$, $\tilde\varkappa$ outperforms the least squares estimator, even though the convergence rate $\sqrt{n}$ is slower than the rate $\sqrt{mn}$ of the least squares estimator. 
However, in case of a larger number of spatial coordinates, the improved rate facilitates a more efficient estimation based on the least squares approach.
\begin{figure}[t]
\centering
\caption{Q--Q plots for the feasible central limit theorem \eqref{gcltf}.}
\includegraphics[width=6.0cm]{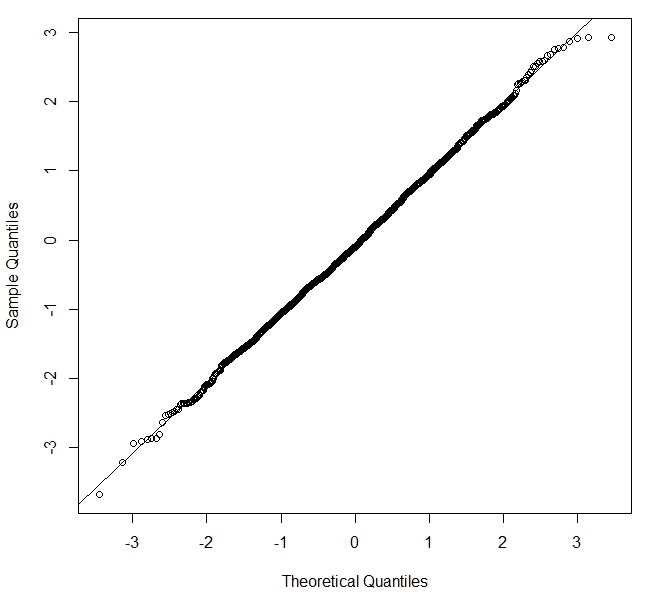}\includegraphics[width=6.0cm]{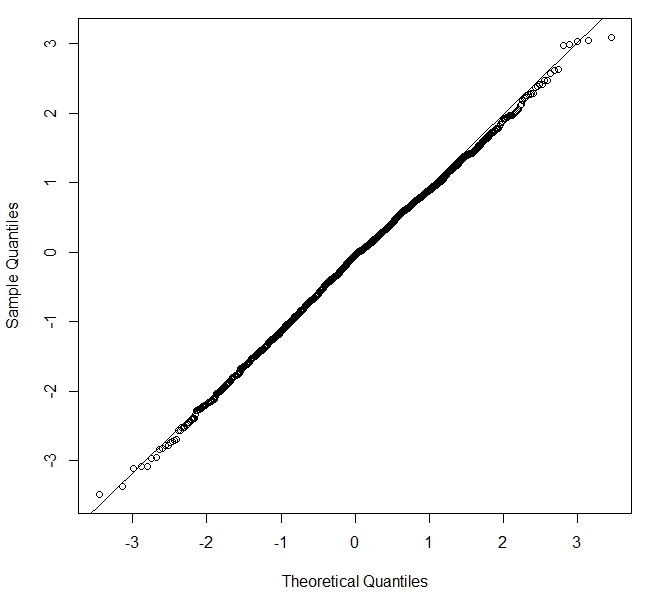}
\begin{quote}{\footnotesize \textit{Notes: Monte Carlo rescaled estimation errors of the statistics in the left-hand side of \eqref{gcltf} compared to a standard normal distribution, for $n=1,000$ and $m=1$ (left) and $m=9$ equi-spaced spatial observations (right).}} \end{quote}
\label{Fig:qq}
\end{figure}
\noindent
Finally, we investigate the semi-parametric estimation of the integrated volatility in the nonparametric model with a time-varying volatility function. We consider \(\sigma_t=1-0.2\sin\big(\tfrac{3}{4}\pi\,t),\,t\in[0,1].\) 
This is a smooth deterministic volatility that mimics the typical intra-day pattern with a strong decrease at the beginning and a moderate increase at the end of the interval. We analyze the finite-sample accuracy of the feasible central limit theorem \eqref{gcltf}. To this end, we rescale the estimation errors of estimator \eqref{estm} with the square root of the estimated integrated quarticity, based on estimator \eqref{quart}, times $\sqrt{\Gamma\pi}$. We compare these statistics rescaled with the rate $\sqrt{mn}$ from 3,000 Monte Carlo iterations to the standard Gaussian limit distribution in the Q--Q plots in Figure \ref{Fig:qq} for $n=1,000$ and $m=1$ and $m=9$, respectively. The plots confirm a high finite-sample accuracy. In particular, the Monte Carlo empirical variances are very close to the ones predicted by the rate and the variance of the limit law as long as $m$ is of moderate size compared to $n$.
\section{Proofs\label{sec:6}}
\subsection{Proof of Proposition~\ref{propExp} and Proposition~\ref{propacf}}
Before we prove the formula for the expected value of the squared increments, we need some auxiliary lemmas.
\begin{lem}\label{lem1}
  On Assumptions \ref{Obs} and \ref{cond}, when $\sigma$ is constant, there is a sequence $(r_{n,i})_{1\le i\le n}$ satisfying $\sum_{i=1}^nr_{n,i}=\mathcal O(\Delta_n^{1/2})$ such that
  \begin{align}
  \E[(\Delta_{i}X)^{2}(y)] & =\sigma^{2}\sum_{k\ge1}\frac{1-e^{-\lambda_{k}\Delta_{n}}}{\lambda_{k}}\Big(1-\frac{1-e^{-\lambda_{k}\Delta_{n}}}{2}e^{-2\lambda_{k}t_{i-1}}\Big)e_{k}^{2}(y)+r_{n,i}\,.\label{eq:Moment}
  \end{align}
 \end{lem}
\begin{proof}
  Recall $X_t(y)=\sum_{k\ge1} x_k(t)e_k(y)$ and \eqref{eq:ABC}.
  It holds that $\E[A_{i,k}B_{i,k}]=\E[A_{i,k}C_{i,k}]=$\\ $\E[B_{i,k}C_{i,k}]=0$, as well as
	\begin{subequations}
  \begin{align}\label{2ma}
    \E[A_{i,k}^2]=\E\big[\langle \xi,e_k\rangle^2_\theta\big] e^{-2\lambda_k(i-1)\Delta_n}(e^{-\lambda_k \Delta_n}-1)^2\,,
  \end{align}
  \begin{align}\notag
    \E[B_{i,k}^2]&=\int_0^{(i-1)\Delta_n}\sigma^2 e^{-2\lambda_k((i-1)\Delta_n-s)}(e^{-\lambda_k \Delta_n}-1)^2\,\d s\\
    &\label{2mb}=\sigma^2 (e^{-\lambda_k \Delta_n}-1)^2\frac{1-e^{-2\lambda_k(i-1)\Delta_n}}{2\lambda_k}\,,
  \end{align}
  \begin{align}\label{2mc}
    \E[C_{i,k}^2]&=\int_{(i-1)\Delta_n}^{i\Delta_n}\sigma^2 e^{-2\lambda_k(i\Delta_n-s)}\,\d s=\frac{1-e^{-2\lambda_k\Delta_n}}{2\lambda_k}\sigma^2\,.
  \end{align}
	\end{subequations}
  Since $(B_{i,k}+C_{i,k})_{k},k\ge1,$ are independent and centered, we obtain 
  \begin{align*}
   \E[(\Delta_{i}X)^{2}(y)]&=\sum_{k,l\ge1}\E[\Delta_{i}x_{k}\Delta_{i}x_{l}]e_{k}(y)e_l(y)\\
  &=  \sigma^{2}\sum_{k\ge1}\Big(\frac{1-e^{-2\lambda_{k}\Delta_{n}}}{2\lambda_{k}}+\big(1-e^{-\lambda_{k}\Delta_{n}}\big)^{2}\big(\frac{1}{2\lambda_{k}}(1-e^{-2\lambda_{k}t_{i-1}})\big)\Big)e_{k}^{2}(y)+ r_{n,i}\\
 & =  \sigma^{2}\sum_{k\ge1}\Big(\frac{1-e^{-2\lambda_{k}\Delta_{n}}+\big(1-e^{-\lambda_{k}\Delta_{n}}\big)^{2}}{2\lambda_{k}}-\frac{\big(e^{-\lambda_{k}\Delta_{n}}-1\big)^{2}}{2\lambda_{k}}e^{-2\lambda_{k}t_{i-1}}\Big)e_{k}^{2}(y)+ r_{n,i}\\
 & =  \sigma^{2}\sum_{k\ge1}\frac{1-e^{-\lambda_{k}\Delta_{n}}}{\lambda_{k}}\Big(1-\frac{1-e^{-\lambda_{k}\Delta_{n}}}{2}e^{-2\lambda_{k}t_{i-1}}\Big)e_{k}^{2}(y)+r_{n,i}\,,
  \end{align*}
  where the remainder $r_{n,i}$ hinges on $\xi$ and $i$:
  \begin{align}
  r_{n,i}  =\sum_{k,l\ge1}\big(1-e^{-\lambda_{k}\Delta_{n}}\big)\big(1-e^{-\lambda_{l}\Delta_{n}}\big)e^{-(\lambda_{k}+\lambda_{l})t_{i-1}}\E[\langle \xi,e_{k}\rangle_{\theta}\langle \xi,e_{l}\rangle_{\theta}]e_k(y)e_l(y)\,.\label{eq:remLem1}
  \end{align}
  If on Assumption~\ref{cond} $\E[\langle \xi,e_{k}\rangle_{\theta}]=0$ and $\E[\langle \xi,e_{k}\rangle_{\theta}^2]\le C/\lambda_k$	with some $C<\infty$, then we have that
	\[ r_{n,i}\le C\sum_{k\ge 1}\frac{\big(1-e^{-\lambda_{k}\Delta_{n}}\big)^2}{\lambda_k}e^{-2\lambda_{k}t_{i-1}}e_k^2(y)\]
	by independence of the coefficients. For the alternative condition from Assumption~\ref{cond}, we use that
  $A_\theta$ is self-adjoint on $H_{\theta}$, such that $-\lambda_k^{1/2}\langle \xi,e_{k}\rangle_{\theta}=\langle A_\theta^{1/2} \xi, e_{k}\rangle_{\theta}$ and thus
  \[
    r_{n,i}=\E\Big[\Big(\sum_{k\ge1}\frac{1-e^{-\lambda_{k}\Delta_{n}}}{\lambda_k^{1/2}}e^{-\lambda_{k}t_{i-1}}e_k(y)\langle A^{1/2}_\theta \xi,e_{k}\rangle_{\theta}\Big)^2\Big].
  \]
  The Cauchy-Schwarz inequality and Parseval's identity yield that
  \begin{align*}  
    r_{n,i}&\le\sum_{k\ge1}\frac{(1-e^{-\lambda_{k}\Delta_{n}})^2}{\lambda_k}e^{-2\lambda_{k}t_{i-1}}e^2_k(y) \E\Big[\sum_{k\ge1}\langle A^{1/2}_\theta \xi,e_{k}\rangle_{\theta}^2\Big]\\
    &\le  2\E\big[\|A^{1/2}_\theta \xi\|_\theta^2\big] \sum_{k\ge1}\frac{(1-e^{-\lambda_{k}\Delta_{n}})^2}{\lambda_k}e^{-2\lambda_{k}(i-1)\Delta_n}\,,
  \end{align*}
  where $\E\big[\|A^{1/2}_\theta \xi\|_\theta^2\big]$ is finite by assumption. We shall often use the elementary inequality
	\begin{align}\label{elin}\frac{(1-p)(1-q)}{1-pq}\le  1-p\,,~~\text{for all}~~0\le p,q<1\,.\end{align}
	Applying this inequality and a Riemann sum approximation yields that
  \begin{align*}
    \sum_{k\ge1}\frac{(1-e^{-\lambda_{k}\Delta_{n}})^2}{\lambda_k}\sum_{i=1}^ne^{-2\lambda_{k}(i-1)\Delta_n}
   &\le\sum_{k\ge1}\frac{(1-e^{-\lambda_{k}\Delta_{n}})^2}{\lambda_k(1-e^{-2\lambda_k\Delta_n})}\\
   \le\Delta_n\sum_{k\ge1}\frac{(1-e^{-\lambda_{k}\Delta_{n}})}{\lambda_k\Delta_n}&\le \Delta_n^{1/2}\int_0^\infty\frac{1-e^{-z^2}}{z^2}\d z+\KLEINO(\Delta_n^{1/2})\,,
  \end{align*}
  such that
  \begin{equation*}
  \sum_{i=1}^nr_{n,i}=\mathcal{O}(\Delta_n^{1/2})\,.\hfill\qedhere
  \end{equation*}
\end{proof}
In order to evaluate the sums in the previous lemma, we apply the following result:
\begin{lem}\label{lem:riemann}
 If $f\colon[0,\infty)\to\R$ is twice continuously differentiable such that $\|f(x^2)\|_{L^1([0,\infty))},\|(1\vee x)f'(x^2)\|_{L^1([0,\infty))},\|(1\vee x^ 2)f''(x^2)\|_{L^1([0,\infty))}\le C$ and $|f(0)|\le C$ for some $C>0$, then:
 \begin{enumerate}
  \item[(i)]  $\displaystyle\sqrt\Delta_n\sum_{k\ge1} f(\lambda_k\Delta_n)=\int_0^\infty f(\pi^2\theta_2z^2)\,\d z-\frac{\sqrt\Delta_n}{2}f(0)+\mathcal O(\Delta_n)$ and
  \item[(ii)]  $\displaystyle\sqrt\Delta_n\sum_{k\ge1} f(\lambda_k\Delta_n)\cos(2\pi ky)=-\frac{\sqrt\Delta_n}{2}f(0)+\mathcal{O}\big((\delta^{-2}\Delta_{n})\wedge(\delta^{-1}\Delta_n^{1/2})\big)$ for any $y\in[\delta,1-\delta]$,
 \end{enumerate}
 with $\delta$ from Assumption \ref{Obs}. The constant in the $\mathcal O$-terms only depends on $C$.
\end{lem}
\begin{proof}
  We perform a substitution with $z^2=k^2\Delta_n$ to approximate the sum by an integral, such that \[\lambda_k \Delta_n=z^2\,\pi^2\theta_2+\Delta_n\Big(\frac{\vartheta_1^2}{4\vartheta_2}-\vartheta_0\Big)\,.\] 
  A Taylor expansion of $f$ and using that $f'(x^2)\in L^1([0,\infty))$ shows that the remainder of setting $\lambda_k\Delta_n\approx \pi^2 z^2 \vartheta_2$ is of order $\Delta_n$. A Riemann midpoint approximation yields for $\tilde f(z):=f(\pi^2\theta_2z^2)$, the grid $a_{k}:=\Delta_{n}^{1/2}(k+1/2),k\ge0,$ and intermediate points $\xi_{k}\in[a_{k-1},a_{k}]$ that
  \begin{align*}
    \sqrt\Delta_n\sum_{k\ge1} f(\lambda_k\Delta_n)-&\int_{\frac{\sqrt{\Delta_{n}}}{2}}^\infty f(\pi^2\theta_2z^2)\,\d z
     = \sum_{k\ge1}\int_{a_{k-1}}^{a_{k}}\Big(\tilde f(\sqrt{\Delta_{n}}k)-\tilde f(z)\Big)\, \d z+\mathcal{O}(\Delta_{n})\\
    & = -\sum_{k\ge1}\int_{a_{k-1}}^{a_{k}}\Big(\big(z-k\sqrt{\Delta_{n}}\big)\tilde f'(\sqrt{\Delta_n}k)+\frac{1}{2}\big(z-k\sqrt{\Delta_{n}}\big)^{2}\tilde f''(\xi_{k})\Big)\, \d z+\mathcal{O}(\Delta_{n})\\
    & = \mathcal{O}\Big(\Delta_{n}\int_{0}^{\infty}|\tilde f''(z)|\, \d z\Big)+\mathcal{O}(\Delta_{n}).
  \end{align*}
  We have used a second-order Taylor expansion where the first-order term vanishes by symmetry.
  Since $\tilde f''\in L^{1}([0,\infty))$, we conclude that
  \begin{align*}
    \sqrt\Delta_n\sum_{k\ge1} f(\lambda_k\Delta_n)&=\int_{0}^{\infty}f(\pi^2\theta_2z^2)\, \d z-\int_{0}^{\frac{\sqrt{\Delta_{n}}}{2}}f(\pi^2\theta_2z^2)\, \d z+\mathcal{O}(\Delta_{n}).    
  \end{align*}
  The subtracted integral over the asymptotically small vicinity close to zero can be approximated by\\ $f(0)\sqrt{\Delta_{n}}/2+\mathcal{O}(\Delta_{n})$ concluding $(i)$. To show $(ii)$, we rewrite with $\Re$ denoting the real part of complex numbers and $\iu$ denoting the imaginary unit
  \begin{align*}
  \sqrt\Delta_n\sum_{k\ge1} f(\lambda_k\Delta_n)\cos(2\pi ky)&=\sqrt{\Delta_n}\sum_{k\ge1}\tilde f(\Delta^{1/2}_nk)\cos(2\pi ky)+\mathcal O(\Delta_n)\\
  &=\Re\Big(\sqrt{\Delta_{n}}\sum_{k\ge1}\tilde f(\Delta^{1/2}_nk)e^{2\pi \mathrm ik y}\Big)+\mathcal O(\Delta_n)\,.
  \end{align*}
  Using $(a_{k})_{k\ge0}$ from above, we derive that 
  \begin{align*}
  \int_{a_{k-1}}^{a_{k}}e^{2\pi \iu uy/\sqrt{\Delta_{n}}}\,\d u & =\frac{\sqrt{\Delta_{n}}}{2\pi \iu y}\big(e^{2\pi \iu a_{k}y/\sqrt{\Delta_{n}}}-e^{2\pi \iu a_{k-1}y/\sqrt{\Delta_{n}}}\big)\\
  & =\frac{\sqrt{\Delta_{n}}}{2\pi \iu y}\big(e^{\pi \iu y}-e^{-\pi \iu y}\big)e^{2\pi \iu ky}
  =\frac{\sin(\pi y)}{\pi y}\sqrt{\Delta_{n}}e^{2\pi \iu ky}.
  \end{align*}
  For a real-valued function $f\in L^1(\R)$, we denote by $\F[f]:\R\to\mathds{C}$ its Fourier transform $\F[f](x)=\int f(t)e^{-\iu x t}\,\d t$. In terms of the Fourier transform $\F$, we obtain
  \begin{align*}
  \sqrt\Delta_n\sum_{k\ge1} f(\lambda_k\Delta_n)\cos(2\pi ky)& =\Re\Big(\frac{\pi y}{\sin(\pi y)}\sum_{k\ge1}\tilde f(\Delta^{1/2}_nk)\int_{a_{k-1}}^{a_{k}}e^{2\pi \iu uy/\sqrt{\Delta_{n}}}\,\d u\Big)\\
  & =\Re\Big(\frac{\pi y}{\sin(\pi y)}\F\Big[\sum_{k\ge1}\tilde f(\Delta^{1/2}_nk)\1_{(a_{k-1},a_{k}]}\Big](-2\pi y\Delta_{n}^{-1/2})\Big)\\
  & =T_1+T_2,
  \end{align*}
\vspace*{-1cm}

  \begin{align*}
    \mbox{where}~~~~T_{1}:= & \Re\Big(\frac{\pi y}{\sin(\pi y)}\F\Big[\sum_{k\ge1}\tilde f(\Delta^{1/2}_nk)\1_{(a_{k-1},a_{k}]}-\tilde f\1_{(a_{0},\infty)}\Big](-2\pi y\Delta_{n}^{-1/2})\Big),\\
  T_{2}:= & \Re\Big(\frac{\pi y}{\sin(\pi y)}\F\big[\tilde f\1_{(a_{0},\infty)}\big](-2\pi y\Delta_{n}^{-1/2})\Big).
  \end{align*}
  Based on that decomposition, we first verify $\sqrt\Delta_n\sum_{k\ge1} f(\lambda_k\Delta_n)\cos(2\pi ky)=\mathcal O(\delta^{-1}\sqrt{\Delta_n})$. To bound $T_1$, a Riemann approximation yields
  \begin{align*}
  |T_1| & \le\frac{\pi y}{\sin(\pi y)}\Big\|\F\Big[\sum_{k\ge1}\tilde f(\Delta_{n}^{1/2}k)\1_{(a_{k-1},a_{k}]}-\tilde f\1_{(a_{0},\infty)}\Big]\Big\|_{\infty}\\
  & \le\frac{\pi y}{\sin(\pi y)}\Big\|\sum_{k\ge1}\tilde f(\Delta_{n}^{1/2}k)\1_{(a_{k-1},a_{k}]}-\tilde f\1_{(a_{0},\infty)}\Big\|_{L^{1}}=\mathcal{O}(\delta^{-1}\Delta^{1/2}_{n}),
  \end{align*}
  where we also have used that $y$ is bounded away from 1 by $\delta$.
  
  For the second term $T_{2}$, we will use the decay of the Fourier
  transformation $\F [\tilde f]$. Since $\tilde f',\tilde f''$\\ $\in L^{1}([0,\infty))$,
  we have $|z\F [\tilde f](z)|=\mathcal{O}(1)$ and $|z^{2}\F [\tilde f](z)|=\mathcal{O}(1)$. Hence, for any $y\ge \delta>0$:
  \[
  |\F [\tilde f](-2\pi y\Delta_{n}^{-1/2})|=\mathcal{O}\Big(\frac{\sqrt{\Delta_n}}{\delta}\wedge\frac{\Delta_n}{\delta^2}\Big)\,.
  \]
  This implies that
  \begin{align}
  T_2 & =-\Re\Big(\frac{\pi y}{\sin(\pi y)}\,\F[\tilde f\1_{(0,a_{0}]}](-2\pi y\Delta_{n}^{-1/2})\Big)+\mathcal{O}\Big(\frac{\sqrt{\Delta_n}}{\delta}\wedge\frac{\Delta_n}{\delta^2}\Big)\notag\\
  & =-\frac{\pi y}{\sin(\pi y)}\int_{0}^{\sqrt{\Delta_{n}}/2}\cos(2\pi yu\Delta_{n}^{-1/2})\tilde f(u)\,\d u+\mathcal{O}\Big(\frac{\sqrt{\Delta_n}}{\delta}\wedge\frac{\Delta_n}{\delta^2}\Big)\notag\\
  & =-\frac{\Delta_{n}^{1/2}\pi y}{\sin(\pi y)}\int_{0}^{1/2}\cos(2\pi yu)\tilde f(\Delta_{n}^{1/2}u)\,\d u+\mathcal{O}\Big(\frac{\sqrt{\Delta_n}}{\delta}\wedge\frac{\Delta_n}{\delta^2}\Big)\,.\label{eq:T2Riemann}
  \end{align}
  In particular, we obtain $|T_2|=\mathcal O(\delta^{-1}\sqrt \Delta)$. Noting that the above estimates for $T_1,T_2$ apply to the imaginary part, too, we conclude
  \begin{equation}\label{eq:auxRiemann}
    \sqrt\Delta_n\sum_{k\ge1} f(\lambda_k\Delta_n)g(2\pi ky)=\mathcal O(\delta^{-1}\sqrt{\Delta_n})\quad\text{for }g=\cos\text{ or }g=\sin.
  \end{equation}
  To prove $(ii)$, we tighten our bounds for $T_1,T_2$.
  For $T_1$ we use a second-order Taylor expansion as for $(i)$:
  \begin{align*}
    |T_1|=& \Big|\Re\Big(\frac{\pi y}{\sin(\pi y)}\sum_{k\ge1}\int_{a_{k-1}}^{a_k}\Big(\tilde f(\Delta^{1/2}_nk)-\tilde f(z)\Big)e^{2\pi\iu zy\Delta_{n}^{-1/2}}\d z\Big)\Big|\\
    =&\frac{\pi y}{\sin(\pi y)}\Big|\Re\Big(-\sum_{k\ge1}\int_{a_{k-1}}^{a_k}\Big((z-\Delta_n^{1/2}k)\tilde f'(\Delta^{1/2}_nk)+\frac12(z-\Delta_n^{1/2}k)^2\tilde f''(\xi_k)\Big)e^{2\pi\iu zy\Delta_{n}^{-1/2}}\d z\Big)\Big|\\
    \le&\frac{\pi y}{\sin(\pi y)}\Big|\sum_{k\ge1}\int_{a_{k-1}}^{a_k}(z-\Delta_n^{1/2}k)\tilde f'(\Delta^{1/2}_nk)\Re(e^{2\pi \iu zy\Delta_n^{-1/2}})\d z\Big|\\
    &\quad +\frac{\pi y}{2\sin(\pi y)}\Big(\sum_{k\ge1}\int_{a_{k-1}}^{a_k}(z-\Delta_n^{1/2}k)^2|\tilde f''(\xi_k)|\d z\Big)\\
    =&\frac{\pi y}{\sin(\pi y)}\Big|\sum_{k\ge1}\tilde f'(\Delta^{1/2}_nk)\int_{-\sqrt{\Delta_n}/2}^{\sqrt{\Delta_n}/2}u\cos\big(2\pi (u+\Delta_n^{1/2}k)y\Delta_n^{-1/2}\big)\d u\Big|+\mathcal{O}\Big(\delta^{-1}\Delta_{n}\int_{0}^{\infty}|\tilde f''(z)|\, \d z\Big)\\
    =&\frac{\pi y}{\sin(\pi y)}\Big|\frac{\Delta_n(\pi y\cos(\pi y)-\sin(\pi y))}{2\pi^2y^2}\sum_{k\ge1}\tilde f'(\Delta^{1/2}_nk)\sin(2\pi ky)\Big|+\mathcal{O}\Big(\delta^{-1}\Delta_{n}\int_{0}^{\infty}|\tilde f''(z)|\, \d z\Big),
  \end{align*}
  where for the last equality, we compute the elementary integral. 
	The sum in the last line is bounded similarly as in \eqref{eq:auxRiemann}. Hence, $|T_1|=\mathcal O(\frac{\sqrt{\Delta_n}}{\delta}\wedge\frac{\Delta_n}{\delta^2})$. For the second term $T_{2}$, we apply integration by parts to \eqref{eq:T2Riemann} to obtain
  \begin{align*}
  T_2 & =-\frac{\Delta_{n}^{1/2}\pi y}{\sin(\pi y)}\Big(\frac{\sin(\pi y)}{2\pi y}\tilde f(\Delta_{n}^{1/2}/2)+\Delta_{n}^{1/2}\int_{0}^{1/2}\frac{\sin(2\pi yu)}{2\pi y}\tilde f'(\Delta_{n}^{1/2}u)\,\d u\Big)+\mathcal{O}\Big(\frac{\sqrt{\Delta_n}}{\delta}\wedge\frac{\Delta_n}{\delta^2}\Big)\\
  & =-\frac12\Delta_{n}^{1/2}\tilde f(\Delta_{n}^{1/2}/2)+\mathcal{O}\Big(\frac{\sqrt{\Delta_n}}{\delta}\wedge\frac{\Delta_n}{\delta^2}\Big)\\
  & =-\frac12\Delta_{n}^{1/2}\tilde f(0)+\mathcal{O}\Big(\frac{\sqrt{\Delta_n}}{\delta}\wedge\frac{\Delta_n}{\delta^2}\Big).
  \end{align*}
  Combining the estimates for $T_1$ and $T_2$ yields $(ii)$.
\end{proof}
\begin{lem}\label{lem2}
  On Assumptions \ref{Obs} and \ref{cond} it holds uniformly for all $y\in[\delta,1-\delta]$, $\delta>0$, that
  \begin{align*}
    \E[(\Delta_i X)^2(y)]=&\sqrt{\Delta_n}\frac{\sigma^2e^{-y\,\vartheta_1/\vartheta_2}}{\pi\sqrt{\vartheta_2}} \int_0^\infty \frac{1-e^{-z^2}}{z^2}\Big(1-\frac{1-e^{-z^2}}{2}e^{-2z^2(i-1)}\Big)\d z
    +r_{n,i}+\mathcal O(\delta^{-2}\Delta_n^{3/2}),
  \end{align*}
  $i\in\{1,\ldots,n\}$, where $r_{n,i}$ is the sequence from Lemma~\ref{lem1}.
\end{lem}
\begin{proof}
  Inserting the eigenfunctions $e_k(y)$ from \eqref{eq:eK} and recalling that $t_{i-1}=(i-1)\Delta_n$, Lemma \ref{lem1} gives for $i\in\{1,\ldots,n\}$ 
    \begin{align*}
      &\E[(\Delta_i X)^2(y)]=2\sigma^{2}\sum_{k\ge1}\frac{1-e^{-\lambda_{k}\Delta_{n}}}{\lambda_{k}}\Big(1-\frac{1-e^{-\lambda_{k}\Delta_{n}}}{2}e^{-2\lambda_{k}(i-1)\Delta_n}\Big)\sin^2(\pi ky)e^{-y\,\vartheta_1/\vartheta_2}+r_{n,i}\,.
  \end{align*}
  Applying the identity $\sin^2(z)=1/2(1-\cos(2z))$, we decompose the sum into three parts, which will be bounded separately:
  \begin{align}
    \E[(\Delta_i X)^2(y)]=\sigma^2\sqrt\Delta_n(I_1-I_2(i)+R(i))e^{-y\,\vartheta_1/\vartheta_2}+r_{n,i}\,,\label{eq:DecompMom}
  \end{align}
	\vspace*{-1cm}
	
  \begin{align*}
    \mbox{where}~~~~I_1&:=\sqrt{\Delta_{n}}\sum_{k\ge1}\frac{1-e^{-\lambda_{k}\Delta_{n}}}{\lambda_{k}\Delta_{n}},~I_2(i):=\sqrt{\Delta_{n}}\sum_{k\ge1}\frac{\big(1-e^{-\lambda_{k}\Delta_{n}}\big)^{2}}{2\lambda_{k}\Delta_{n}}e^{-2\lambda_{k}(i-1)\Delta_n},\\
    R(i)&:=-\sqrt{\Delta_{n}}\sum_{k\ge1}\frac{1-e^{-\lambda_{k}\Delta_{n}}}{\Delta_n\lambda_{k}}\Big(1-\frac{1-e^{-\lambda_{k}\Delta_{n}}}{2}e^{-2\lambda_{k}(i-1)\Delta_n}\Big)\cos(2\pi ky)\,.
  \end{align*}
  The terms $I_1$ and $I_2(i)$ can be approximated by integrals, determining the size of the expected squared increments. The remainder terms $R(i)$ from approximating $\sin^2$ turn out to be negligible.

  For $I_1$ we apply Lemma~\ref{lem:riemann}$(i)$ to the function $f:[0,\infty)\to\R,f(x)\mapsto (1-e^{-x})/x$ for $x>0$ and $f(0)=1$, and obtain
  \begin{align}
    I_1=\int_{0}^{\infty}\frac{(1-e^{-\pi^2\vartheta_2 z^{2}})}{\pi^2\vartheta_2 z^2}\, \d z-\frac{\sqrt{\Delta_{n}}}{{2}}+\mathcal{O}(\Delta_{n})\,.\label{eq:expT1}
  \end{align}
  For the second $i$-dependent term $I_2(i)$ in \eqref{eq:DecompMom}, we apply Lemma~\ref{lem:riemann}$(i)$ to the functions 
  \[
    g_i(x)=\begin{cases}\frac{(1-e^{-x})^2e^{-2x(i-1)}}{2x}\,,&x>0\,,\\ 0\,,&x=0\,,\end{cases}
  \]
  for $1\le i\le n$. We obtain that
  \begin{align}\label{eq:expI2}
  &I_2(i) =\int_{0}^{\infty}g_i(\pi^2\theta_2 z^2)\,\d z+\mathcal O(\Delta_n)\,.
  \end{align}
  Finally, for $R(i)$ we use Lemma~\ref{lem:riemann}$(ii)$ applied to 
  \[
  h_{i}(x)=\frac{1-e^{-x}}{x}\Big(1-\frac{(1-e^{-x})e^{-2x(i-1)}}{2}\Big)\quad\text{for}\quad x>0,
  \]
  satisfying $\lim_{x\downarrow0}h_i(x)=1$. We conclude $R(i)=\frac12\sqrt{\Delta_{n}}+\mathcal{O}(\delta^{-2}\Delta_{n})$.
\end{proof}

With the above lemmas we can deduce Proposition~\ref{propExp}.

\begin{proof}[Proof of Proposition~\ref{propExp}]
  Computing the integrals in Lemma \ref{lem2}
  \begin{align*}\int_0^\infty \frac{1-e^{-z^2}}{z^2} \,\d z=\sqrt{\pi}\,,\,\int_0^\infty \frac{(1-e^{-z^2})^2}{z^2} e^{-2\tau z^2}\,\d z=\sqrt{\pi}\big(2\sqrt{1+2\tau}-\sqrt{2\tau}-\sqrt{2(1+\tau)}\big)
  \end{align*}
  for any $\tau\in\R$, we derive that
  \begin{align*} 
    \E[(\Delta_i X)^2(y)]&=\sqrt{\Delta_n}\,e^{-y\,\vartheta_1/\vartheta_2}\sigma^2\sqrt{\frac{1}{\vartheta_2\pi}}\big(1-\sqrt{1+2(i-1)}+\frac{1}{2}\sqrt{2(i-1)}+\frac{1}{2}\sqrt{2+2(i-1)}\big)\\
    &\quad+r_{n,i}+\mathcal O(\Delta_n^{3/2}) \,.
  \end{align*}
  Since for the $i$-dependent term, by a series expansion we have that
  \begin{align}\label{idep}
    \sum_{i=1}^n\big(\sqrt{1+2(i-1)}-\frac{1}{2}\sqrt{2(i-1)}-\frac{1}{2}\sqrt{2+2(i-1)}\big)=\mathcal O\Big(\sum_{i=1}^{n} i^{-3/2}\Big)\,,
  \end{align}
  where the latter sum converges, we obtain
  \begin{align*}
    &\frac{1}{n}\sum_{i=1}^ne^{-y\,\vartheta_1/\vartheta_2}\sigma^2\sqrt{\frac{1}{\vartheta_2\pi}}\big(1-\sqrt{1+2(i-1)}+\frac{1}{2}\sqrt{2(i-1)}+\frac{1}{2}\sqrt{2+2(i-1)}\big)\\
		&\qquad=e^{-y\,\vartheta_1/\vartheta_2}\sqrt{\frac{1}{\vartheta_2\pi}}\sigma^2 +\mathcal O(\Delta_n).
  \end{align*}
  Since also the remainders from Lemmas \ref{lem1} and \ref{lem2} are bounded uniformly in $i$ with the claimed rate, we obtain the assertion.
\end{proof}

Next, we prove the decay of the covariances.

\begin{proof}[Proof of Proposition~\ref{propacf}]Let us first calculate correlations between the terms $B_{i,k}$ and $C_{i,k}$, which are required frequently throughout the following proofs.
\begin{subequations}
  \begin{align}
  \label{Sbb}\Sigma_{ij}^{B,k}=\cov(B_{i,k},B_{j,k})&=\sigma^2 (e^{-\lambda_k \Delta_n}-1)^2\int_{0}^{(\min(i,j)-1)\Delta_n} e^{-\lambda_k((i+j-2)\Delta_n-2s)}\,\d s\\
  &\notag=\big(e^{-\lambda_k|i-j|\Delta_n}-e^{-\lambda_k(i+j-2)\Delta_n}\big)\frac{(e^{-\lambda_k \Delta_n}-1)^2}{2\lambda_k}\sigma^2\,,\\
	\label{Scc}\Sigma_{ij}^{C,k}=\cov(C_{i,k},C_{j,k})&=\1_{\{i=j\}}\,\sigma^2\,\frac{1-e^{-2\lambda_k\Delta_n}}{2\lambda_k}\,,\\
  \label{Sbc}\Sigma_{ij}^{BC,k}=\cov(C_{i,k},B_{j,k})&=\1_{\{i<j\}}\,\sigma^2 (e^{-\lambda_k \Delta_n}-1)\int_{(i-1)\Delta_n}^{i\Delta_n} e^{-\lambda_k((i+j-1)\Delta_n-2s)}\,\d s\\
  &\notag=\1_{\{i<j\}}\big(e^{-\lambda_k(j-i-1)\Delta_n}-e^{-\lambda_k(j-i+1)\Delta_n}\big)\frac{(e^{-\lambda_k \Delta_n}-1)}{2\lambda_k}\sigma^2\,\\
  &\notag=\1_{\{i<j\}}\,e^{-\lambda_k(j-i)\Delta_n}\big(e^{\lambda_k\Delta_n}-e^{-\lambda_k\Delta_n}\big)\frac{(e^{-\lambda_k \Delta_n}-1)}{2\lambda_k}\sigma^2\,.
  \end{align}
\end{subequations}
The indicator $\1_{\{i<j\}}$ in \eqref{Scc} reflects that $\E[C_{i,k}B_{j,k}]=0$ whenever $i\ge j$. Since the Ornstein-Uhlenbeck processes $(x_k)_{k\ge1}$ are mutually independent, covariances between terms with $k\ne l$ vanish. Note also the independence between $A_{i,k}$ and all $B_{i,k}$ and $C_{i,k}$.\\ Without loss of generality, consider $\cov\big(\Delta_i X(y),\Delta_j X(y)\big)$ for $j>i$. Decompose increments as in \eqref{eq:ABC}:
\begin{align*}&\cov\big(\Delta_i X(y),\Delta_j X(y)\big)=\sum_{k\ge 1}\cov(\Delta_i x_k,\Delta_j x_k)\,e_k^2(y)\\
&~=\sum_{k\ge 1}\cov(A_{i,k}+B_{i,k}+C_{i,k},A_{j,k}+B_{j,k}+C_{j,k})\,e_k^2(y)\\
&~=r_{i,j} +\sum_{k\ge 1}\big(\Sigma_{ij}^{B,k}+\Sigma_{ij}^{BC,k}\big)\,e_k^2(y)
\end{align*}
with $\Sigma_{ij}^{B,k}$ and $\Sigma_{ij}^{BC,k}$ from \eqref{Sbb} and \eqref{Sbc} and with 
\begin{align*}
  r_{i,j}&=\sum_{k\ge 1}\var\big(\langle \xi,e_k\rangle_{\theta}\big)e^{-\lambda_k(i+j-2)\Delta_n}(e^{-\lambda_k\Delta_n}-1)^2\,e_k^2(y)\\
  &\le 2\sup_k\E[\langle A_\theta^{1/2} \xi,e_k\rangle_\theta^2]\sum_{k\ge 1}\frac{(e^{-\lambda_k\Delta_n}-1)^2}{\lambda_k}e^{-\lambda_k(i+j-2)\Delta_n}\,.
\end{align*}
Using analogous steps as in the proof of Proposition \ref{propExp}, we derive that
\begin{align}
&\cov\big(\Delta_i X(y),\Delta_j X(y)\big)
=\sum_{k\ge 1}\big(\Sigma_{ij}^{B,k}+\Sigma_{ij}^{BC,k}\big)\,e_k^2(y)+r_{i,j}\notag\\
&\quad=\sum_{k\ge 1}e^{-\lambda_k(j-i)\Delta_n}\sigma^2\frac{(e^{-\lambda_k\Delta_n}-1)^2+1-e^{\lambda_k\Delta_n}-e^{-2\lambda_k\Delta_n}+e^{-\lambda_k\Delta_n}}{2\lambda_k}e_k^2(y)\notag\\
&\qquad -\underbrace{\sum_{k\ge 1}e^{-\lambda_k(j+i-2)\Delta_n}\sigma^2\frac{(e^{-\lambda_k\Delta_n}-1)^2}{2\lambda_k}e_k^2(y)}_{=:s_{i,j}}+r_{i,j}\notag\\
&~=\sum_{k\ge 1}e^{-\lambda_k(j-i)\Delta_n}\sigma^2\frac{\big(2-e^{-\lambda_k\Delta_n}-e^{\lambda_k\Delta_n}\big)}{2\lambda_k}e_k^2(y)+s_{i,j}+r_{i,j}\notag\\
&~=\sum_{k\ge 1}e^{-\lambda_k(j-i-1)\Delta_n}\sigma^2\frac{\big(2e^{-\lambda_k\Delta_n}-e^{-2\lambda_k\Delta_n}-1\big)}{2\lambda_k}e_k^2(y)+s_{i,j}+r_{i,j}\notag\\
&~=-\Delta_n^{1/2}\sqrt{\Delta_n}\sum_{k\ge 1}e^{-\lambda_k(j-i-1)\Delta_n}\sigma^2\frac{\big(e^{-\lambda_k\Delta_n}-1\big)^2}{2\lambda_k\Delta_n}e_k^2(y)+s_{i,j}+r_{i,j}.\label{eq:cov}
\end{align}
Using the same approximation based on the equality $\sin^2(z)=\frac{1}{2}-\cos(2z)$ as in Lemma \ref{lem2} combined with the Riemann sum approximations from Lemma~\ref{lem:riemann}, the previous line equals
\begin{align*}
&-\Delta_n^{1/2}\exp(-y\,\vartheta_1/\vartheta_2)\frac{\sigma^2}{2\pi\sqrt{\theta_2}}\int_0^{\infty} \frac{\big(1-e^{-z^2}\big)^2}{z^2} \,e^{-z^2(j-i-1)}\,\d z\,+s_{i,j}+r_{i,j}+\mathcal O(\Delta_n^{3/2})\\
=&-\Delta_n^{1/2}\exp(-y\,\vartheta_1/\vartheta_2)\frac{\sigma^2}{2\sqrt{\pi\theta_2}}\big(2\sqrt{j-i}-\sqrt{j-i-1}-\sqrt{j-i+1}\big)\,+s_{i,j}+r_{i,j}+\mathcal O(\Delta_n^{3/2})\,.\notag
\end{align*}

The claim is implied by the estimate 
\begin{align*}
  \sum_{i,j=1}^n(s_{i,j}+r_{i,j})&\le \big(\sigma^2+2\sup_k\E[\langle A_\theta^{1/2} \xi,e_k\rangle_\theta^2]\big)\sum_{i,j=1}^n\sum_{k\ge 1}e^{-\lambda_k(j+i-2)\Delta_n}\frac{(e^{-\lambda_k\Delta_n}-1)^2}{\lambda_k}\\
  &=\big(\sigma^2+2\sup_k\E[\langle A_\theta^{1/2} \xi,e_k\rangle_\theta^2]\big)\sum_{k\ge 1}\frac{(e^{-\lambda_k\Delta_n}-1)^2}{\lambda_k}\Big(\sum_{i=0}^{n-1}e^{-\lambda_ki\Delta_n}\Big)^2\\
  &\le\big(\sigma^2+2\sup_k\E[\langle A_\theta^{1/2} \xi,e_k\rangle_\theta^2]\big)\sum_{k\ge 1}\frac{(e^{-\lambda_k\Delta_n}-1)^2}{\lambda_k}(1-e^{-\lambda_k\Delta_n})^{-2}=\mathcal{O}(1).\,\,\qedhere
\end{align*}
\end{proof}
\subsection{Proofs of the central limit theorems}
We establish the central limit theorems based on a general central limit theorem for $\rho$-mixing triangular arrays by \citet[Thm.\;4.1]{utev}.
More precisely, we use a version of Utev's theorem where the mixing condition is replaced by explicit assumptions on the partial sums that has been reported by \citet[Theorem B]{peligradUtev1997}. Set
\begin{align}\label{zetaM}
  \zeta_{n,i}:=\frac{\sqrt{\theta_2\pi}}{\sqrt m}\sum_{j=1}^{m}(\Delta_i X)^2(y_j)\exp(y_j\,\theta_1/\theta_2)\,,
\end{align}
where we consider either $m=1$ or $m=m_n$ for a sequence satisfying $m_n=\mathcal O(n^{\rho})$ for some $\rho\in(0,1/2)$ according to Assumption \ref{Obs}. 

If $\xi$ is distributed according to the stationary distribution, i.e.\ the independent coefficients satisfy $\langle\xi,e_{k}\rangle_\theta\sim\mathcal{N}(0,\sigma^{2}/(2\lambda_{k}))$, we have the representation $\Delta_{i}\tilde X(y)=\sum_{k\ge1}\Delta_{i}\tilde x_{k}e_{k}(y)$ with
\begin{equation}\label{statIncrement}
\Delta_{i}\tilde x_{k}=C_{i,k}+\tilde{B}_{i,k},\quad\tilde{B}_{i,k}=\int_{-\infty}^{(i-1)\Delta_{n}}\sigma(e^{-\lambda_{k}\Delta_{n}}-1)e^{-\lambda_{k}((i-1)\Delta_{n}-s)}\d W_{s}^k\,,
\end{equation}
extending $(W^k)_{k\ge 1}$ to Brownian motions on the whole real line for each $k$ defined on a suitable extension of the probability space. Note that $\Delta_{i}\tilde X(y)$ is centered, Gaussian and stationary.

As the following lemma verifies, under Assumption~\ref{cond} it is sufficient to establish the central limit theorems for the stationary sequence
\begin{align}\label{zetatilde}
   \widetilde \zeta_{n,i}=\frac{\sqrt{\theta_2\pi}}{\sqrt m}\sum_{j=1}^{m}(\Delta_i \widetilde X)^2(y_j)\exp(y_j\,\theta_1/\theta_2)\,.
\end{align}

\begin{lem}\label{lem:neglectA}
  On Assumptions \ref{Obs} and \ref{cond}, we have
  \[
  \sum_{i=1}^n(\widetilde \zeta_{n,i} -\zeta_{n,i} )\stackrel{\P}{\rightarrow} 0\,.
\]
\end{lem}
\begin{proof}
  It suffices to show that we have uniformly in $y$
  \begin{align}\label{eq:neglectA}
     \sqrt{m}\,\sum_{i=1}^n\big((\Delta_i \widetilde X)^2(y) -(\Delta_i X)^2(y)\big)\stackrel{\P}{\rightarrow}0\,.
  \end{align}
  By the decomposition \eqref{eq:ABC}, where we write $\widetilde A_{i,k}$ in the stationary case such that $\widetilde A_{i,k}+B_{i,k}=\widetilde B_{i,k}$, we obtain that
  \begin{align*}
     (\Delta_i \tilde X)^2(y) -(\Delta_i X)^2(y)
     =\sum_{k,l}\big(\Delta_i \widetilde x_k\Delta_i \widetilde x_l-\Delta_i x_k\Delta_i x_l\big)e_k(y)e_l(y)=\tilde T_i- T_i
  \end{align*}
  with
    \(T_i:=\sum_{k,l}\big(A_{i,k} A_{i,l}+A_{i,k}(B_{i,l}+C_{i,l})+A_{i,l}(B_{i,k}+C_{i,k})\big)e_k(y)e_l(y)\), 
  and an analogous definition of $\tilde T_i$ where $A_{i,k}$ and $A_{i,l}$ are replaced by $\tilde A_{i,k}$ and $\tilde A_{i,l}$, respectively. We show that $\sqrt{m}\sum_{i=1}^n T_i\stackrel{\P}{\rightarrow}0$ under Assumption \ref{cond}. Since this assumption is especially satisfied under the stationary initial distribution, we then conclude $\sqrt{m}\sum_{i=1}^n \tilde T_i\stackrel{\P}{\rightarrow}0$ implying \eqref{eq:neglectA}. To show that $\sqrt{m}\sum_{i=1}^n T_i\stackrel{\P}{\rightarrow}0$, we further separate
  \begin{align}\label{eq:neglectA2}
    \sum_{i=1}^n T_i=\sum_{i=1}^n\Big(\sum_{k\ge 1} A_{i,k}e_k(y)\Big)^2+2\sum_{i=1}^n\Big(\sum_{k\ge1} A_{i,k}e_k(y)\Big)\Big(\sum_{k\ge1}(B_{i,k}+C_{i,k})e_k(y)\Big).
  \end{align}
  We have that
  \begin{align*}
    &\E\Big[\sum_{i=1}^n\Big(\sum_{k\ge 1} A_{i,k}e_k(y)\Big)^2\Big]
    \le2\E\Big[\sum_{i=1}^n\sum_{k\ge1} A_{i,k}^2\Big]+\E\Big[\Big|\sum_{i=1}^n\sum_{k\neq l} A_{i,k}A_{i,l}e_k(y)e_l(y)\Big)\Big|^2\Big]^{1/2}
  \end{align*}
Denoting $C_\xi:=\sup_k\lambda_k\E[\langle \xi,e_k\rangle_\theta^2]$, with \eqref{elin}, we obtain for the first term:
\begin{align*}
  \sum_{i=1}^n\sum_{k\ge1} \E\big[A_{i,k}^2\big]
  &=\sum_{i=1}^n\sum_{k\ge 1}(1-e^{-\lambda_k\Delta_n})^2e^{-2\lambda_k(i-1)\Delta_n}\E\big[\langle\xi,e_k\rangle_\theta^2\big]\\
  &\le C_\xi\sum_{k\ge1}\frac{(1-e^{-\lambda_k\Delta_n})^2}{\lambda_k}\sum_{i=1}^ne^{-2\lambda_k(i-1)\Delta_n}\\
  &\le C_\xi\sum_{k\ge1}\frac{(1-e^{-\lambda_k\Delta_n})^2}{\lambda_k(1-e^{-2\lambda_k\Delta_n})}
  \le C_\xi\sum_{k\ge1}\frac{(1-e^{-\lambda_k\Delta_n})}{\lambda_k}=\mathcal O(\Delta_n^{1/2}).
\end{align*}
For the second term we have to distinguish between the case $\E[\langle \xi,e_k\rangle_\theta]=0$ in Assumption \ref{cond} $(i)$, in that we can bound
\begin{align*}
  &\sum_{i,j=1}^n\sum_{k\neq l}\E\big[A_{i,k}A_{i,l}A_{j,k}A_{j,l}\big]\\
  &\quad=\sum_{i,j=1}^n\sum_{k\neq l}(1-e^{-\lambda_l\Delta_n})^2(1-e^{-\lambda_k\Delta_n})^2e^{-2(\lambda_k+\lambda_l)(i+j-2)\Delta_n}\E\big[\langle\xi,e_k\rangle_\theta^2\big]\E\big[\langle\xi,e_l\rangle_\theta^2\big]\\
  &\quad\le C_\xi^2\sum_{k,l\ge1}\frac{(1-e^{-\lambda_k\Delta_n})^2(1-e^{-\lambda_l\Delta_n})^2}{\lambda_k\lambda_l(1-e^{-(\lambda_k+\lambda_l)\Delta_n})^2}\le C_\xi^2\sum_{k,l\ge1}\frac{(1-e^{-\lambda_k\Delta_n})(1-e^{-\lambda_l\Delta_n})}{\lambda_k\lambda_l}
  =\mathcal O(\Delta_n),
\end{align*}
using Assumption \ref{cond} $(ii)$, the geometric series and inequality \eqref{elin}, and the second case in Assumption \ref{cond} $(i)$ where by Parseval's identity $C_\xi':=\sum_k\lambda_k\E[\langle \xi,e_k\rangle_\theta^2]<\infty$. Under this condition we have an upper bound
\begin{align*}
 &\sum_{i,j=1}^n\sum_{k\neq l,k'\neq l'}\E\big[A_{i,k}A_{i,l}A_{j,k'}A_{j,l'}\big]\\
 &\quad\le\Big(\sum_{i=1}^n\Big(\sum_{k\ge 1}(1-e^{-\lambda_k\Delta_n})e^{-\lambda_k(i-1)\Delta_n}\E\big[\langle\xi,e_k\rangle_\theta^2\big]^{1/2}\Big)^2\Big)^2\\
 &\quad\le\Big(\sum_{i=1}^n\Big(\sum_{k\ge 1}\frac{(1-e^{-\lambda_k\Delta_n})^2e^{-2\lambda_k(i-1)\Delta_n}}{\lambda_k}\Big)\Big(\sum_{k\ge 1} \lambda_k\E\big[\langle\xi,e_k\rangle_\theta^2\big]\Big)\Big)^2\\
 &\quad\le C_\xi'^2\Big(\sum_{k\ge 1}\frac{(1-e^{-\lambda_k\Delta_n})^2}{\lambda_k(1-e^{-2\lambda_k\Delta_n})}\Big)^2=\mathcal O(\Delta_n).
\end{align*}
The last term is of the same structure as the one in the last step of Lemma \ref{lem1} from where we readily obtain that it is $\mathcal O(\Delta_n)$. Hence, Markov's inequality yields that $\sum_{i=1}^n(\sum_{k} A_{i,k}e_k(y))^2=\mathcal O_{\P}(\Delta_n^{1/2})$. To bound the second term in \eqref{eq:neglectA2}, we use independence of the terms $A_{j,k}$ and $(B_{i,k}+C_{i,k})$ to obtain
\begin{align*}
  &\E\Big[\Big|\sum_{i=1}^n\Big(\sum_{k\ge1} A_{i,k}e_k(y)\Big)\Big(\sum_{k\ge1}(B_{i,k}+C_{i,k})e_k(y)\Big)\Big|^2\Big]\\
  &\quad=\sum_{i,j=1}^n\Big(\sum_{k\ge1}\E[A_{i,k}A_{j,k}]e_k^2(y)\Big)\Big(\sum_{k\ge1}\E\big[(B_{i,k}+C_{i,k})(B_{j,k}+C_{j,k})\big]e_k^2(y)\Big)
  =:\sum_{i,j=1}^n R_{i,j}S_{i,j}\,.
\end{align*}
The first factor is bounded by
\[
  R_{i,j}=\sum_{k\ge1}\E[A_{i,k}A_{j,k}]e_k^2(y)
  \le2C_\xi\sum_{k\ge1}\frac{(1-e^{-\lambda_k\Delta_n})^2}{\lambda_k}e^{-\lambda_k(i+j-2)\Delta_n}\,,
\]
and satisfies $\sup_{i,j} |R_{i,j}|=\mathcal O(\Delta_n^{1/2})$, as well as $\sup_j\sum_i|R_{i,j}|=\mathcal O(\Delta_n^{1/2})$.
The off-diagonal elements with $i<j$ of the second factor $S_{i,j}$ have been calculated in the proof of Proposition~\ref{propacf}:
\begin{align*}
  &S_{i,j}=\sum_{k\ge1}\big(\Sigma_{ij}^{B,k}+\Sigma_{ij}^{BC,k}+\Sigma_{ji}^{BC,k}+\Sigma_{ij}^{C,k}\big)e_k^2(y)
  =\sum_{k\ge1}\big(\Sigma_{ij}^{B,k}+\Sigma_{ij}^{BC,k}\big)e_k^2(y)\\
  &\hspace*{-0.1cm}\le\hspace*{-0.05cm}\sqrt{\Delta_n}\frac{\sigma^2e^{-\frac{y\theta_1}{\theta_2}}}{2\sqrt{\pi\theta_2}}\big(2\sqrt{j\hspace*{-0.025cm}-\hspace*{-0.025cm}i}\hspace*{-0.05cm}-\hspace*{-0.05cm}\sqrt{j\hspace*{-0.025cm}-\hspace*{-0.025cm}i\hspace*{-0.025cm}-\hspace*{-0.025cm}1}\hspace*{-0.05cm}-\hspace*{-0.05cm}\sqrt{j\hspace*{-0.025cm}-\hspace*{-0.025cm}i\hspace*{-0.025cm}+\hspace*{-0.025cm}1}\big)\hspace*{-0.05cm}+\hspace*{-0.05cm}\sum_{k\ge1}\frac{(1-e^{-\lambda_k\Delta_n})^2}{\lambda_k}e^{-\lambda_k(i+j-2)\Delta_n}\hspace*{-0.05cm}+\hspace*{-0.05cm}\mathcal O(\Delta_n^{3/2})
\end{align*}
for $i<j$, while we obtain for $i=j$ that
\begin{align*}
  S_{i,i}=\sum_{k\ge1}\big(\Sigma_{ii}^{B,k}+\Sigma_{ii}^{C,k}\big)e_k^2(y)
  &\le 2\sigma^2\sum_{k\ge1}\Big(\frac{(1-e^{-\lambda_k\Delta_n})^2}{2\lambda_k}+\frac{1-e^{-2\lambda_k\Delta_n}}{2\lambda_k}\Big)= 2\sigma^2\sum_{k\ge1}\frac{1-e^{-\lambda_k\Delta_n}}{\lambda_k}\,.
\end{align*}
Therefore, the second term in \eqref{eq:neglectA2} has a second moment bounded by a constant times
\begin{align*}
  &\sum_{i, j=1}^n\Big(\sum_{k\ge1}\frac{(1-e^{-\lambda_k\Delta_n})^2}{\lambda_k}e^{-\lambda_k(i+j-2)\Delta_n}\Big)\Big(\Delta_n^{1/2}\1_{\{i\neq j\}}|j-i|^{-3/2}\\
  &\qquad+\sum_{k\ge1}\Big(\1_{\{i=j\}}\frac{1-e^{-\lambda_k\Delta_n}}{\lambda_k}+\frac{(1-e^{-\lambda_k\Delta_n})^2}{\lambda_k}e^{-\lambda_k(i+j-2)\Delta_n}\Big)+ \mathcal O(\Delta_n^{3/2})\Big)\\
  &\quad=\mathcal O\Big(\Delta_n\sum_{j\ge1} j^{-3/2}+\Delta_n\Big)=\mathcal O(\Delta_n).
\end{align*}
We conclude that both terms in \eqref{eq:neglectA2} are of the order $\mathcal O_{\P}(\Delta_n^{1/2})$ and thus that $\sqrt{m}\sum_{i=1}^n T_i\stackrel{\P}{\rightarrow}0$ under Assumption \ref{Obs}. This implies \eqref{eq:neglectA}.
\end{proof}

Thanks to the previous lemma we can throughout investigate the solution process under the stationary initial condition.
To prepare the central limit theorems, we compute the asymptotic variances in the next step. According to \eqref{estm1}, we define the rescaled realized volatility at $y\in(0,1)$ based on the first $p\le n$ increments as
\begin{align}\label{rvr}
  V_{p,\Delta_n}(y):=\frac{1}{p\sqrt{\Delta_n}}\sum_{i=1}^p(\Delta_i \tilde X)^2(y)e^{y\,\vartheta_1/\vartheta_2}.
\end{align}
\begin{prop}\label{varprop}
  On Assumptions \ref{Obs} and \ref{cond}, the covariance of the rescaled realized volatility for two spatial points $y_1,y_2\in[\delta,1-\delta]$ satisfies for any $\eta\in(0,1)$
  \begin{align}
    \cov\big(V_{p,\Delta_n}(y_1),V_{p,\Delta_n}(y_2)\big)&=\1_{\{y_1=y_2\}} p^{-1}\Gamma\sigma^4\vartheta_2^{-1}\big(1+\mathcal O(1\wedge(p^{-1}\Delta_n^{\eta-1}))\big)\label{varpropeq}\\
    &\qquad+ \mathcal O\Big(\frac{\Delta_n^{1/2}}{p}(\1_{\{y_1\neq y_2\}}|y_1-y_2|^{-1}+ \delta^{-1})\Big)\notag
  \end{align}
  where $\Gamma>0$ is a numerical constant given in \eqref{eq:constantVariance} with $\Gamma \approx 0.75$. In particular, we have $\var(V_{n,\Delta_n}(y))=n^{-1}\Gamma\sigma^4\vartheta_2^{-1}(1+\mathcal O(\sqrt{\Delta_n}))$.
\end{prop}

\begin{proof}
Since
\[(\Delta_i \tilde X)^2(y)=\Big(\sum_{k\ge 1}\Delta_i \tilde x_k e_k(y)\Big)^2=\sum_{k, l}\Delta_i \tilde x_k\Delta_i \tilde x_l e_k(y)e_l(y),\]
we obtain the following variance-covariance structure of the rescaled realized volatilities \eqref{rvr} in points $y_1,y_2\in[\delta,1-\delta]$:
\begin{align*}
  &\cov\big(V_{p,\Delta_n}(y_1),V_{p,\Delta_n}(y_2)\big)\notag\\
  &\phantom{D_{k,l}}=\frac{e^{(y_1+y_2)\vartheta_1/\vartheta_2}}{\Delta_np^2}\sum_{i,j=1}^p\bigg(\sum_{k,l\ge1}\,e_k(y_1)e_k(y_2)e_l(y_1)e_l(y_2)\,\cov(\Delta_i \tilde x_k\Delta_i \tilde x_l,\Delta_j \tilde x_k\Delta_j \tilde x_l)\bigg)\notag\\
  &\phantom{D_{k,l}}=\frac{e^{(y_1+y_2)\vartheta_1/\vartheta_2}}{\Delta_np}\sum_{k,l\ge1}\,e_k(y_1)e_k(y_2)e_l(y_1)e_l(y_2)D_{k,l}\quad\text{with}\notag\\
  &D_{k,l}=\frac1p\sum_{i,j=1}^p\cov\big((\tilde B_{i,k}+C_{i,k})(\tilde B_{i,l}+C_{i,l}),(\tilde B_{j,k}+C_{j,k})(\tilde B_{j,l}+C_{j,l})\big)\,,
\end{align*}
while other covariances vanish by orthogonality. We will thus need the covariances of the terms in the decomposition from \eqref{statIncrement} corresponding to \eqref{Sbb} to \eqref{Sbc}. Since 
\[
  \widetilde B_{i,k}=B_{i,k}+\sigma\frac{(e^{-\lambda\Delta_n}-1)}{\sqrt{2\lambda_k}}\,e^{\lambda_k(i-1)\Delta_n}Z_k\,
\]
for a sequence of independent standard normal random variables $(Z_k)_{k\ge1}$, we find that
\begin{align}
  \cov(\widetilde B_{j,k},C_{i,k})&=\cov(B_{j,k}, C_{i,k})=\Sigma_{ij}^{BC,k},\notag\\
  \cov(\widetilde B_{i,k},\widetilde B_{j,k})&=(2\lambda_k)^{-1} \,\sigma^2(e^{-\lambda_k\Delta_n}-1)^2e^{-\lambda_k|i-j|\Delta_n}=:\widetilde\Sigma_{ij}^{B,k}.\label{eq:covTilde}
\end{align}
We deduce from Isserlis' theorem
\begin{align}\label{eq:CovV}
D_{k,l}=\frac2p\sum_{i,j=1}^p\big(\tilde\Sigma_{ij}^{B,k}+\Sigma_{ij}^{BC,k}+\Sigma_{ji}^{BC,k}+\Sigma_{ij}^{C,k}\big)\big(\tilde\Sigma_{ij}^{B,l}+\Sigma_{ij}^{BC,l}+\Sigma_{ji}^{BC,l}+\Sigma_{ij}^{C,l}\big).
\end{align}
In order to calculate $D_{k,l}$, we first study the case $k<l$ and consider the different addends consecutively.
\begin{subequations}
Using \eqref{eq:covTilde}, we derive that
\begin{align}\label{h1}\frac1p\sum_{i,j=1}^p\tilde\Sigma_{ij}^{B,k}\tilde\Sigma_{ij}^{B,l}&=\sigma^4\frac{(e^{-\lambda_k\Delta_n}-1)^2(e^{-\lambda_l\Delta_n}-1)^2}{4\lambda_k\lambda_l}\frac{1+e^{-(\lambda_k+\lambda_l)\Delta_n}}{1-e^{-(\lambda_k+\lambda_l)\Delta_n}}\\
&\quad\times \notag\Big(1+\mathcal O\Big(\frac{p^{-1}}{1-e^{-(\lambda_k+\lambda_l)\Delta_n}}\Big)\Big)\,
\end{align}
exploiting the well-known formula for the geometric series, that we have for any $q>0$
	\[\sum_{i,j=1}^pq^{|i-j|}
    =2\sum_{i=1}^p\sum_{j=1}^{i-1}q^{i-j}+p
    =p\frac{1+q}{1-q}+2\,\frac{q^{p+1}-q}{(q-1)^2},\]
which gives for $q=\exp(-(\lambda_k+\lambda_l)\Delta_n)$ that
\begin{align*}
  \frac1p\sum_{i,j=1}^pe^{-(\lambda_k+\lambda_l)|i-j|\Delta_n}=\frac{1+e^{-(\lambda_k+\lambda_l)\Delta_n}}{1-e^{-(\lambda_k+\lambda_l)\Delta_n}}\Big(1+\mathcal O\Big( 1\wedge \frac{p^{-1}}{1-e^{-(\lambda_k+\lambda_l)\Delta_n}}\Big)\Big)\,.
\end{align*}
Using \eqref{Scc}, we obtain that
\begin{align} \label{h2}
  \frac 1p \sum_{i,j=1}^p\Sigma_{ij}^{C,k}\Sigma_{ij}^{C,l}=\sigma^4\frac{(1-e^{-2\lambda_k\Delta_n})(1-e^{-2\lambda_l\Delta_n})}{4\lambda_k\lambda_l}\,.
\end{align}
With \eqref{Sbc}, we derive that
\begin{align}\label{h3}
  \frac 1p\sum_{i,j=1}^p\Sigma_{ij}^{BC,k}\Sigma_{ij}^{BC,l}&=\sigma^4\frac{(e^{-\lambda_k\Delta_n}-1)(e^{-\lambda_l\Delta_n}-1)}{4\lambda_k\lambda_l}\frac{(1-e^{-2\lambda_k\Delta_n})(1-e^{-2\lambda_l\Delta_n})}{1-e^{-(\lambda_k+\lambda_l)\Delta_n}}\,\\
  &\qquad\times \Big(1+\mathcal O\Big(1\wedge \frac{p^{-1}}{1-e^{-(\lambda_k+\lambda_l)\Delta_n}}\Big)\Big)\notag
\end{align}
and analogously for $\Sigma_{ji}^{BC,k}\Sigma_{ji}^{BC,l}$, using the auxiliary calculation:
\begin{align*}
  &\frac 1p\sum_{i=1}^p\sum_{j=i+1}^p(e^{\lambda_k\Delta_n}-e^{-\lambda_k\Delta_n})(e^{\lambda_l\Delta_n}-e^{-\lambda_l\Delta_n})e^{-(\lambda_k+\lambda_l)(j-i)\Delta_n}\\
  &=(e^{\lambda_k\Delta_n}-e^{-\lambda_k\Delta_n})(e^{\lambda_l\Delta_n}-e^{-\lambda_l\Delta_n})p^{-1}\sum_{i=1}^p\frac{e^{-(\lambda_k+\lambda_l)\Delta_n}-e^{-(\lambda_k+\lambda_l)(p-i+1)\Delta_n}}{1-e^{-(\lambda_k+\lambda_l)\Delta_n}}\\
  &=(e^{\lambda_k\Delta_n}-e^{-\lambda_k\Delta_n})(e^{\lambda_l\Delta_n}-e^{-\lambda_l\Delta_n})\frac{e^{-(\lambda_k+\lambda_l)\Delta_n}}{1-e^{-(\lambda_k+\lambda_l)\Delta_n}}\Big(1+\mathcal O\Big(1\wedge \frac{p^{-1}}{1-e^{-(\lambda_k+\lambda_l)\Delta_n}}\Big)\Big)\\
  &=\frac{(1-e^{-2\lambda_k\Delta_n})(1-e^{-2\lambda_l\Delta_n})}{1-e^{-(\lambda_k+\lambda_l)\Delta_n}}\Big(1+\mathcal O\Big(1\wedge \frac{p^{-1}}{1-e^{-(\lambda_k+\lambda_l)\Delta_n}}\Big)\Big)\,.
  \end{align*}
For the mixed terms, we infer that
\begin{align}\label{h4}
\frac 1p\sum_{i,j=1}^p\tilde \Sigma_{ij}^{B,k}\big(\Sigma_{ij}^{BC,l}+\Sigma_{ji}^{BC,l}\big)&=\sigma^4\frac{(e^{-\lambda_k\Delta_n}-1)^2(e^{-\lambda_l\Delta_n}-1)}{2\lambda_k\lambda_l}e^{-\lambda_k\Delta_n}\frac{1-e^{-2\lambda_l\Delta_n}}{1-e^{-(\lambda_k+\lambda_l)\Delta_n}}\\
&\qquad \times \Big(1+\mathcal O\Big(1\wedge \frac{p^{-1}}{1-e^{-(\lambda_k+\lambda_l)\Delta_n}}\Big)\Big),\notag\\
  \label{h5}\frac 1p\sum_{i,j=1}^p\tilde\Sigma_{ij}^{B,k}\Sigma_{ij}^{C,l}&=\sigma^4\frac{(e^{-\lambda_k\Delta_n}-1)^2(1-e^{-2\lambda_l\Delta_n})}{4\lambda_k\lambda_l}\,,\\
  \label{hl}\frac 1p\sum_{i,j=1}^p\Sigma_{ij}^{C,k}\Sigma_{ij}^{BC,l}&=\frac 1p\sum_{i,j=1}^p\Sigma_{ij}^{BC,k}\Sigma_{ij}^{C,l}
  =\frac 1p\sum_{i,j=1}^p\Sigma_{ij}^{BC,k}\Sigma_{ji}^{BC,l}=0\,,
\end{align}
\end{subequations}
based on \eqref{Sbb}-\eqref{Sbc} and similar ingredients as above. With \eqref{h1}-\eqref{hl}, we obtain that $D_{k,l},k<l,$ from \eqref{eq:CovV} is given by $\sigma^4(1+\mathcal O(1\wedge\frac{p^{-1}}{1-e^{-(\lambda_k+\lambda_l)\Delta_n}}))$ times
\begin{align*}
  &2\bigg(\frac{(e^{-\lambda_k\Delta_n}-1)^2(e^{-\lambda_l\Delta_n}-1)^2}{4\lambda_k\lambda_l}\frac{1+e^{-(\lambda_k+\lambda_l)\Delta_n}}{1-e^{-(\lambda_k+\lambda_l)\Delta_n}}+\frac{(1-e^{-2\lambda_k\Delta_n})(1-e^{-2\lambda_l\Delta_n})}{4\lambda_k\lambda_l}\\
  &\quad +\frac{(e^{-\lambda_k\Delta_n}-1)(e^{-\lambda_l\Delta_n}-1)}{2\lambda_k\lambda_l}\frac{(1-e^{-2\lambda_k\Delta_n})(1-e^{-2\lambda_l\Delta_n})}{1-e^{-(\lambda_k+\lambda_l)\Delta_n}}\\
  &\quad +\frac{(e^{-\lambda_k\Delta_n}-1)^2(e^{-\lambda_l\Delta_n}-1)}{2\lambda_k\lambda_l}e^{-\lambda_k\Delta_n}\frac{1-e^{-2\lambda_l\Delta_n}}{1-e^{-(\lambda_k+\lambda_l)\Delta_n}}+\frac{(e^{-\lambda_k\Delta_n}-1)^2(1-e^{-2\lambda_l\Delta_n})}{4\lambda_k\lambda_l}\\ &\quad +\frac{(e^{-\lambda_l\Delta_n}-1)^2(1-e^{-2\lambda_k\Delta_n})}{4\lambda_k\lambda_l}+\frac{(e^{-\lambda_l\Delta_n}-1)^2(e^{-\lambda_k\Delta_n}-1)}{2\lambda_k\lambda_l}e^{-\lambda_l\Delta_n}\frac{1-e^{-2\lambda_k\Delta_n}}{1-e^{-(\lambda_k+\lambda_l)\Delta_n}}\bigg)\\
  &=\frac{(1-e^{-\lambda_k\Delta_n})^2(1-e^{-\lambda_l\Delta_n})^2}{2\lambda_k\lambda_l}\frac{3-e^{-(\lambda_k+\lambda_l)\Delta_n}}{1-e^{-(\lambda_k+\lambda_l)\Delta_n}}\\
  &\quad+\frac{(1-e^{-2\lambda_k\Delta_n})(1-e^{-2\lambda_l\Delta_n})}{2\lambda_k\lambda_l}+\frac{(1-e^{-\lambda_k\Delta_n})^2(1-e^{-2\lambda_l\Delta_n})}{{2\lambda_k\lambda_l}}
  +\frac{(1-e^{-\lambda_l\Delta_n})^2(1-e^{-2\lambda_k\Delta_n})}{{2\lambda_k\lambda_l}}\\
  &=\frac{(1-e^{-\lambda_k\Delta_n})^2(1-e^{-\lambda_l\Delta_n})^2}{2\lambda_k\lambda_l}\frac{4-2 e^{-(\lambda_k+\lambda_l)\Delta_n}}{1-e^{-(\lambda_k+\lambda_l)\Delta_n}}\\
  &\quad +\frac{(1-e^{-\lambda_k\Delta_n})(1-e^{-\lambda_l\Delta_n})}{2\lambda_k\lambda_l}\Big(-(1-e^{-\lambda_k\Delta_n})(1-e^{-\lambda_l\Delta_n})+(1+e^{-\lambda_k\Delta_n})(1+e^{-\lambda_l\Delta_n})\\
  &\quad\qquad+(1-e^{-\lambda_k\Delta_n})(1+e^{-\lambda_l\Delta_n})+(1+e^{-\lambda_k\Delta_n})(1-e^{-\lambda_l\Delta_n})\Big)\\
  &=\frac{(1-e^{-\lambda_k\Delta_n})^2(1-e^{-\lambda_l\Delta_n})^2}{\lambda_k\lambda_l}\frac{2-e^{-(\lambda_k+\lambda_l)\Delta_n}}{1-e^{-(\lambda_k+\lambda_l)\Delta_n}}\\
  &\quad+\frac{(1-e^{-\lambda_k\Delta_n})(1-e^{-\lambda_l\Delta_n})}{\lambda_k\lambda_l}\big(2-(1-e^{-\lambda_k\Delta_n})(1-e^{-\lambda_l\Delta_n})\big)\\
  &=\frac{(1-e^{-\lambda_k\Delta_n})^2(1-e^{-\lambda_l\Delta_n})^2}{\lambda_k\lambda_l}\frac{1}{1-e^{-(\lambda_k+\lambda_l)\Delta_n}}+2\frac{(1-e^{-\lambda_k\Delta_n})(1-e^{-\lambda_l\Delta_n})}{\lambda_k\lambda_l}=:\bar D_{k,l}.
\end{align*}
Since $(1-e^{-(\lambda_k+\lambda_l)\Delta_n})^{-1}\le(1-e^{-\lambda_k\Delta_n})^{-1/2}(1-e^{-\lambda_l\Delta_n})^{-1/2}$, the remainder involving $1\wedge\frac{p^{-1}}{1-e^{-(\lambda_k+\lambda_l)\Delta_n}}$ is negligible if $p$ is sufficiently large:
\begin{align*}
  \frac1{\Delta_np^2}\sum_{k<l}\frac{\bar D_{k,l}}{1-e^{-(\lambda_k+\lambda_l)\Delta_n}}
  &\le \frac{3}{\Delta_n^2p^2}\Big(\Delta_n^{1/2}\sum_{k\ge 1}\frac{(1-e^{-\lambda_k\Delta_n})^{1/2}}{\lambda_k}\Big)^2\\
	&=\mathcal O\Big(\frac{\Delta_n^{\eta-1}}{p^2}\Big(\int_0^\infty\frac{(1-e^{-z^2})^{1/2}}{z^{\eta+1}}\d z\Big)^2\Big),
\end{align*}
for any $\eta\in(0,1)$. For smaller $p$ we can always obtain a bound of order $\mathcal O(p^{-1})$. The diagonal terms $k=l$ in \eqref{eq:CovV} are bounded by
\vspace*{-.4cm}
\[D_{k,k}\le\frac{8}{p}\sum_{i,j=1}^p\big(\big(\tilde{\Sigma}_{ij}^{B,k}\big)^2+2\big(\Sigma_{ij}^{BC,k}\big)^2+\big(\Sigma_{ij}^{C,k}\big)^2\big)\,.
\]
With estimates analogous  to \eqref{h1}, \eqref{h2} and \eqref{h3} we obtain $\frac{1}{\Delta_np^2}\sum_{k\ge1}D_{k,k}=\mathcal O(\Delta_n^{1/2}/p)$ which is thus negligible. By symmetry $D_{k,l}=D_{l,k}$, we conclude
\begin{align}
  \cov\big(V_{p,\Delta_n}(y_1),V_{p,\Delta_n}(y_2)\big)
  &=\frac{2\sigma^4e^{(y_1+y_2)\vartheta_1/\vartheta_2}}{\Delta_np}\sum_{k<l}\,e_k(y_1)e_k(y_2)e_l(y_1)e_l(y_2)\bar D_{k,l}\label{eq:CovV2}\\
  &\qquad+\mathcal O\Big(\frac1p\big(\Delta_n^{1/2}+\Delta_n^{\eta-1}p^{-1}\wedge1\big)\Big)\notag\,.
\end{align}
To evaluate the main term, we use $\lambda_k\Delta_n\ge 0$ for all $k$ and thus $\exp(-(\lambda_k+\lambda_l)\Delta_n)\le 1$ implying
\[
\bar D_{k,l}=\sum_{r=0}^{\infty}\frac{(1-e^{-\lambda_k\Delta_n})^2(1-e^{-\lambda_l\Delta_n})^2}{\lambda_k\lambda_l}e^{-r(\lambda_k+\lambda_l)\Delta_n}+2\frac{(1-e^{-\lambda_k\Delta_n})(1-e^{-\lambda_l\Delta_n})}{\lambda_k\lambda_l}\,.
\] 
In particular, each above addend has a product structure with respect to $k,l$. In case that $y_1\ne y_2$, the identity
\vspace*{-.4cm}
\begin{align}
  e_k(y_1)e_k(y_2)&=2\sin(\pi k y_1)\sin(\pi k y_2) e^{-(y_1+y_2)\,\vartheta_1/(2\vartheta_2)}\notag\\
  &=\big(\cos(\pi k(y_1-y_2))-\cos(\pi k (y_1+y_2))\big)e^{-(y_1+y_2)\,\vartheta_1/(2\vartheta_2)}\label{eq:offDiag}
\end{align}
can be used to prove that covariances from \eqref{eq:CovV2} vanish asymptotically at first order. More precisely, applying Lemma~\ref{lem:riemann}$(ii)$ to $f_r(x)=(1-e^{-x})^2e^{-rx}/x$, $x>0$, $f_r(0)=0$, \eqref{eq:CovV2} results, with $z,z'\in\{(y_1-y_2)/2,(y_1+y_2)/2\}$, in terms of the form
\begin{align}
  &\frac{\Delta_n}{p}\sum_{r\ge0}\sum_{k, l}\cos(2\pi kz)\cos(2\pi l z')f_r(\lambda_k\Delta_n)f_r(\lambda_l\Delta_n)\notag\\
  =&\frac{1}{p}\sum_{r\ge0}\Big(\Delta_n^{1/2}\sum_{k\ge 1}f_r(\lambda_k\Delta_n)\cos(2\pi kz)\Big)\Big(\Delta_n^{1/2}\sum_{l\ge 1}f_r(\lambda_l\Delta_n)\cos(2\pi l z')\Big)\notag\\
  =&\mathcal O\Big(\frac{\Delta_n^{1/2}(|y_1-y_2|^{-1}+\delta^{-1})}{p}\sum_{r\ge0}\Delta_n^{1/2}\sum_{k}|f_r(\lambda_k\Delta_n)|\Big)\notag\\
  =&\mathcal O\Big(\frac{\Delta_n^{1/2}(|y_1-y_2|^{-1}+\delta^{-1})}{p}\Big(\Delta_n^{1/2}\sum_{k\ge 1}\frac{1-e^{-\lambda_k\Delta_n}}{\lambda_k\Delta_n}\Big)\Big)
  =\mathcal O\Big(\frac{\Delta_n^{1/2}(|y_1-y_2|^{-1}+\delta^{-1})}{p}\Big)\label{eq:VarRem}.
\end{align}
We are thus left to consider for any $y\in[\delta,1-\delta]$ the variance.
Using again Lemma~\ref{lem:riemann}$(ii)$, one can approximate the terms $e_k^2(y)=2\sin^2(\pi k y)e^{-y\,\vartheta_1/\vartheta_2}\approx e^{-y\,\vartheta_1/\vartheta_2}$ for the variances with $y_1=y_2=y$, up to a term of order $\mathcal O(\delta^{-1}\sqrt{\Delta_n}/p)$, similar to the calculation \eqref{eq:VarRem}. Hence, all variances of the statistics \eqref{rvr} are equal at first order and we find
\begin{align*}
\var(V_{p,\Delta_n}(y))&=\frac{\sigma^4}{\Delta_np}\sum_{k\neq l}\bar D_{k,l}
+\mathcal O\Big(\frac1p\big(\Delta_n^{1/2}\delta^{-1}+\Delta_n^{\eta-1}p^{-1}\wedge1\big)\Big)\\
&=\frac{\sigma^4}{\Delta_np}\bigg(\sum_{r=0}^{\infty}\Big(\sum_{k\ge 1}\frac{(1-e^{-\lambda_k\Delta_n})^2e^{-r\lambda_k\Delta_n}}{\lambda_k}\Big)^2+2\Big(\sum_{k\ge 1}\frac{1-e^{-\lambda_k\Delta_n}}{\lambda_k}\Big)^2\bigg)\\
&\qquad+\mathcal O\Big(\frac1p\big(\Delta_n^{1/2}\delta^{-1}+\Delta_n^{\eta-1}p^{-1}\wedge1\big)\Big),
\end{align*}
using again that the diagonal terms $k=l$ are negligible. To evaluate the leading term, we now apply the integral approximations of the Riemann sums. For $r\in\N$ we have
\begin{align*}
 \lim_{\Delta_n\to0}\Delta_n^{1/2}\sum_{k\ge 1}\frac{(1-e^{-\lambda_k\Delta_n})^2e^{-r\lambda_k\Delta_n}}{\lambda_k\Delta_n}
 &=\frac{1}{\sqrt{\vartheta_2} \pi}\int_0^\infty \frac{(1-e^{-x^2})^2e^{-rx^2}}{x^2}\,\d x\\
 &=\frac{1}{\sqrt{\vartheta_2\pi} }\big(2\sqrt{r+1}-\sqrt{r+2}-\sqrt{r}\big)=:\sqrt{\frac{1}{\pi\vartheta_2}}I(r),\\
 \lim_{\Delta_n\to0}\Delta_n^{1/2}\sum_{k\ge 1}\frac{1-e^{-\lambda_k\Delta_n}}{\lambda_k\Delta_n}
 &=\frac{1}{\sqrt{\vartheta_2} \pi}\int_0^\infty \frac{1-e^{-x^2}}{x^2}\,\d x=\sqrt{\frac{1}{\pi\vartheta_2}}.
\end{align*}
Hence,
\begin{align*}
\var(V_{p,\Delta_n}(y)) &= \frac1p\frac{\sigma^4}{\pi\vartheta_2}\Big(\sum_{r=0}^{\infty} I(r)^2+2\Big)+\mathcal O\Big(\frac{\delta^{-1}\Delta_n^{1/2}}{p}+\frac{1}{p}\wedge\frac{\Delta_n^{\eta-1}}{p^2}\Big)\,.
\end{align*}
Since the series above converges, we obtain \eqref{varpropeq} with the constant
\begin{align}\label{eq:constantVariance}
\Gamma := \frac{1}{\pi}\sum_{r=0}^{\infty} I(r)^2+\frac{2}{\pi}
\end{align}
with $I(r)=2\sqrt{r+1}-\sqrt{r+2}-\sqrt{r}$.
To evaluate $\Gamma$ numerically, we rely on the series approximation
\begin{equation*}
  \sum_{r=0}^{\infty}I(r)^2\approx 0.357487\quad\text{and thus}\quad \Gamma\approx\frac{1}{\pi}\Big( 0.357487+2\Big)\approx 0.75\,.\hfill\qedhere
\end{equation*}
\end{proof}
Before we prove the central limit theorems we provide a result that will be used to verify the generalized mixing-type condition.
\begin{prop}
\label{prop:mixing}For $y\in(0,1)$, and natural numbers $1\le r<r+u\le v\le n$, and
\begin{align}
Q_{1}^r & =\sum_{i=1}^{r}(\Delta_{i}\tilde{X})^{2}(y),\qquad Q_{r+u}^v=\sum_{i=r+u}^{v}(\Delta_{i}\tilde{X})^{2}(y)\,,\label{eq:Qs}
\end{align}
there exists a constant $C,0<C<\infty$, such that for all $t\in\R$
\[
\big|\cov(e^{\iu t(Q_{1}^r-\E[Q_{1}^r])},e^{\iu t(Q_{r+u}^v-\E[Q_{r+u}^v])})\big|\le \frac{C\,t^{2}}{u^{3/4}}\,\var\big(Q_{1}^r\big)^{1/2}\var\big(Q_{r+u}^v\big)^{1/2}.
\]
\end{prop}

\begin{proof}
For brevity we write $\bar{Y}=Y-\E[Y]$ for random variables $Y$. For any decomposition $Q_{r+u}^v=A_{1}+A_{2}$, with some $A_{2}$ which is independent of $Q_{1}^r$,
we have
\begin{align*}
\cov(e^{\iu t\bar{Q}_{1}^r},e^{\iu t\bar{Q}_{r+u}^v}) & =\cov(e^{\iu t\bar{Q}_{1}^r}-1,e^{\iu t\bar{Q}_{r+u}^v})\\
 & =\E\big[(e^{\iu t\bar{Q}_{1}^r}-1)e^{\iu t\bar{A}_{2}}(e^{\iu t\bar{A}_{1}}-1)\big]+\E\big[e^{\iu t\bar{Q}_{1}^r}-1\big]\,\E\big[e^{\iu t\bar{A}_{2}}(1-e^{\iu t\bar{A}_{1}})\big]\,.
\end{align*}

Using that $|e^{\iu x}-1|\le |x|$ and $|e^{\iu x}|=1$, $x\in\R$, we thus get
\begin{align}
\big|\cov(e^{\iu t\bar{Q}_{1}^r},e^{\iu t\bar{Q}_{r+u}^v})\big| & \le\E[|t\bar{Q}_{1}^r|\,|t\bar{A}_{1}|]+\E[|t\bar{Q}_{1}^r|]\E[|t\bar{A}_{1}|]\le2t^{2}\E[(\bar{Q}_{1}^r)^{2}]^{1/2}\E[\bar{A}_{1}^{2}]^{1/2}\,,\label{eq:boundingCharDiff}
\end{align}
by the Cauchy-Schwarz and Jensen's inequality. To obtain such a decomposition, we write for $i>r$
\begin{align}
\Delta_{i}\widetilde X(y) & =\sum_{k\ge1}D_{1}^{k,i}e_{k}(y)+\sum_{k\ge1}D_{2}^{k,i}e_{k}(y),\qquad\text{where}\label{eq:decompMixing}\\
D_{1}^{k,i} & :=\int_{-\infty}^{r\Delta_{n}}e^{-\lambda_{k}((i-1)\Delta_{n}-s)}(e^{-\lambda_{k}\Delta_{n}}-1)\sigma \,\d W_{s}^{k}\,,\nonumber \\
D_{2}^{k,i} & :=\int_{r\Delta_{n}}^{(i-1)\Delta_{n}}e^{-\lambda_{k}((i-1)\Delta_{n}-s)}(e^{-\lambda_{k}\Delta_{n}}-1)\sigma \,\d W_{s}^{k}+\int_{(i-1)\Delta}^{i\Delta_{n}}e^{-\lambda_{k}(i\Delta_{n}-s)}\sigma \d W_{s}^{k}\,.\nonumber 
\end{align}
Then we can decompose
\begin{align*}
Q_{r+u}^v=\sum_{i=r+u}^{v}(\Delta_{i}\widetilde X)^{2}(y)= & \underbrace{\sum_{i=r+u}^{v}\Big(\sum_{k\ge1}D_{1}^{k,i}e_{k}(y)\Big)^{2}+2\sum_{i=r+u}^{v}\Big(\sum_{k\ge1}D_{1}^{k,i}e_{k}(y)\Big)\Big(\sum_{k\ge1}D_{2}^{k,i}e_{k}(y)\Big)}_{=:A_{1}}\\
 & +\underbrace{\sum_{i=r+u}^{v}\Big(\sum_{k\ge1}D_{2}^{k,i}e_{k}(y)\Big)^{2}}_{=:A_{2}}.
\end{align*}
We bound $\E[\bar{A}_{1}^{2}]=\var(A_1)$ by
\begin{align*}
\E[\bar{A}_{1}^{2}]\le\E[A_{1}^{2}] & =\sum_{i,j=r+u}^{v}\Big(\E\Big[\Big(\sum_{k\ge1}D_{1}^{k,i}e_{k}(y)\Big)^{2}\Big(\sum_{k\ge1}D_{1}^{k,j}e_{k}(y)\Big)^{2}\Big]\\
 & \qquad\quad+4\,\E\Big[\Big(\sum_{k\ge1}D_{1}^{k,i}e_{k}(y)\Big)\Big(\sum_{k\ge1}D_{2}^{k,i}e_{k}(y)\Big)\Big(\sum_{k\ge1}D_{1}^{k,j}e_{k}(y)\Big)\Big(\sum_{k\ge1}D_{2}^{k,j}e_{k}(y)\Big)\Big]\Big)\\
 & =:T_{1}+T_{2}\,,
\end{align*}
where the cross term vanishes by independence of $D_{1}^{k,i}$ and
$D_{2}^{k,j}$. Set $p:=v-r-u+1$. Since $D_{1}^{k,i}=e^{-\lambda_{k}(i-r-1)\Delta_{n}}\tilde{B}_{r+1,k}$,
with $\tilde{B}_{r+1,k}$ from \eqref{statIncrement}, we have for $u\ge 2$ that
\begin{align*}
T_{1} & =\sum_{i,j=r+u}^{v}\E\Big[\Big(\sum_{k\ge1}e^{-\lambda_{k}(i-r-1)\Delta_{n}}\tilde{B}_{r+1,k}e_{k}(y)\Big)^{2}\Big(\sum_{k\ge1}e^{-\lambda_{k}(j-r-1)\Delta_{n}}\tilde{B}_{r+1,k}e_{k}(y)\Big)^{2}\Big]\\
 & =\sum_{i,j=r+u}^{v}\sum_{k\neq l}\E[\tilde{B}_{r+1,k}^{2}]\E[\tilde{B}_{r+1,l}^{2}]\Big(e^{-2\lambda_{k}(i-r-1)\Delta_{n}-2\lambda_{l}(j-r-1)\Delta_{n}}+2e^{-(\lambda_{k}+\lambda_{l})(i+j-2r-2)\Delta_{n}}\Big)\\
 & \qquad\times e_{k}^{2}(y)e_{l}^{2}(y)\,+\sum_{i,j=r+u}^{v}\sum_{k\ge 1}\E[\tilde{B}_{r+1,k}^{4}]e^{-2\lambda_{k}(i+j-2r-2)\Delta_{n}}e_{k}^{4}(y)\\
 & =\sigma^{4}\sum_{k\neq l}\frac{(1-e^{-\lambda_{k}\Delta_{n}})^{2}(1-e^{-\lambda_{l}\Delta_{n}})^{2}}{4\lambda_{k}\lambda_{l}}\Big(\Big(\sum_{i=r+u}^{v}e^{-2\lambda_{k}(i-r-1)\Delta_{n}}\Big)\Big(\sum_{j=r+u}^{v}e^{-2\lambda_{l}(j-r-1)\Delta_{n}}\Big)\\
 & \qquad+2\Big(\sum_{i=r+u}^{v}e^{-(\lambda_{k}+\lambda_{l})(i-r-1)\Delta_{n}}\Big)^{2}\Big)e_{k}^{2}(y)e_{l}^{2}(y)\\
 & \qquad+3\sigma^{4}\sum_{k\ge 1}\frac{(1-e^{-\lambda_{k}\Delta_{n}})^{4}}{4\lambda_{k}^{2}}\Big(\sum_{i=r+u}^{v}e^{-2\lambda_{k}(i-r-1)\Delta_{n}}\Big)^{2}e_{k}^{4}(y)\\
 & \le 4\sigma^{4}p\sum_{k\neq l}\frac{(1-e^{-\lambda_{k}\Delta_{n}})^{2}(1-e^{-\lambda_{l}\Delta_{n}})^{2}}{4\lambda_{k}\lambda_{l}}e^{-2(\lambda_{k}+\lambda_{l})(u-1)\Delta_{n}}\Big(\frac{1}{(1-e^{-2\lambda_{k}\Delta_n})}+\frac{2}{(1-e^{-(\lambda_{k}+\lambda_{l})\Delta_n})}\Big)\\
 & \qquad+12\sigma^{4}p\sum_{k\ge 1}\frac{(1-e^{-\lambda_{k}\Delta_{n}})^{4}}{4\lambda_{k}^{2}}\frac{e^{-4\lambda_{k}(u-1)\Delta_{n}}}{(1-e^{-2\lambda_{k}\Delta_n})}\\
 & \le 3\sigma^{4}p\Big(\sum_{k\ge1}e^{-2\lambda_{k}(u-1)\Delta_{n}}\frac{1-e^{-\lambda_{k}\Delta_{n}}}{\lambda_{k}}\Big)\Big(\sum_{k\ge1}e^{-2\lambda_{k}(u-1)\Delta_{n}}\frac{(1-e^{-\lambda_{k}\Delta_{n}})^{2}}{\lambda_{k}}\Big)\,,
\end{align*}
where we use \eqref{eq:covTilde} and $e_k^2(y)\le 2 $ for all $k$. For the last inequality we use \eqref{elin} three times, once with $p=e^{-\lambda_k\Delta_n}$ and $q=e^{-\lambda_l\Delta_n}$, and twice with $p=q=e^{-\lambda_k\Delta_n}$. With the two integral bounds
\begin{align}\label{intbound}\int_0^{\infty}\frac{(1-e^{-z^{2}})^{j}}{z^{2}}e^{-2(u-1)z^2}\,\d z=\mathcal{O}\big(u^{1/2-j}\big)\,, \text{ for } j=1,2\,,\end{align}
we derive that for some constant $C$:
\begin{equation} T_1\le C\sigma^4\frac{p\Delta_n}{(u-1)^{2}} \,.\label{eq:covCF1}\end{equation}
For the second term, the independence of $D_{1}^{k,i}$
and $D_{2}^{l,j}$ and the Cauchy-Schwarz inequality yield
\begin{align*}
T_{2} & =\sum_{i,j=r+u}^{v}\E\Big[\Big(\sum_{k\ge1}D_{1}^{k,i}e_{k}(y)\Big)\Big(\sum_{k\ge1}D_{1}^{k,j}e_{k}(y)\Big)\Big]\E\Big[\Big(\sum_{k\ge1}D_{2}^{k,i}e_{k}(y)\Big)\Big(\sum_{k\ge1}D_{2}^{k,j}e_{k}(y)\Big)\Big]\nonumber \\
 & =\sum_{i,j=r+u}^{v}\Big(\sum_{k\ge1}\E\big[D_{1}^{k,i}D_{1}^{k,j}\big]e_{k}^{2}(y)\Big)\Big(\sum_{k\ge1}\E\big[D_{2}^{k,i}D_{2}^{k,j}\big]e_{k}^{2}(y)\Big)\,.
\end{align*}
For $i\le j$, with \eqref{Scc}, \eqref{Sbc}, and a term similar to \eqref{Sbb}, we have that
\begin{align*}
\E\big[D_{2}^{k,i}D_{2}^{k,j}\big]= & \E\Big[\Big(\int_{r\Delta_{n}}^{(i-1)\Delta_{n}}e^{-\lambda_{k}((i-1)\Delta_{n}-s)}(e^{-\lambda_{k}\Delta_{n}}-1)\sigma \,\d W_{s}^{k}+C_{i,k}\Big)\\
 & \qquad\times\Big(\int_{r\Delta_{n}}^{(j-1)\Delta_{n}}e^{-\lambda_{k}((j-1)\Delta_{n}-s)}(e^{-\lambda_{k}\Delta_{n}}-1)\sigma \,\d W_{s}^{k}+C_{j,k}\Big)\Big]\\
= & \sigma^{2}(1-e^{-\lambda_{k}\Delta_{n}})^{2}e^{-\lambda_{k}(i+j-2)\Delta_{n}}\int_{r\Delta_{n}}^{(i-1)\Delta_{n}}e^{2\lambda_{k}s}\,\d s+\Sigma_{ij}^{BC,k}+\Sigma_{ij}^{C,k}\\
= & \sigma^{2}\frac{(1-e^{-\lambda_{k}\Delta_{n}})^{2}}{2\lambda_{k}}\big(e^{-\lambda_{k}(j-i)\Delta_{n}}-e^{-\lambda_{k}(i+j-2r-2)\Delta_{n}}\big)+\Sigma_{ij}^{BC,k}+\Sigma_{ij}^{C,k}\,,
\end{align*}
where the second addend is zero for $i=j$ and the last one for $i<j$. By \eqref{Sbc}, we especially observe for $i<j$ that 
\begin{align*}
\E\big[D_{2}^{k,i}D_{2}^{k,j}\big] & =\sigma^{2}\frac{1-e^{-\lambda_{k}\Delta_{n}}}{2\lambda_{k}}e^{-\lambda_{k}(j-i)\Delta_{n}}\big((1-e^{-\lambda_{k}\Delta_{n}})(1-e^{-2\lambda_{k}(i-r-1)\Delta_{n}})-(e^{\lambda_{k}\Delta}-e^{-\lambda_{k}\Delta})\big)\\
 & \le\sigma^{2}\frac{1-e^{-\lambda_{k}\Delta_{n}}}{2\lambda_{k}}e^{-\lambda_{k}(j-i)\Delta_{n}}\big(1-e^{\lambda_{k}\Delta}\big)\le0.
\end{align*}
Thereby and with \eqref{eq:covTilde}, we obtain for $u\ge 2$ and some constant $C$ that
\begin{align*}
T_{2} & =\sigma^{4}\sum_{i=r+u}^{v}\Big(\sum_{k\ge 1}e^{-2\lambda_{k}(i-r-1)\Delta_{n}}\frac{(1-e^{-\lambda_{k}\Delta_{n}})^{2}}{2\lambda_{k}}e_{k}^{2}(y)\Big)\\
 & \qquad\times\Big(\sum_{k\ge 1}\Big(\frac{(1-e^{-\lambda_{k}\Delta_{n}})^{2}}{2\lambda_{k}}\big(1-e^{-2\lambda_{k}(i-r-1)\Delta_{n}}\big)+\Sigma_{ii}^{C,k}\Big)e_{k}^{2}(y)\Big)\\
 & \qquad+2\sigma^{4}\sum_{r+u\le i<j\le v}\Big(\sum_{k\ge 1}e^{-\lambda_{k}(i+j-2r-2)\Delta_{n}}\frac{(1-e^{-\lambda_{k}\Delta_{n}})^{2}}{2\lambda_{k}}e_{k}^{2}(y)\Big)\Big(\sum_{k\ge1}\E\big[D_{2}^{k,i}D_{2}^{k,j}\big]e_{k}^2(y)\Big)\\
& \le4\sigma^{4}\sum_{i=r+u}^{v}\Big(\sum_{k\ge 1}e^{-2\lambda_{k}(i-r-1)\Delta_{n}}\frac{(1-e^{-\lambda_{k}\Delta_{n}})^{2}}{2\lambda_{k}}\Big)\Big(\sum_{k\ge 1}\frac{(1-e^{-\lambda_{k}\Delta_{n}})^2+1-e^{-2\lambda_{k}\Delta_{n}}}{2\lambda_{k}}\Big)\\
 & \le 2\sigma^{4}p\Big(\sum_{k\ge 1}e^{-2\lambda_{k}(u-1)\Delta_{n}}\frac{(1-e^{-\lambda_{k}\Delta_{n}})^{2}}{\lambda_{k}}\Big)\Big(\sum_{k\ge 1}\frac{(1-e^{-\lambda_{k}\Delta_{n}})}{\lambda_{k}}\Big)\\
 & \le C\,\frac{\sigma^{4}p\Delta_{n}}{(u-1)^{3/2}}.
\end{align*}
For the last inequality, we use \eqref{intbound} and the same application of Lemma 6.2 as to $I_1$ in \eqref{eq:expT1}. In combination with \eqref{eq:covCF1}, we
conclude for $u\ge 2$ that with some constant $C$:
\[
\E[\bar A_{1}^{2}]\le C \, \frac{\sigma^{4}p\Delta_{n}}{(u-1)^{3/2}}.
\]
We finally note that
\begin{align*}
\hspace*{-.5cm}\E\big[\big(Q_{r+u}^v\big)^{2}\big]=\E\Big[\Big(\sum_{i=r+u}^{v}(\Delta_{i}\tilde{X})^{2}(y)\Big)^{2}\Big]\ge\sum_{i=r+u}^{v}\E\big[(\Delta_{i}\tilde{X})^{4}(y)\big]\ge C''\sigma^{4}p\Delta_n
\end{align*}
for some constant $C''>0$. Combining these estimates with \eqref{eq:boundingCharDiff}, and a simple bound for the case $u=1$, completes the proof.
\end{proof}

Based on the two previous propositions, we now prove the central limit theorems.

\begin{proof}[Proof of Theorem \ref{cltm1}]
For $m=1$, the random variables $\tilde\zeta_{n,i}$ from \eqref{zetatilde} simplify to
\begin{align*}
  \widetilde \zeta_{n,i}=\sqrt{\theta_2\pi}({\Delta_i \widetilde X})^2(y)\exp(y\,\theta_1/\theta_2)\,,
\end{align*}
such that $\hat\sigma_y^2=\tfrac{1}{n\sqrt{\Delta_n}}\sum_{i=1}^n\tilde \zeta_{n,i}$. By Proposition~\ref{propExp} we have $\E[\sum_{i=1}^n\tilde \zeta_{n,i}]=\sqrt{n}\E[\hat \sigma_y^2]=\sqrt{n}(\sigma^2+\KLEINO(1))$. Therefore, we are left to establish a central limit theorem for the triangular array \[Z_{n,i}=\widetilde \zeta_{n,i}-\E[\widetilde \zeta_{n,i}]\,.\] 

According to \citet[Thm. B]{peligradUtev1997}, the weak convergence $\sum_{i=1}^nZ_{n,i}\overset{d}{\rightarrow}\mathcal N(0,v^2)$ with asymptotic variance $v^2=\lim_{n\to\infty}\var(\sum_{i=1}^nZ_{n,i})<\infty$ holds under the following sufficient conditions:
\begin{subequations}\begin{gather}
  \var\Big(\sum_{i=a}^b Z_{n,i}\Big)\le C\sum_{i=a}^b\var(Z_{n,i})\quad\text{for all }1\le a\le b\le n\,,\label{eq:condVariances}\\
  \limsup_{n\rightarrow\infty}\sum_{i=1}^n\E[Z_{n,i}^2]<\infty,\label{sumVariances}\\
  \sum_{i=1}^n\E[Z_{n,i}^2\1_{|Z_{n,i}|>\epsilon}]\rightarrow 0~\text{as}~n\rightarrow\infty\,,\quad\text{for all }\epsilon>0\,,\label{eq:lindeberg}\\
  \cov\Big(e^{\iu t\sum_{i=a}^bZ_{n,i}},e^{\iu t\sum_{i=b+u}^cZ_{n,i}}\Big)\le \rho_t(u)\sum_{i=a}^c\var(Z_{n,i})\,,\label{eq:genMixing}\\
  \hspace*{18em}\text{for all }1\le a\le b<b+u\le c\le n, \text{and }t\in\R\notag\,,
\end{gather}\end{subequations}
with a universal constant $C>0$ and a function $\rho_t(u)\ge0$ satisfying $\sum_{j\ge1}\rho_t(2^j)<\infty$.\\
Since $\sum_{i=1}^n \widetilde \zeta_{n,i}=n\sqrt{\Delta_n}V_{n,\Delta_n}(y)\sqrt{\theta_2\pi}$, by Proposition~\ref{varprop} the asymptotic variance is 
\[\lim_{n\to\infty}\theta_2\pi\Delta_n n^2\var(V_{n,\Delta_n}(y))=\Gamma \pi\sigma^4=v^2\,.\]
Proposition~\ref{varprop} implies that $\var(\sum_{i=a}^b Z_{n,i})=\mathcal O(\Delta_n (b-a+1))$. Since $\Delta_i \widetilde X(y)$ are centered normally distributed random variables with $\var(\Delta_i \widetilde X(y))\propto \Delta_n^{1/2}$ uniformly in $i$ owing to Proposition~\ref{propExp}, there exists a constant $c>0$, such that $\sum_{i=a}^b\var(Z_{n,i})\ge c(b-a+1)\Delta_n$. Thus, we have verified Condition~\eqref{eq:condVariances}. From Proposition~\ref{varprop} we also obtain that $\sum_{i=1}^n\var(Z_{n,i})=\mathcal O(\Delta _n n)$, which grants Condition \eqref{sumVariances}. The generalized mixing-type Condition \eqref{eq:genMixing} is verified by Proposition~\ref{prop:mixing}.

It remains to establish the Lindeberg Condition \eqref{eq:lindeberg} which is by Markov's inequality implied by a Lyapunov condition. Using again that $\Delta_i \widetilde X(y)$ are centered normally distributed with $\var(\Delta_i \widetilde X(y))=\mathcal O(\Delta_n^{1/2})$ uniformly in $i$, we conclude $\E\big[(\Delta_i \widetilde X)^8(y)\big]=\mathcal O(\Delta_n^2)$ and thus $\E\big[\widetilde\zeta_{n,i}^4\big]=\mathcal{O}(\Delta_n^2)$ uniformly in $i$. Therefore,
\begin{equation*}
  \sum_{i=1}^n \E\big[Z_{n,i}^4\big]=\mathcal O(n\Delta_n^2)=\KLEINO(1)~\mbox{as}~n\rightarrow\infty\,.\hfill\qedhere 
\end{equation*}
\end{proof}

For the case of $m>1$ spatial observations we need the following corollary from Proposition~\ref{prop:mixing}:
\begin{cor}\label{cor:mixing}
Under the conditions of Theorem \ref{cltm2}, for natural numbers $1\le r<r+u\le v\le n$,
and
\begin{align*}
\tilde Q_{1}^r & =\sum_{i=1}^{r}\tilde{\zeta}_{n,i},\qquad \tilde Q_{r+u}^v=\sum_{i=r+u}^{v}\tilde{\zeta}_{n,i}\,,
\end{align*}
with $\tilde{\zeta}_{n,i}$ from \eqref{zetatilde}, there exists a constant $C$, $0<C<\infty$, such that for all $t\in\R$ 
\[
\big|\cov(e^{\iu t (\tilde Q_{1}^r-\E[\tilde Q_{1}^r])},e^{\iu t(\tilde Q_{r+u}^v-\E[\tilde Q_{r+u}^v])})\big|\le \frac{C t^2}{u^{3/4}}\var\big(\tilde Q_{1}^r\big)^{1/2}\var\big(\tilde Q_{r+u}^v\big)^{1/2}\,.
\]
\end{cor}

\begin{proof}
With $D_{1}^{k,i}$ and $D_{2}^{k,j}$ from (\ref{eq:decompMixing})
we decompose
\begin{align*}
\tilde Q_{r+u}^v= & \sum_{i=r+u}^{v}\tilde{\zeta}_{n,i}=\frac{\sqrt{\theta_{2}\pi}}{\sqrt{m}}\sum_{j=1}^{m}\Big(A_{1}(y_{j})+A_{2}(y_{j})\Big)e^{y_{j}\theta_{1}/\theta_{2}}\qquad\text{with}\\
A_{1}(y):= & \sum_{i=r+u}^{v}\Big(\sum_{k\ge1}D_{1}^{k,i}e_{k}(y)\Big)^{2}+2\sum_{i=r+u}^{v}\Big(\sum_{k\ge1}D_{1}^{k,i}e_{k}(y)\Big)\Big(\sum_{k\ge1}D_{2}^{k,i}e_{k}(y)\Big),\\
A_{2}(y):= & \sum_{i=r+u}^{v}\Big(\sum_{k\ge1}D_{2}^{k,i}e_{k}(y)\Big)^{2}
\end{align*}
and set $p:=v-r-u+1$.
In view of (\ref{eq:boundingCharDiff}), we thus have to estimate
\begin{align*}
\var\Big(m^{-1/2}\sum_{j=1}^{m}A_{1}(y_{j})\Big) & =\frac{1}{m}\sum_{j=1}^{m}\var\big(A_{1}(y_{j})\big)+\frac{1}{m}\sum_{j\neq j'}\cov\big(A_{1}(y_{j}),A_{1}(y_{j'})\big).
\end{align*}
The covariances can be handled using \eqref{eq:offDiag} together with Lemma \ref{lem:riemann}$(ii)$ as in the proof of Proposition \ref{varprop}. By an analogous computation to the bounds for $T_1$ and $T_2$ as in the proof of Proposition \ref{prop:mixing}, using that $m\cdot\min_{j=2,\dots,m}|y_j-y_{j-1}|$ is bounded from below, we thus obtain for $u\ge 2$ that
\begin{align}\label{spatialcov}
\frac{1}{m}\sum_{j\neq j'}\cov\big(A_{1}(y_{j}),A_{1}(y_{j'})\big)=\mathcal{O}\bigg(\frac{\Delta_n^{3/2}p}{(u-1)^{3/2}}\frac{1}{m}\sum_{j\neq j'}\frac{1}{|y_{j}-y_{j'}|}\bigg)=\mathcal{O}\bigg(\frac{\Delta_{n}^{3/2}p}{(u-1)^{3/2}}m\log m\bigg)\,.
\end{align}
On the other hand, Proposition~\ref{prop:mixing} yields, with a constant $C$, that
\[
\frac{1}{m}\sum_{j=1}^{m}\var\big(A_{1}(y_{j})\big)\le \frac{C \sigma^4\,p\Delta_n}{(u-1)^{3/2}}\,.
\]
Using that $\E\big[\big(\tilde Q_{r+u}^v\big)^{2}\big]\ge C''\sigma^{4} p\Delta_n$ completes the proof.
\end{proof}

\begin{proof}[Proof of Theorem \ref{cltm2}]
Similar to the proof of Theorem~\ref{cltm1} we verify the Conditions \eqref{eq:condVariances} to \eqref{eq:genMixing} for the triangular array $Z_{n,i}=\tilde\zeta_{n,i}-\E[\tilde\zeta_{n,i}]$, with $\tilde\zeta_{n,i}$ from \eqref{zetatilde}, where $m=m_n$ is a sequence satisfying $m_n=\mathcal O(n^{\rho})$ for some $\rho\in(0,1/2)$ according to Assumption \ref{Obs}.\\ From the proof of Proposition~\ref{varprop}, using that $m\cdot\min_{j=2,\dots,m}|y_j-y_{j-1}|$ is bounded from below by Assumption \ref{Obs}, we deduce for any $0\le a<b\le n$, with a similar estimate to \eqref{spatialcov}, that
\begin{align*}
  \var\Big(\sum_{i=a}^b\widetilde \zeta_{n,i}\Big)
  &\hspace*{-0.05cm}=\hspace*{-0.05cm}\Delta_n(b\hspace*{-0.04cm}-\hspace*{-0.04cm}a\hspace*{-0.04cm}+\hspace*{-0.04cm}1)\Gamma\sigma^4\pi\big(1+\hspace*{-0.05cm}\mathcal O(1\wedge \Delta_n^{-1/4}/(b\hspace*{-0.04cm}-\hspace*{-0.04cm}a\hspace*{-0.04cm}+\hspace*{-0.04cm}1))\big)\hspace*{-0.05cm}+\hspace*{-0.05cm}\mathcal O\big((b\hspace*{-0.04cm}-\hspace*{-0.04cm}a\hspace*{-0.04cm}+\hspace*{-0.04cm}1)\Delta_n^{3/2} m\log m\big)\\
  &=\mathcal O\big(\Delta_n(b\hspace*{-0.04cm}-\hspace*{-0.04cm}a\hspace*{-0.04cm}+\hspace*{-0.04cm}1)\big)\,.
\end{align*}
We conclude in particular Conditions \eqref{eq:condVariances} and \eqref{sumVariances}. By the Cauchy-Schwarz inequality we moreover have that
\begin{align}\notag
  \E\big[\widetilde \zeta_{n,i}^4\big]&=\frac{\pi^2\theta_2^2}{m^2}\sum_{j_1,\dots,j_4=1}^me^{(y_{j_1}+\dots+y_{j_4})\theta_1/\theta_2}\E\big[(\Delta_i \widetilde X)^2(y_{j_1})\cdots(\Delta_i \widetilde X)^2(y_{j_4})\big]\\
  &\notag\le \frac{\pi^2\theta_2^2}{m^2}\sum_{j_1,\dots,j_4=1}^me^{(y_{j_1}+\dots+y_{j_4})\theta_1/\theta_2}\E\big[(\Delta_i \widetilde X)^8(y_{j_1})\big]^{1/4}\cdots\E\big[(\Delta_i \widetilde X)^8(y_{j_4})\big]^{1/4}\\
  &\label{eq:4thMoments}\le \pi^2\theta_2^2m^2e^{4\theta_1/\theta_2}\max_{y\in\{y_1,\dots,y_m\}}\E\big[(\Delta_i \widetilde X)^8(y)\big]\,.
\end{align}
Using as in the proof of Theorem \ref{cltm1} that $\E\big[(\Delta_i \widetilde X)^8(y)\big]=\mathcal O(\Delta_n^2)$, we obtain that $\sum_{i=1}^n\E\big[\widetilde \zeta_{n,i}^4\big]=\mathcal{O}(m^2\Delta_n)=\KLEINO(1)$, such that a Lyapunov condition is satisfied implying the Lindeberg Condition \eqref{eq:lindeberg}. Finally, the mixing-type Condition \eqref{eq:genMixing} is verified by Corollary~\ref{cor:mixing}.
\end{proof}

\begin{proof}[Proof of Proposition \ref{propfclt}]
First, we compute the expectation of the quarticity estimator. We have that $\E[\tilde\sigma_{n,m}^4]=\sigma^4+\KLEINO(1)$, since for any $y\in[\delta,1-\delta]$
\begin{align*}
&\sum_{i=1}^n\pi\theta_2\,\exp(2y\,\theta_1/\theta_2)\E\big[(\Delta_i X)^4(y)\big]\\
&\quad=\pi\theta_2\,\exp(2y\,\theta_1/\theta_2)\sum_{i=1}^n\sum_{k,l}3\,\E[(\Delta_i x_k)^2]\E[(\Delta_i x_l)^2]\,e_k^2(y)e_l^2(y)\big(1+\KLEINO(1)\big)\\
&\quad=3\,\pi\theta_2\Delta_n\sum_{i=1}^n\Big(\Delta_n^{-1/2}\sum_{k\ge 1}\big(\Sigma_{ii}^{B,k}+\Sigma_{ii}^{C,k}\big)\Big)^2\big(1+\KLEINO(1)\big)\\
&\quad=3\,n\pi\theta_2\Delta_n\Big(\Delta_n^{-1/2}\sigma^2\sum_{k\ge 1}\lambda_k^{-1}\big(1-\exp(-\lambda_k\Delta_n)\big)\Big)^2\big(1+\KLEINO(1)\big)\\
&\quad=3\,\sigma^4\big(1+\KLEINO(1)\big)\,.
\end{align*}
We use that the fourth-moment terms of $A_{i,k}$ from \eqref{eq:ABC} are negligible under the condition that\\ $\sup_k\lambda_k\E[\langle \xi,e_k\rangle_{\theta}^4]<\infty$.
Recall the statistic $\zeta_{n,i}$ from \eqref{zetaM}.
In order to establish consistency of the quarticity estimator, observe that for $j>i$ with generic constant $C$: 
\begin{align*}
  &\cov\big(\zeta_{n,j}^2,\zeta_{n,i}^2\big)=\frac{\pi^2\theta_2^2}{m^2}\sum_{u_1,\dots,u_4=1}^me^{(y_{u_1}+\dots+y_{u_4})\theta_1/\theta_2}\\
	&\hspace*{3.5cm}\times \cov\big((\Delta_i X)^2(y_{u_1})(\Delta_i X)^2(y_{u_2}),(\Delta_j X)^2(y_{u_3})(\Delta_j X)^2(y_{u_4})\big)\\
	&\quad\le\pi^2\theta_2^2\,e^{4\vartheta_1/\vartheta_2}\,m^2\hspace*{-.15cm}\max_{u_1,u_2,u_3,u_4\in\{1,\dots,m\}}\big|\cov\big((\Delta_i X)^2(y_{u_1})(\Delta_i X)^2(y_{u_2}),(\Delta_j X)^2(y_{u_3})(\Delta_j X)^2(y_{u_4})\big)\big|\\
  &\quad\le Cm^2\sum_{k,l,p,q}\cov\big(\Delta_i x_k\Delta_i x_l \Delta_i x_p\Delta_i x_q,\Delta_j x_k\Delta_j x_l \Delta_j x_p\Delta_j x_q\big)\\
  &\quad= Cm^2\sum_{k,l,p,q}\big(\Sigma_{ij}^{B,k}+\Sigma_{ij}^{BC,k}\big)\big(\Sigma_{ij}^{B,l}+\Sigma_{ij}^{BC,l}\big)\big(\Sigma_{ij}^{B,p}+\Sigma_{ij}^{BC,p}\big)\big(\Sigma_{ij}^{B,q}+\Sigma_{ij}^{BC,q}\big) \big(1+\KLEINO(1)\big)\\
  &\quad= Cm^2\Delta_n^2 \Big(\Delta_n^{-1/2}\sum_{k\ge 1}\big(\Sigma_{ij}^{B,k}+\Sigma_{ij}^{BC,k}\big)\Big)^4\, \big(1+\KLEINO(1)\big)\\
  &\quad= \mathcal O\Big(m^2\Delta_n^2 \sigma^8\Big(\int_0^{\infty}z^{-2}(1-e^{-2z^2})^2 \,e^{-z^2(j-i-1)}\,\d z\Big)^4\Big)\\
  &\quad= \mathcal O\Big(m^2\Delta_n^2 \sigma^8\big(2\sqrt{j-i}-\sqrt{j-i-1}-\sqrt{j-i+1}\big)^4\Big)\\
  &\quad= \mathcal O\big(m^2\Delta_n^2 \sigma^8 (j-i-1)^{-6}\big)\,.
\end{align*}
Here, we use that the terms with fourth and eighth moments of $A_{i,k}$ are negligible under the condition that $\sup_k\lambda_k\E[\langle \xi,e_k\rangle_{\theta}^4]<\infty$ and $\sup_k\lambda_k\E[\langle \xi,e_k\rangle_{\theta}^8]<\infty$. With $\E\big[\zeta_{n,i}^4\big]=\mathcal{O}(m^2\Delta_n^2)$ and the order of $\cov\big(\zeta_{n,j}^2,\zeta_{n,i}^2\big)$ for $i\ne j$, we obtain that 
\begin{align*}
\var\Big(\sum_{i=1}^n\zeta_{n,i}^2\Big)&\le \sum_{i=1}^n\E\big[\zeta_{n,i}^4\big]+\sum_{i\ne j}\cov\big(\zeta_{n,j}^2,\zeta_{n,i}^2\big)\\
&\le \mathcal O\Big(m^2\Delta_n+\sum_{i\ne j} m^2\Delta_n^2 |j-i-1|^{-6}\Big)= \mathcal{O}\Big( m^2\Delta_n\Big(1+\Delta_n\sum_{i\ne j}|j-i-1|^{-6}\Big)\Big)\\
&=\mathcal O(m^2\Delta_n)\,.\end{align*}
Under Assumption \ref{Obs} we thus have $\tilde\sigma_{n,m}^4 \stackrel{\P}{\rightarrow}\sigma^4$. Slutsky's lemma implies the normalized central limit theorem \eqref{fclt}.
\end{proof}

\subsection{Proof of Theorem~\ref{thm:leastSquares}}
\begin{prop}\label{apprprop}
  On Assumptions \ref{Obs}, \ref{cond} and \ref{assvola}, it holds for $m\in\N$ fix and, when $\sqrt{m_n}\Delta_n^{\alpha'-1/2}\to 0$ for some $1/2<\alpha'<\alpha$, for $m=m_n\to\infty$, that
  \begin{align} \label{exptv}\E\Big[\frac{1}{mn\sqrt{\Delta_n}}\sum_{j=1}^m\sum_{i=1}^n(\Delta_i X)^2(y_j)\Big]=e^{-\varkappa y}\frac{IV_0}{\sqrt{\pi}} +\KLEINO\Big(\frac{1}{\sqrt{mn}}\Big)\,. \end{align}
\end{prop}
\begin{proof}
Decompose $\Delta_i x_k$ as in \eqref{eq:ABC} with terms
\begin{align*}
\mathcal{B}_{i,k}&=\int_0^{(i-1)\Delta_n}\sigma_s e^{-\lambda_k((i-1)\Delta_n-s)}(e^{-\lambda_k \Delta_n}-1)\,\d W_s^k\,,\\
\mathcal{C}_{i,k}&=\int_{(i-1)\Delta_n}^{i\Delta_n}\sigma_s e^{-\lambda_k(i\Delta_n-s)}\,\d W_s^k\,,
\end{align*}
and $A_{i,k}$ unchanged. Since the term which hinges on the initial condition does not change under time-dependent volatility, we can neglect it under Assumption \ref{cond} as proved in Lemma \ref{lem1}.\\
Writing $|\sigma^2_{t+s}-\sigma_t^2|=|\sigma_{t+s}-\sigma_t|(\sigma_{t+s}+\sigma_t)$, we see that $(\sigma_t^2)_{t\ge 0}$ satisfies the same Hölder regularity as $(\sigma_t)_{t\ge 0}$. Instead of \eqref{2mc}, we thus obtain on Assumption \ref{assvola}:
\begin{align}\label{2mct}
    \E[\mathcal{C}_{i,k}^2]&=\int_{(i-1)\Delta_n}^{i\Delta_n}\sigma_s^2 e^{-2\lambda_k(i\Delta_n-s)}\,\d s=\frac{1-e^{-2\lambda_k\Delta_n}}{2\lambda_k}\big(\sigma_{(i-1)\Delta_n}^2+\mathcal{O}\big(\Delta_n^{\alpha}\big)\big)\,.
  \end{align}
$\mathcal{B}_{i,k}$ hinges on the whole time interval $[0,(i-1)\Delta_n]$, and the approximation is less obvious. We obtain
\begin{align*}
\E[\mathcal{B}_{i,k}^2]&=\int_0^{(i-1)\Delta_n}\sigma_s^2 e^{-2\lambda_k((i-1)\Delta_n-s)}(e^{-\lambda_k \Delta_n}-1)^2\,\d s\\
    &=\sigma_{(i-1)\Delta_n}^2 (e^{-\lambda_k \Delta_n}-1)^2\frac{1-e^{-2\lambda_k(i-1)\Delta_n}}{2\lambda_k}\\
		&\quad +(e^{-\lambda_k \Delta_n}-1)^2\int_0^{(i-1)\Delta_n}\big(\sigma_s^2-\sigma_{(i-1)\Delta_n}^2\big) e^{-2\lambda_k((i-1)\Delta_n-s)}\,\d s\,.
		\end{align*}
Consider the remainder in the last line, and decompose for some $J,0\le J\le (i-1)$:
\begin{align}\notag
&\Big|\int_0^{(i-1)\Delta_n}\hspace*{-.1cm}\big(\sigma_s^2-\sigma_{(i-1)\Delta_n}^2\big) e^{-2\lambda_k((i-1)\Delta_n-s)}\,\d s\Big| \\
&\notag \le\hspace*{-.05cm}\int_0^{J\,\Delta_n}\hspace*{-.1cm}\big|\sigma_s^2\hspace*{-.05cm}-\hspace*{-.05cm}\sigma_{(i-1)\Delta_n}^2\big| e^{-2\lambda_k((i-1)\Delta_n-s)}\d s
 +\hspace*{-.1cm}\int_{J\,\Delta_n}^{(i-1)\Delta_n}\hspace*{-.1cm}\big|\sigma_s^2\hspace*{-.05cm}-\hspace*{-.05cm}\sigma_{(i-1)\Delta_n}^2\big| e^{-2\lambda_k((i-1)\Delta_n-s)}\d s\\
&\label{tworem}\le \sup_{s\in[0,1]}\sigma_s^2\, \big(2\lambda_k\big)^{-1}e^{-2\lambda_k(i-1-J)\Delta_n}+\big((i-1-J)\Delta_n\big)^{\alpha}\big(2\lambda_k\big)^{-1}\,.
\end{align}
For each $i,k$ this decomposition holds for any $0\le J\le (i-1)$, in particular when setting $J=J(i,k)=\big(\lfloor i-1-\frac{\alpha\log\lambda_k}{2\lambda_k\Delta_n}\rfloor\vee 0\big)$. We then obtain for all $i$ that
\begin{align}\label{best}
	\E[\mathcal{B}_{i,k}^2]=\sigma_{(i-1)\Delta_n}^2 (e^{-\lambda_k \Delta_n}-1)^2\frac{1-e^{-2\lambda_k(i-1)\Delta_n}}{2\lambda_k}+\mathcal O\Big(\frac{(\log \lambda_k)^\alpha(e^{-\lambda_k \Delta_n}-1)^2}{\lambda_k^{1+\alpha}}\Big)\,.
\end{align}
For any $1/2<\alpha'<\alpha$, it holds by \eqref{2mct}, \eqref{best} and Lemma \ref{lem1} that
\begin{align*}&\E\Big[\sum_{i=1}^n(\Delta_i X)^2(y_j)\Big]=\sum_{i=1}^n\E\Big[\sum_{k,l}e_k(y_j)e_l(y_j)\Big(\big(A_{i,k}+ \mathcal{B}_{i,k}+\mathcal{C}_{i,k}\big)\big(A_{i,l}+ \mathcal{B}_{i,l}+\mathcal{C}_{i,l}\big)\Big)\Big]\\
&=\sum_{i=1}^n\sum_{k\ge 1}e_k^2(y_j)\sigma_{(i-1)\Delta_n}^2\Big( (e^{-\lambda_k \Delta_n}-1)^2\frac{1-e^{-2\lambda_k(i-1)\Delta_n}}{2\lambda_k}+\frac{1-e^{-2\lambda_k\Delta_n}}{2\lambda_k}\Big)+\mathcal{O}\big(\Delta_n^{1/2}\big)\\
&\quad +\mathcal{O}\bigg(\sum_{i=1}^n\Delta_n^{\alpha+1/2}\Big(\Delta_n^{1/2}\sum_{k\ge 1}\frac{1-e^{-2\lambda_k\Delta_n}}{2\lambda_k\Delta_n}+\Delta_n^{1/2}\sum_{k\ge 1}\frac{(1-e^{-\lambda_k\Delta_n})^2}{(\lambda_k\Delta_n)^{1+\alpha}}(\log\lambda_k)^{\alpha}\Big)\bigg)\,.\end{align*}
Using the Riemann sum $n^{-1}\sum_{i=1}^n \sigma_{(i-1)\Delta_n}^2= \int_0^1\sigma_s^2\,\d s+\mathcal{O}(\Delta_n^{\alpha})$, and the analogous steps as in the proof of Proposition \ref{propExp} applied to the leading term, we obtain that multiplied with $\Delta_n^{-1/2}$ it converges to $e^{-\varkappa y}IV_0/\sqrt{\pi} $. The remainders are
\begin{align*}
\mathcal{O}\Big(\sum_{i=1}^n\Delta_n^{\alpha'+1/2}\Big)+\mathcal{O}\big(\Delta_n^{1/2}\big)=\mathcal{O}\Big(\Delta_n^{\alpha'-1/2}+\Delta_n^{1/2}\Big)\,
\end{align*}
which tend to zero due to $\alpha'>1/2$. For $m=m_n\to\infty$, $\sqrt{m_n}\Delta_n^{\alpha'-1/2}\to 0$ then implies \eqref{exptv}.\hfill\qedhere
\end{proof}
The next result generalizes Proposition \ref{varprop} to a setting with time-dependent volatility. 
\begin{prop}\label{zjprop}
On Assumptions \ref{Obs}, \ref{cond} and \ref{assvola}, for $m\in\N$ fix and, when $\sqrt{m_n}\Delta_n^{\alpha'-1/2}\to 0$ for some $1/2<\alpha'<\alpha$, for $m=m_n\to\infty$, the random variables $\delta_{n,j}$ in \eqref{zj} satisfy $\E[\delta_{n,j}]=\KLEINO((mn)^{-1/2})$ and
\begin{align*}\cov(\delta_{n,j},\delta_{n,k})=\1_{\{j=k\}}\big(1+\KLEINO(1)\big)\Delta_{n}\frac{\Gamma}{\theta_2} e^{-2\varkappa y_{j}}\int_0^1 \sigma_{s}^{4}\,\d s + \mathcal{O}\big(\Delta_{n}^{3/2}(\delta^{-1}+\1_{\{j\ne k\}}|y_j-y_k|^{-1})\big).\end{align*}
\end{prop}
\begin{proof}
Using Proposition \ref{apprprop}, we obtain in analogy to Proposition \ref{propExp} that
\begin{align*}\E[Z_j]&=\frac{1}{\sqrt{n}}\sum_{i=1}^n\Big(\frac{\sigma^2_{(i-1)\Delta_n}}{\sqrt{\theta_2\pi}}\Delta_n^{1/2}e^{-\varkappa y_j}\Big)+\KLEINO\Big(\frac{1}{\sqrt{mn}}\Big)=\frac{e^{-\varkappa y_j}}{\sqrt{\theta_2\pi}}\int_0^1\sigma_s^2\,\d s+\KLEINO\Big(\frac{1}{\sqrt{mn}}\Big)\end{align*}
which shows that $\E[\delta_{n,j}]=\KLEINO((mn)^{-1/2})$ for all $j\in\{1,\ldots,m\}$.\\
For analyzing the variance-covariance structure of $(\delta_{n,j})$, we reconsider and extend the proof of Proposition \ref{varprop} to time-dependent volatility. 
We use the following notation generalizing \eqref{Sbb}-\eqref{Sbc}:
\begin{align*}{\Sigma}_{ij}^{\mathcal{B},k}:=\cov({\mathcal{B}}_{i,k},{\mathcal{B}}_{j,k})\,,\,{\Sigma}_{ij}^{\mathcal{C},k}\,:=\cov({\mathcal{C}}_{i,k},{\mathcal{C}}_{j,k})~,\,{\Sigma}_{ij}^{\mathcal{BC},k}\,:=\cov({\mathcal{C}}_{i,k},{\mathcal{B}}_{j,k})\,.
\end{align*}
Transferring \eqref{Scc} and \eqref{Sbc} to time-dependent volatility, we obtain
\begin{subequations}
\begin{align}
\label{Scc2}{\Sigma}_{ij}^{\mathcal{C},k}&=\1_{\{i=j\}}\,\big(\sigma_{(i-1)\Delta_n}^2+\mathcal{O}\big(\Delta_n^{\alpha}\big)\big)\,\frac{1-e^{-2\lambda_k\Delta_n}}{2\lambda_k},\\
\label{Sbc2}{\Sigma}_{ij}^{\mathcal{BC},k}&=\1_{\{j>i\}}e^{-\lambda_k(j-i)\Delta_n}\big(e^{\lambda_k\Delta_n}\hspace*{-.05cm}-\hspace*{-.05cm}e^{-\lambda_k\Delta_n}\big)\frac{(e^{-\lambda_k \Delta_n}-1)}{2\lambda_k}\big(\sigma_{(i-1)\Delta_n}^2\hspace*{-.1cm}+\hspace*{-.05cm}\mathcal{O}\big(\Delta_n^{\alpha}\big)\big),
\end{align}
as well as for $j>i$ generalizing \eqref{Sbb}:
\begin{align}\label{Sbb2}
{\Sigma}_{ij}^{\mathcal{B},k}&=(e^{-\lambda_k \Delta_n}-1)^2\int_{0}^{(i-1)\Delta_n} \sigma_s^2 e^{-\lambda_k((i+j-2)\Delta_n-2s)}\,\d s\\
&\notag=(e^{-\lambda_k \Delta_n}-1)^2 e^{-\lambda_k(j-i)\Delta_n}\int_{0}^{(i-1)\Delta_n} \sigma_s^2 e^{-2\lambda_k((i-1)\Delta_n-s)}\,\d s=\E[\mathcal{B}_{i,k}^2]e^{-\lambda_k(j-i)\Delta_n}\,.
\end{align}
\end{subequations}
It holds for any $r\in\{1,\ldots,m\}$, analogously to the illustration with $(D_{k,l})$ in \eqref{eq:CovV}, that
\begin{align*}\var\big(\delta_{n,r}\big)=\frac{2}{n}\sum_{k,l}e_k^2(y_r)e_l^2(y_r)\sum_{i,j=1}^n\big({\Sigma}_{ij}^{\mathcal{B},k}+{\Sigma}_{ij}^{\mathcal{BC},k}+{\Sigma}_{ji}^{\mathcal{BC},k}+{\Sigma}_{ij}^{\mathcal{C},k}\big)\big({\Sigma}_{ij}^{\mathcal{B},l}+{\Sigma}_{ij}^{\mathcal{BC},l}+{\Sigma}_{ji}^{\mathcal{BC},l}+{\Sigma}_{ij}^{\mathcal{C},l}\big)\,.
\end{align*}
We determine generalizations of the terms $(D_{k,l})$ from \eqref{eq:CovV}. The other steps of the proof of Proposition \ref{varprop} can be adopted.
We begin with the generalization of \eqref{h1}. We approximate ${\Sigma}_{ij}^{\mathcal{B},k}$ by
\begin{align}\label{hbappr}
\overline{\Sigma}_{ij}^{\mathcal{B},k}=e^{-\lambda_k|j-i|\Delta_n}\sigma_{((i\wedge j)-1)\Delta_n}^2 (e^{-\lambda_k \Delta_n}-1)^2\frac{1-e^{-2\lambda_k((i\wedge j)-1)\Delta_n}}{2\lambda_k}\,.
\end{align}
Writing
\[{\Sigma}_{ij}^{\mathcal{\mathcal{B}},k}{\Sigma}_{ij}^{\mathcal{B},l}-\overline{\Sigma}_{ij}^{\mathcal{B},k}\overline{\Sigma}_{ij}^{\mathcal{B},l}={\Sigma}_{ij}^{\mathcal{B},k}\big({\Sigma}_{ij}^{\mathcal{B},l}-\overline{\Sigma}_{ij}^{\mathcal{B},l}\big)+\big({\Sigma}_{ij}^{\mathcal{B},k}-\overline{\Sigma}_{ij}^{\mathcal{B},k}\big)\overline{\Sigma}_{ij}^{\mathcal{B},l}\,\]
this approximation error yields two terms of analogous form. With \eqref{Sbb2}, \eqref{Sbb} and \eqref{best}, and since $(\sigma_t)_{t\in[0,1]}$ is uniformly bounded, we conclude that
\begin{align}\notag\sum_{k,l}\sum_{i,j=1}^n {\Sigma}_{ij}^{\mathcal{B},k}\big({\Sigma}_{ij}^{\mathcal{B},l}-\overline{\Sigma}_{ij}^{\mathcal{B},l}\big)&=\mathcal O\Big(\sum_{k,l}\sum_{i,j=1}^ne^{-\lambda_k|i-j|\Delta_n}\frac{(1-e^{-\lambda_k\Delta_n})^2}{\lambda_k}\frac{(1-e^{-\lambda_l\Delta_n})^2}{(\lambda_l)^{1+\alpha}}(\log\lambda_l)^{\alpha}\Big)\\
&\notag=\mathcal O\Big(\sum_{k,l} n \frac{1+e^{-\lambda_k\Delta_n}}{1-e^{-\lambda_k\Delta_n}}\frac{(1-e^{-\lambda_k\Delta_n})^2}{\lambda_k}\frac{(1-e^{-\lambda_l\Delta_n})^2}{(\lambda_l)^{1+\alpha}}(\log\lambda_l)^{\alpha}\Big)\\
&\notag=\mathcal O\Big( n \sum_{k\ge 1} \frac{1-e^{-2\lambda_k\Delta_n}}{\lambda_k}\sum_{l\ge 1}\frac{(1-e^{-\lambda_l\Delta_n})^2}{(\lambda_l)^{1+\alpha}}(\log\lambda_l)^{\alpha}\Big)\\
&\label{err1}=\mathcal O \big(n\Delta_n^{1+\alpha'}\big)=\mathcal O \big(\Delta_n^{\alpha'}\big)\,,
\end{align}
for any $\alpha'<\alpha$. We obtain that
\(\sum_{k,l}\sum_{i,j=1}^n {\Sigma}_{ij}^{\mathcal{B},k}{\Sigma}_{ij}^{\mathcal{B},l} = \sum_{k,l}\sum_{i,j=1}^n\overline{\Sigma}_{ij}^{\mathcal{B},k}\overline{\Sigma}_{ij}^{\mathcal{B},l}+\KLEINO(1)\,.\)
Decomposing the double sum, we obtain that 
\begin{align*}\sum_{i,j=1}^n\overline{\Sigma}_{ij}^{\mathcal{B},k}\overline{\Sigma}_{ij}^{\mathcal{B},l}&=
\sum_{i=1}^n\overline{\Sigma}_{ii}^{\mathcal{B},k}\overline{\Sigma}_{ii}^{\mathcal{B},l} +\frac{2}{n}\sum_{i=1}^n\sigma_{(i-1)\Delta_n}^4\hspace*{-.05cm}\sum_{j=i+1}^{n}e^{-(\lambda_k+\lambda_l)(j-i)\Delta_n}\frac{(1-e^{-\lambda_k\Delta_n})^2(1-e^{-\lambda_l\Delta_n})^2}{4\Delta_n\lambda_k\lambda_l}\\
&=\frac{1}{n}\sum_{i=1}^n\sigma_{(i-1)\Delta_n}^4\Big(1+\frac{e^{-(\lambda_k+\lambda_l)\Delta_n}}{1-e^{-(\lambda_k+\lambda_l)\Delta_n}}\Big)\frac{(1-e^{-\lambda_k\Delta_n})^2(1-e^{-\lambda_l\Delta_n})^2}{4\Delta_n\lambda_k\lambda_l}+R_{k,l}\,,\end{align*}
with remainders $R_{k,l}$.
Here we used, analogously to the computations below \eqref{h3}, the geometric series 
\begin{align}\label{geohelp}\sum_{j=i+1}^ne^{-(\lambda_k+\lambda_l)(j-i)\Delta_n}=\frac{e^{-(\lambda_k+\lambda_l)\Delta_n}-e^{-(\lambda_k+\lambda_l)(n-i+1)\Delta_n}}{1-e^{-(\lambda_k+\lambda_l)\Delta_n}}\end{align}
where the resulting $i$-dependent terms induce $R_{k,l}$, which satisfy $\sum_{k,l} R_{k,l}\to 0$. Since in these two terms only addends with $i=j$ are non-zero, \eqref{h2} and \eqref{h5} directly generalize to analogous expressions just replacing $\sigma^4$ by $n^{-1}\sum_{i=1}^n\sigma_{(i-1)\Delta_n}^4$ $\rightarrow \int_0^1\sigma_s^4\,\d s$. The generalization of \eqref{h3} is analogous to the one of \eqref{h1} and we omit it. We address the generalization of \eqref{h4}. We decompose
\[\sum_{i,j=1}^n{\Sigma}_{ij}^{\mathcal{B},k}{\Sigma}_{ij}^{\mathcal{BC},l}=\sum_{i=1}^n\sum_{j=i+1}^n\overline{\Sigma}_{ij}^{\mathcal{B},k}{\Sigma}_{ij}^{\mathcal{BC},l}+ \sum_{i=1}^n\sum_{j=i+1}^n\big({\Sigma}_{ij}^{\mathcal{B},k}-\overline{\Sigma}_{ij}^{\mathcal{B},k}\big){\Sigma}_{ij}^{\mathcal{BC},l}\,.\]
Similar to \eqref{err1}, for any $\alpha'<\alpha$:
\(\sum_{k,l}\sum_{i=1}^n\sum_{j=i+1}^n\big({\Sigma}_{ij}^{\mathcal{B},k}-\overline{\Sigma}_{ij}^{\mathcal{B},k}\big){\Sigma}_{ij}^{\mathcal{BC},l}=\mathcal O \big(\Delta_n^{\alpha'}\big).\)
By \eqref{Sbc2} and \eqref{hbappr}, we obtain that
\begin{align*}\sum_{i,j=1}^n\overline{\Sigma}_{ij}^{\mathcal{B},k}{\Sigma}_{ij}^{\mathcal{BC},l}&=\sum_{i=1}^n\sum_{j=i+1}^n \big(\sigma^4_{(i-1)\Delta_n}+\mathcal{O}(\Delta_n^{\alpha})\big)e^{-(\lambda_k+\lambda_l)(j-i)\Delta_n}(1-e^{-\lambda_k\Delta_n})^2\frac{1-e^{-2\lambda_k(i-1)\Delta_n}}{2\lambda_k}\\
&\hspace*{7cm}\times (e^{\lambda_l\Delta_n}-e^{-\lambda_l\Delta_n})\frac{e^{-\lambda_l\Delta_n}-1}{2\lambda_l}\,,\end{align*}
where we apply \eqref{geohelp} again. Then, the only dependence on $i$ remaining in the leading term is by $\sigma^4_{(i-1)\Delta_n}$, while all other terms are exactly the same as in the proof of Proposition \ref{varprop}. This directly yields the generalization. The term 
\begin{align*}\sum_{i,j=1}^n\overline {\Sigma}_{ij}^{\mathcal{B},k}{\Sigma}_{ji}^{\mathcal{BC},l}=\sum_{i=2}^n\sum_{j=1}^{i-1}\overline{\Sigma}_{jj}^{\mathcal{B},k}e^{-(\lambda_k+\lambda_l)(i-j)\Delta_n}(e^{\lambda_l\Delta_n}-e^{-\lambda_l\Delta_n})\frac{e^{-\lambda_l\Delta_n}-1}{2\lambda_l}\big(\sigma^2_{(j-1)\Delta_n}+\mathcal{O}(\Delta_n^{\alpha})\big)\end{align*}
is more involved. We used \eqref{Sbc2} and an illustration analogous to \eqref{Sbb2} for the equality. Here, we have to show that the remainder of approximating the volatility factors by $\sigma^4_{(i-1)\Delta_n}$ is asymptotically negligible. Thereto, we decompose 
\begin{align*}&\sum_{i=2}^n\sum_{j=1}^{i-1}\big(\sigma^4_{(j-1)\Delta_n}-\sigma^4_{(i-1)\Delta_n}\big)e^{-(\lambda_k+\lambda_l)(i-j)\Delta_n}\\
&\quad \le C\,\sum_{i=2}^n\Big(\sum_{j=J+1}^{i-1}e^{-(\lambda_k+\lambda_l)(i-j)\Delta_n}\big((i-j)\Delta_n\big)^{\alpha}+\sum_{j=1}^J e^{-(\lambda_k+\lambda_l)(i-j)\Delta_n}\Big)\\
&\quad \le C\,\sum_{i=2}^n\Big(\sum_{j=J+1}^{i-1}e^{-(\lambda_k+\lambda_l)(i-j)\Delta_n}\big((i-1-J)\Delta_n\big)^{\alpha}+e^{-(\lambda_k+\lambda_l)(i-J)\Delta_n}\frac{1-e^{-(\lambda_k+\lambda_l)J\Delta_n}}{1-e^{-(\lambda_k+\lambda_l)\Delta_n}}\Big)\,,\end{align*}
with some $C<\infty$ and $J\in\{1,\dots,i-1\}$ (for $J=i-1$ set the left inner sum equal to zero). We use this decomposition with $J=J(i;b)=\big(\lfloor i-1-i^b\rfloor	\vee 0\big)$, $b\in(0,1)$. Since we omit the factors $n^{-1}$ compared to \eqref{h1}-\eqref{hl}, we only require that remainders tend to zero. The first term of the decomposition then induces a remainder of order
\begin{align*}\mathcal{O}\Big(\sum_{k,l}\sum_{i=2}^n\sum_{j=1}^{i-1}\frac{(e^{-\lambda_k\Delta_n}-1)^2}{2\lambda_k}e^{-(\lambda_k+\lambda_l)(i-j)\Delta_n}(e^{\lambda_l\Delta_n}-e^{-\lambda_l\Delta_n})\frac{e^{-\lambda_l\Delta_n}-1}{2\lambda_l}\big((i-1-J)\Delta_n\big)^{\alpha}\Big)\\
=\mathcal{O}\Big(\sum_{k,l}\frac{(e^{-\lambda_k\Delta_n}-1)^2}{4\lambda_k\lambda_l\Delta_n}(e^{-\lambda_l\Delta_n}-1)\frac{1-e^{-2\lambda_l\Delta_n}}{1-e^{-(\lambda_k+\lambda_l)\Delta_n}}e^{-\lambda_k\Delta_n}\big(n^b\Delta_n\big)^{\alpha}\Big)\,,
\end{align*}
where we use a geometric series again. Following the same steps as in the proof of Proposition \ref{varprop}, we obtain that this term is of order $n^{(b-1)\alpha}\to 0$. 
The remainder induced by the second term of the decomposition is bounded by
\begin{align*}&\mathcal{O}\Big(\Delta_n\sum_{i=2}^n\sum_{r\ge 0}\Big(\sqrt{\Delta_n}\sum_{k\ge 1}\frac{(1-e^{-\lambda_k\Delta_n})^2}{2\lambda_k\Delta_n}e^{-\lambda_k(i^b+r)\Delta_n}\Big)\Big(\sqrt{\Delta_n}\sum_{l\ge 1}\frac{(1-e^{-2\lambda_l\Delta_n})^2}{2\lambda_l\Delta_n}e^{-\lambda_l(i^b-1+r)\Delta_n}\Big)\Big)\\
&=\mathcal{O}\Big(\Delta_n\sum_{i=1}^n i^{-3b}\Big)=\mathcal{O}\Big( n^{-(3b\wedge 1)}\Big)\,,\end{align*}
where we use the integral bound \eqref{intbound} and an analogous one. Setting $b=\alpha/(\alpha+3)$, shows that the total remainder is of order $n^{-3\alpha/(\alpha+3)}\to 0$. This completes the discussion of the generalization of \eqref{h4}. The analogues of the terms \eqref{hl} are again zero. Combining these generalization steps with the same integral approximations as in the proof of Proposition \ref{varprop}, yields
\begin{align*}\var\big(\delta_{n,r}\big)&=\frac{4}{n}\sum_{k<l}e^{-2\varkappa y_r}(1+\KLEINO(1))\hspace*{-.1cm}\sum_{i,j=1}^n\hspace*{-.1cm}\big({\Sigma}_{ij}^{\mathcal{B},k}\hspace*{-.02cm}+\hspace*{-.02cm}{\Sigma}_{ij}^{\mathcal{BC},k}\hspace*{-.02cm}+\hspace*{-.02cm}{\Sigma}_{ji}^{\mathcal{BC},k}\hspace*{-.02cm}+\hspace*{-.02cm}{\Sigma}_{ij}^{\mathcal{C},k}\big)\big({\Sigma}_{ij}^{\mathcal{B},l}\hspace*{-.02cm}+\hspace*{-.02cm}{\Sigma}_{ij}^{\mathcal{BC},l}\hspace*{-.02cm}+\hspace*{-.02cm}{\Sigma}_{ji}^{\mathcal{BC},l}\hspace*{-.02cm}+\hspace*{-.02cm}{\Sigma}_{ij}^{\mathcal{C},l}\big)\\
&=\big(1+\KLEINO(1)\big)\Delta_{n}\frac{\Gamma}{\theta_2} e^{-2\varkappa y_{r}}\int_0^1 \sigma_{s}^{4}\,\d s+\mathcal{O}\big(\delta^{-1}\Delta_n^{3/2}\big).
\end{align*}
The proof that covariances for different spatial observations are negligible directly extends to time-dependent volatility and we adopt the order from Proposition \ref{varprop}. 
\end{proof}
Based on the two previous propositions, we now prove Theorem~\ref{thm:leastSquares}. The M-estimator is associated to the contrast function, with $\eta:=(IV_{0},\varkappa)$,
\begin{equation*}
K(\eta,\eta'):=\int_\delta^{1-\delta}\big(f_{\eta}(y)-f_{\eta'}(y)\big)^{2}\,\d y\,.
\end{equation*}

\emph{Step 1:} We prove for $K_{n}(\eta'):=\frac{1}{m}\sum_{j=1}^{m}(Z_{j}-f_{\eta'}(y_{j}))^{2}$ that
\[
K_{n}(\eta')-K(\eta,\eta')\overset{\P_{\eta}}{\longrightarrow}0.
\]
To verify this convergence in probability, we decompose
\begin{align}\label{eq:contrast}
K_{n}(\eta') & =\frac{1}{m}\sum_{j=1}^{m}\big(f_{\eta}(y_{j})-f_{\eta'}(y_{j})\big)^{2}-\frac{2}{m}\sum_{j=1}^{m}\delta_{n,j}\big(f_{\eta}(y_{j})-f_{\eta'}(y_{j})\big)+\frac{1}{m}\sum_{j=1}^{m}\delta_{n,j}^{2}.
\end{align}
The first term converges to $\int_\delta^{1-\delta}\big(f_{\eta}(y)-f_{\eta'}(y)\big)^{2}\,\d y$ as $m\to\infty$. For the second term we have with $g:=f_{\eta}-f_{\eta'}$, being uniformly bounded on $\Xi\times[0,1]$, by Proposition~\ref{zjprop}
\begin{align*}
&\P_\eta\Big(\Big|\frac{2}{m}\sum_{j=1}^{m}\delta_{n,j}g(y_{j})\Big|>\varepsilon\Big)
\le  \varepsilon^{-2}\frac{4}{m^{2}}\sum_{j,k=1}^{m}|\E[\delta_{n,j}\delta_{n,k}]|\cdot|g(y_{j})g(y_{k})|\\
&\quad \le  \varepsilon^{-2}\frac{4}{m^{2}}\sum_{j,k=1}^{m}\big(|\cov(\delta_{n,j}\,,\,\delta_{n,k})|+|\E[\delta_{n,j}]\E[\delta_{n,k}]|\big)\cdot|g(y_{j})g(y_{k})|\\
&\quad =  \mathcal{O} \Big(\frac{4}{m^{2}}\sum_{j,k=1}^{m}\Big(\1_{\{j=k\}}\Delta_{n}\frac{\Gamma}{\theta_2} e^{-2\varkappa y_{j}}\int_0^1\sigma_s^4\,\d s+\Delta_{n}^{3/2}\big(\delta^{-1}+\1_{\{j\ne k\}}|y_j-y_k|^{-1}\big)+(mn)^{-1}\Big)\Big)\\
&\quad = \mathcal{O}\big(\Delta_{n}/m\big).
\end{align*}
For the third term in \eqref{eq:contrast} we obtain with $\E[\delta_{n,j}^2]=\Delta_n\Gamma e^{-2\varkappa y_j}\int_0^1\sigma_s^4\,\d s/\theta_2(1+\KLEINO(1))$:
\begin{align*}
\P_\eta\Big(\frac{1}{m}\sum_{j=1}^{m}\delta_{n,j}^{2}>\varepsilon\Big) 
& \le m\sup_{j}\P_\eta\Big(\delta_{n,j}^{2}>\varepsilon/2\Big)
  \le\frac{2m}{\varepsilon}\sup_{j}\E\big[\delta_{n,j}^{2}\big]
 =\mathcal{O}(m\Delta_{n})\,.
\end{align*}

\emph{Step 2:} We prove consistency of $\hat{\eta}=(\widehat{IV}_{\negthinspace 0},\hat{\varkappa})$. Since $K(\eta,\cdot)$ and $K_{n}$ are continuous, consistency follows from continuous mapping if the stochastic
convergence of $K_{n}$ is also uniform. Now, tightness follows from the equicontinuity:
\begin{align*}
&\lim_{\delta\to0}\limsup_{n\to\infty}\P_\eta\Big(\sup_{|\eta'-\eta|<\delta}|K_{n}(\eta)-K_{n}(\eta')|\ge\varepsilon\Big) \\
 &\quad\le\lim_{\delta\to0}\limsup_{n\to\infty}\frac{1}{\varepsilon}\E\Big[\sup_{|\eta'-\eta|<\delta}|K_{n}(\eta)-K_{n}(\eta')|\Big]\\
 &\quad =\lim_{\delta\to0}\limsup_{n\to\infty}\frac{1}{\varepsilon}\E\Big[\sup_{|\eta'-\eta|<\delta}\Big|\frac{1}{m}\sum_{j=1}^{m}\Big(2Z_{j}(f_{\eta}(y_{j})-f_{\eta'}(y_{j}))+f_{\eta'}^{2}(y_{j})-f_{\eta}^{2}(y_{j})\Big)\Big|\Big].
\end{align*}
The last line converges to zero by dominated convergence and uniform continuity of $f_{\eta}$ in $\eta$.

\emph{Step 3:} To prove the central limit theorem, we follow the classical theory elaborated in \cite[Chapter 3.3.4]{dacunha}. We calculate the first
and second derivative of $K_{n}(\eta)$. The gradient and the Hessian
matrix are given by
\begin{align*}
\dot{K}_{n}(\eta) & =-\frac{2}{m}\sum_{j=1}^{m}(Z_{j}-f_{\eta}(y_{j}))\frac{e^{-\varkappa y_{j}}}{\sqrt{\pi}}\begin{pmatrix}1\\
-IV_0\,y_{j}
\end{pmatrix},\\
\ddot{K}_{n}(\eta) & =\frac{2}{m}\sum_{j=1}^{m}\big(Z_{j}-f_{\eta}(y_{j})\big)\frac{e^{-\varkappa y_{j}}}{\sqrt{\pi}}\begin{pmatrix}0 & y_{j}\\
y_{j} & -IV_0\,y_{j}^{2}
\end{pmatrix}+\frac{2}{m}\sum_{j=1}^{m}\frac{e^{-2\varkappa y_{j}}}{\pi}\begin{pmatrix}1 & -IV_0\,y_{j}\\
-IV_0\,y_{j} & IV_0^2 y_{j}^{2}
\end{pmatrix},
\end{align*}
respectively. The mean value theorem yields on an event with probability converging to one that
\[
  0=\dot K_n(\hat\eta)=\dot K_n(\eta)+\ddot K_n(\eta^*)(\hat\eta-\eta)
\]
for some $\eta^*$ between $\eta$ and $\hat\eta$ and thus
$
  \hat\eta-\eta=- \ddot K_n(\eta^*)^{-1}\dot K_n(\eta)
$, on the event where $\hat\eta$ is close to $\eta$ and where $\ddot K_n(\eta^*)$ is invertible.

To show a central limit theorem for $\dot K_n(\eta)$, we set
\[
  \chi_{n,i}:=-\frac{2}{\sqrt m}\sum_{j=1}^m\big((\Delta_i \widetilde X)^2(y_j)-\sqrt{\Delta_n} f_\eta(y_j)\big)\frac{e^{-\varkappa y_j}}{\sqrt \pi}\begin{pmatrix} 1\\ -IV_0\,y_{j}\end{pmatrix},
\]
such that $\sqrt{nm}\dot K_n(\eta)=\sum_{i=1}^n\chi_{n,i}$. Using Proposition~\ref{zjprop}, the variance-covariance matrix of $\sum_{i=1}^n\chi_{n,i}$ can be calculated as follows:
\begin{align*}
  &\Cov\Big(\sum_{i=1}^n\chi_{n,i}\Big)\hspace*{-.025cm}=\hspace*{-.025cm}\frac{4}{m}\hspace*{-.025cm}\sum_{j,k=1}^m\frac{e^{-\varkappa(y_j+y_k)}}{\pi}\cov\Big(\sum_{i=1}^n(\Delta_iX)^2(y_j),\sum_{i=1}^n(\Delta_iX)^2(y_k)\Big)\hspace*{-.025cm}\begin{pmatrix} 1 & -IV_0 \,y_k\\ -IV_0\,y_{j} & IV_0^2 y_jy_k\end{pmatrix}\\
  &\,=\frac{4n^2\Delta_n}m\sum_{j,k=1}^m\frac{e^{-2\varkappa(y_j+y_k)}}{\pi}\cov\big(V_n(y_j),V_n(y_k)\big)\begin{pmatrix} 1 & -IV_0 \,y_k\\ -IV_0\,y_{j} & IV_0^2 y_jy_k\end{pmatrix}\\
  &\,=\frac{4}m\sum_{j,k=1}^m\frac{e^{-2\varkappa(y_j+y_k)}}{\pi}\Big(\1_{\{j=k\}}\frac{\Gamma}{\theta_2}\int_0^1\sigma_s^4\,\d s+\mathcal O\Big(\Delta_n^{1/2}\frac{\1_{\{j\neq k\}}}{|y_j-y_k|\wedge\delta}\Big)\Big)\hspace*{-4pt}\begin{pmatrix} 1 & -IV_0\,y_k\\ -IV_0\,y_{j} & IV_0^2y_jy_k\end{pmatrix}\\
  &\,=\frac{4\,\Gamma\int_0^1\sigma_s^4\,\d s}{\theta_2\pi}\begin{pmatrix} \frac{1}{m}\sum_je^{-4\varkappa y_j} & -IV_0\frac{1}{m}\sum_jy_je^{-4\varkappa y_j}\\ -IV_0\frac{1}{m}\sum_jy_je^{-4\varkappa y_j} & IV_0^2\frac{1}{m}\sum_jy_j^2e^{-4\varkappa y_j}\end{pmatrix}\big(1+\mathcal O(\Delta_n^{1/2}(\log m)m)\big).
\end{align*}
If $m\to\infty$, then the matrix in the last display converges on Assumption~\ref{Obs}, when $y_1=\delta$ and $y_m=1-\delta$, to $U(\eta)$ from \eqref{eq:U}. In particular, Jensen's inequality shows that the covariance matrix is positive definite for sufficiently large $m$.
We will apply again Utev's central limit theorem together with the Cramér-Wold device to deduce the two-dimensional central limit theorem. Analogously to $\tilde\zeta_{n,i}$ from \eqref{zetatilde} in the proof of Theorem~\ref{cltm2}, the triangular array $(\chi_{n,i})$ and any linear combination of its components satisfy the mixing-type Condition \eqref{eq:genMixing} based on a direct modification of Corollary \ref{cor:mixing}. By the above considerations for the variance, $(\chi_{n,i})$ and any linear combination of its components satisfy Conditions~\eqref{eq:condVariances} and \eqref{sumVariances}. The Lyapunov condition and hence \eqref{eq:lindeberg} follows from $$\E[|\chi_{n,i}|^4]=\mathcal O\Big(m^2\max_j\E[(\Delta_i\tilde X(y_j))^8]\Big)=\mathcal O(m^2\Delta_n^2)=\KLEINO(1)\,,$$ similar to \eqref{eq:4thMoments}.
Therefore, we have under $\P_\eta$ for $n\to\infty$ and $m=m_n\to\infty$, when $\sqrt{m_n}\Delta_n^{\alpha'-1/2}\to 0$ for some $\alpha'<\alpha$, that 
\[
  \sqrt{mn}K_n(\eta)\overset{d}{\rightarrow}\mathcal N\Big(0,\frac{4\,\Gamma\int_0^1 \sigma_s^4\,\d s}{\theta_2\pi} U(\eta)\Big).
\]
As the final step, we will verify $\P_\eta$-convergence in probability of $\ddot K_n(\eta_n)$ to a deterministic, invertible matrix for any sequence $\eta_n\overset{\P_\eta}{\to}\eta$. Since for any such sequence, $Z_j-f_{\eta_n}(y_j)=f_\eta(y_j)-f_{\eta_n}(y_j)+\delta_{n,j}\overset{\P_\eta}{\to}0$ holds under $\P_\eta$ by continuity of $f_\eta$ in the parameter and owing to $\delta_{n,j}\overset{\P_\eta}{\to}0$, we indeed have
\begin{align*}
  \ddot K_n(\eta_n)&=\frac{2}{m}\sum_{j=1}^{m}\frac{e^{-2\varkappa y_{j}}}{\pi}\begin{pmatrix}1 & -IV_0\,y_{j}\\
 -IV_0\,y_{j} & IV_0^{2}y_{j}^{2}\end{pmatrix}+\KLEINO_{\P_{\eta}}(1)\overset{\P_\eta}{\to}\frac{2}{\pi}V(\eta)\,,
\end{align*}
with $V(\eta)$ from \eqref{eq:V} using that $y_1=\delta$ and $y_m=1-\delta$. Jensen's inequality verifies that $V(\eta)$ is strictly positive definite. We obtain the limit distribution in \eqref{cltlq}.
\hfill\qed

\subsection{Proof of Corollary~\ref{gclt}}
For the statistics analogous to \eqref{zetaM}, the bias is asymptotically negligible by Proposition \ref{apprprop}. For the centered versions, we obtain Conditions \eqref{eq:condVariances} and \eqref{sumVariances} with the asymptotic variance in \eqref{gclteq} from Proposition \ref{zjprop}. Since $(\sigma_t^2)_{0\le t\le 1}$ is uniformly bounded, the proofs of the Lyapunov condition and the mixing-type conditions from Proposition \ref{prop:mixing} and Corollary \ref{cor:mixing} readily extend to the time-varying setting. With \citet[Thm. B]{peligradUtev1997}, we thus conclude \eqref{gcltm1eq} and \eqref{gclteq}. A simple generalization shows consistency of the integrated quarticity estimator and thus \eqref{gcltf}.\hfill\qed
\addcontentsline{toc}{section}{Acknowledgements}
\section*{Acknowledgements}
The authors are grateful to two anonymous referees whose valuable comments helped improving the paper. We also thank Jan Kallsen and Mathias Vetter for helpful comments. Mathias Trabs gratefully acknowledges financial support by the DFG research fellowship TR 1349/1-1. 

\addcontentsline{toc}{section}{References}
\bibliographystyle{apalike}
\bibliography{library}

\end{document}